\documentclass[11pt]{amsart}
\usepackage[active]{srcltx}
\usepackage{amsmath}
\usepackage{amsthm}
\usepackage{amssymb}
\usepackage{fullpage}
\usepackage[all,cmtip]{xy}
\usepackage{amsfonts}
\usepackage{verbatim}
\usepackage{amscd,latexsym}
\usepackage{tikz}
\usepackage{caption}
\usepackage{array}
\usepackage{multirow}
\usepackage{mathrsfs}
\usepackage{hyperref}
\usepackage{tabularx}
\usepackage{pgfplots}
\usepackage{float}
\usepackage{bm}
\usepackage{graphicx}
\usepackage{makecell}
\usepackage{longtable}
\usepackage{mathabx}
\usepackage{tabstackengine}
\usepackage{mathtools}
\usepackage{array}
\newsavebox{\mybox}
\usepackage{comment}
\usepackage[utf8]{inputenc}
\usepackage{nicematrix}
\usepackage{kbordermatrix}
\usepackage[T1]{fontenc}
\usepackage{enumitem}
\usepackage{booktabs}
\usepackage{bigints}
\usepackage{chngcntr}
\usepackage{apptools}
\usepackage{soul}
\usepackage{ragged2e} 

\hypersetup{colorlinks=true, linkcolor=blue, filecolor=magenta, urlcolor=cyan, citecolor=purple}
\newtheorem{theorem}{Theorem}
\numberwithin{theorem}{section}
\newtheorem{corollary}[theorem]{Corollary}
\newtheorem{proposition}[theorem]{Proposition}
\newtheorem{lemma}[theorem]{Lemma}
\newtheorem{conjecture}[theorem]{Conjecture}
\theoremstyle{definition}
\numberwithin{theorem}{section}
\newtheorem{remark}[theorem]{Remark}
\newtheorem{definition}[theorem]{Definition}

\newcommand{\rec}{{\rm rec}}

\newcommand{\N}{\mathbb{N}}

\newcommand{\C}{\mathbb{C}}
\newcommand{\Z}{\mathbb{Z}}
\newcommand{\F}{\mathbb{F}}

\newcommand{\mfr}{\mathfrak}

\newcommand{\bt}{\begin{theorem}}
\newcommand{\et}{\end{theorem}}
\newcommand{\bd}{\begin{definition}}
\newcommand{\ed}{\end{definition}}

\newcommand{\one}{1\!\!1}

\newcommand{\Hom}{{\rm{Hom}}}

\newcommand{\ind}{{\rm{ind}}}
\newcommand{\Irr}{{\rm{Irr}}}
\newcommand{\Alg}{{\rm{Alg}}}

\newcommand{\Ind}{{\rm{Ind}}}

\newcommand{\St}{{\mathrm{St}}}

\newcommand{\diag}{{\rm diag}}

\newcommand{\tr}{{\rm tr}}
\newcommand{\Ker}{{\rm Ker}}

\newcommand{\Gal}{{\rm Gal}}

\newcommand{\Nm}{{\rm Nm}}

\newcolumntype{P}[1]{>{\RaggedRight\arraybackslash}p{#1}}

\tikzset{
	node/.style={
		circle,
		fill,
		inner sep=0.5pt
	},
	font=\footnotesize,
	every pin edge/.style={
		very thin,
		shorten <=1.5pt
	}
}
\usepackage[margin=2.2cm,top=2.5cm,headheight=16pt,headsep=0.3in,heightrounded]{geometry}
\usepackage{fancyhdr}
\def\VR{\kern-\arraycolsep\strut\vrule &\kern-\arraycolsep}

\numberwithin{equation}{section}

\title {Structure of twisted Jacquet Modules of principal series representations of $GL_{2n}(F)$}

\author{C. Harshitha}
\address{C. Harshitha, Department of Mathematics, Indian Institute of Science Education and Research Tirupati, Srinivasapuram,
	Yerpedu, Tirupati, Andhra Pradesh, 517619,
	India}
\email{harshitha.c@students.iisertirupati.ac.in}

\author{C. G. Venketasubramanian}
\address{C. G. Venketasubramanian, Department of Mathematics, Indian Institute of Science Education and Research Tirupati, Srinivasapuram,
	Yerpedu, Tirupati, Andhra Pradesh, 517619,
	India}
\email{venketcg@iisertirupati.ac.in}

\begin{document}

\begin{abstract}  
	{\footnotesize Let $F$ be a non-archimedean local field or a finite field.  Let $\pi$ be a principal series representation of $GL_{2n}(F)$  induced from any of its maximal standard parabolic subgroups. Let $N$ be the unipotent radical of the maximal parabolic subgroup $P$ of $GL_{2n}(F)$ corresponding to the partition $(n,n).$ In this article, we describe the structure of the twisted Jacquet module $\pi_{N,\psi}$ of $\pi$ with respect to $N$ and a non-degenerate character $\psi$ of $N.$  We also provide a necessary and sufficient condition for  $\pi_{N,\psi}$ to be non-zero and show that the twisted Jacquet module is non-zero under certain assumptions on the inducing data. As an application of our results, we obtain the structure of twisted Jacquet modules of certain non-generic irreducible representations of $GL_{2n}(F)$ and establish the existence of their Shalika model in the non-archimedean case. We conclude our article with  a conjecture by Dipendra Prasad classifying the smooth irreducible representations of $GL_{2n}(F)$ with a non-zero twisted Jacquet module.}
\end{abstract}
\maketitle
\tableofcontents
\section{Introduction}\label{Introduction}
\subsection{Background and Motivation}\label{Background}
\subsubsection{}
Let $F$  be a non-archimedean local field or a finite field of order $q.$ Let $G$ be the $F$-rational points of a connected reductive group defined over $F.$ Let $P$ be a parabolic subgroup of $G$ with Levi decomposition $P=MN,$ $(\pi,V)$  a smooth representation of $G$ and $\psi$ a non-trivial character of $N.$ Let $V_{N,\psi}(\pi)$ denote the largest quotient of $\pi$ on which $N$ acts via the character $\psi.$ Since $M$ normalizes $N,$ it acts on the group of characters of $N$ by conjugation. The space of coinvariants $V_{N,\psi}(\pi)$ is naturally an $M_{\psi}$-module, where  $M_{\psi}$ is the stabilizer of the character $\psi$ in $M.$ The representation $(\pi_{N,\psi},V_{N,\psi}(\pi))$  of the group $M_{\psi}$ is called the twisted Jacquet module of $\pi$ with respect to the pair  $(N,\psi).$ Determining the structure of twisted Jacquet modules is not only an interesting question in its own right, but also has found applications \cite{DPSarah,GS, DPRamin,DPRaghuramDuke, Brooks-Schmidt} in various problems in the representation theory of $p$-adic groups.
\subsubsection{}
Determining  the structure of twisted Jacquet modules of representations of various reductive groups with respect to specific pairs $(N,\psi)$ has been a topic of interest in many works in the literature. We briefly recall certain instances which provides motivation to the present work.  D. Prasad \cite{DPDegIMRN}  determined  the structure of twisted Jacquet module of a cuspidal representation of $GL_{2n}(\F_q)$ with respect to a non-degenerate character of the unipotent radical of standard parabolic subgroup $P_{n,n}$ corresponding to the partition $(n,n)$ of $2n.$  Gorodetsky and Hazan \cite{OfirZahi} extended the work of Prasad to the case of $GL_{k.n}(\F_q).$  Balasubramanian et al. (\cite{KH1},\cite{KH2},\cite{KH3})  extended the work of Prasad to the computation of twisted Jacquet module of a cuspidal representation of $GL_{2n}(\F_q)$ with respect to degenerate characters of the  unipotent radical of $P_{n,n}.$
\subsubsection{}
In the case of a non-archimedean local field $F,$  the computation of the structure of twisted Jacquet modules for principal series representations of $GSp_4(F)$  was considered Prasad-Takloo-Bighash  \cite{DPRamin} and Roberts-Schmidt \cite{Brooks-Schmidt} in the context  study of Bessel models.  Pandey and Venketasubramanian \cite{SV} have computed the twisted Jacquet module of a principal series representation  of $Sp_4(F)$ induced from either its Siegel or Klingen parabolic subgroup, with respect to a character of the unipotent radical of the Siegel parabolic subgroup associated to a rank one quadratic form, where $F$ is a non-archimedean local field or a finite field. Some other recent works in the computation of twisted Jacquet modules are by Nadimpalli-Seth \cite{NS} and Parashar-Patel \cite{PP}.
\subsubsection{}
The primary aim of this article is to describe the structure of twisted Jacquet modules of a principal series representation of the general linear group $GL_{2n}(F)$ induced from any of its maximal standard parabolic subgroups, where $F$ is either a non-archimedean local field or a  finite field. We also illustrate a few applications of our results. Our work is motivated by the results of D. Prasad (\cite{DPDegIMRN} and \cite[Propsoition 7.1]{GS}) where the structure of the twisted Jacquet module of a principal series representation of $GL_4(F)$  induced from the maximal standard parabolic subgroup associated to the partition $(2,2)$ was determined.

\subsubsection{}
Based on the results obtained in this work, Dipendra Prasad has formulated a conjecture on non-vanishing of twisted Jacquet modules for irreducible smooth representations $\pi$ of $GL_{2n}(F)$ for a non-archimedean local field $F.$ The conjecture is stated in terms of existence of poles of the adjoint $L$-function of $\pi$ at $s=1,\dots,n.$ We state the conjecture at the end of our article and verify that the conjecture holds in few cases as a consequence of our results. 

\subsection{Summary of main results} \label{summaryofresults}
\subsubsection{}
We shall state the main results of our article after setting a few basic notations for which we shall largely follow \cite{BZ1}, \cite{BZ2} and \cite{Zelevinsky}. For a locally compact Hausdorff group $G,$ $\delta_G$ shall denote its modular character.  If $G$ is an $\ell$-group, the category of smooth complex representations of $G$ is denoted by $\Alg(G).$  If $H$ is a closed subgroup of $G$ and $\rho\in \Alg(H),$ $i_H^G(\rho)$ shall denote normalized compactly induced representation (see Section \ref{Inductionfunctors}).
\subsubsection{}
If $\pi\in \Alg(G)$, $U$ is  a closed subgroup of $G$ which is a union of its compact open subgroups and $\theta$ is a character of $U,$ the twisted Jacquet module of $\pi$ taken with respect to $(U,\theta)$ will be denoted by $\pi_{U,\theta}.$ The  normalized twisted Jacquet module corresponding to $\pi_{U,\theta}$ will be denoted by $r_{U,\theta}(\pi)$ (see Section \ref{TJMdefn}).
\subsubsection{}
Let $F$ be a non-archimedean local field or a finite field. Let $\mathcal{M}_{m,n}$ denote the $F$-vector space of all $m \times n$ matrices over $F.$ If $m=n,$ we shall write $\mathcal{M}_n$  instead of $\mathcal{M}_{n,n}.$ The identity matrix of order $n$ is denoted by  $I_n.$ For brevity, the group $GL_n(F)$ of all invertible matrices in $\mathcal{M}_n$ will be denoted  by $ G_n.$   For a non-archimedean local field $F$, $|~|_{F}$ shall denote the standard absolute value on $F$ and $\nu$ shall denote the character  $\det \circ |~|_F$ of $G_n.$ If $n_j  (1\leq j\leq k)$ are positive integers such that $\sum n_j=n,$  the maximal standard parabolic subgroup of $G_{n}$ corresponding to the partition $(n_1,\dots, n_k)$ of $n$ will be denoted by $P_{n_1,\dots,n_k}.$ It has a Levi decomposition $P_{n_1,\dots, n_k}=M_{n_1,\dots, n_k}N_{n_1,\dots,n_k}.$ The maximal standard parabolic subgroup $P_{n,n}$ of $G_{2n}$ shall be denoted in short by $P.$  The Levi subgroup $M_{n,n}$ and the unipotent radical $N_{n,n}$ of $P$ shall be respectively denoted by $M$ and $N.$
\subsubsection{}
 We shall fix a non-trivial additive character $\psi_0$ of $F$ throughout this article.  Define a  character $\psi $ of $N$ by
\begin{equation}\label{definitionofpsi}
\psi\left(\begin{pmatrix}
I_n& x\\0 &I_n
\end{pmatrix}\right) = \psi_0(tr(x)),
\end{equation} for $x\in \mathcal{M}_n.$ The stabilizer of the character $\psi$ in $M$ is the group $\Delta G_n,$ the diagonal copy of $G_n$ in the Levi subgroup $M$ of $P.$ If $\pi\in\Alg(G_{2n}),$ then  $\pi_{N,\psi}\in \Alg(\Delta G_n).$
\subsubsection{} We now formulate the question that we have considered in this article. Let $r$ be an integer such that $1\leq r <2n$ and let $\rho\in \Alg(M_{r,2n-r}).$ Denote by $i_{P_{r,2n-r}}^{G_{2n}}(\rho)$ the principal representation of $G_{2n}$ obtained from the pair $(\rho,M_{r,2n-r})$ by normalized parabolic induction. Our task is to determine the structure of $(i_{P_{r,2n-r}}^{G_{2n}}(\rho))_{N,\psi}$ as a $\Delta G_n$-module, where $\psi$ is given by \eqref{definitionofpsi}.
\subsubsection{}Our first result  provides a necessary and sufficient condition for the twisted Jacquet module of a principal series representation of $G_{2n}$ parabolically induced from a maximal standard parabolic subgroup to be non-zero. For an integer $r,$ we shall denote by $\lfloor \frac{r}{2} \rfloor$ the greatest integer not exceeding $\frac{r}{2}.$
\begin{theorem}\label{necsuffintro}
Let $\rho\in \Alg(M_{r,2n-r})$ and let $\pi$ denote the normalized parabolically induced representation $i_{P_{r,2n-r}}^{G_{2n}}(\rho).$ Then,  $\pi_{N,\psi}$  is non-zero if and only if there exists an  integer $k$ with $\max \{0,r-n\} \leq k \leq  \lfloor \frac{r}{2} \rfloor$ satisfying $\rho_{N'_{k},\psi'_k}\neq 0,$ 	where $N_k'$ is the subgroup of $M_{r,2n-r}$ given by  \renewcommand{\kbldelim}{(}
\renewcommand{\kbrdelim}{)}
\begin{equation*}
	N_k'=\left\{\begin{pmatrix}
		I_k & x & y & \vrule & 0      & 0      & 0      \\
		0 & I_k & 0 & \vrule & 0      & 0      & 0      \\
		0 & 0& I_{r-2k}& \vrule & 0      & 0      & 0      \\
		\noalign{\hrule}
		0      & 0      & 0      & \vrule & I_{r-2k} & 0 & w \\
		0      & 0      & 0      & \vrule & 0 & I_{n-r+k} & z \\
		0      & 0      & 0      & \vrule & 0& 0 & I_{n-r+k} \\
	\end{pmatrix}\right\},
\end{equation*}
and  $\psi'_k$ is the character of $N_k'$ defined by   $\psi'_k(n')	=\psi_0(\tr(x)+\tr(z))$ for $n'\in N_k'.$ 
\end{theorem}
We note that Theorem \ref{necsuffintro} reduces  the question of non-vanishing of the twisted Jacquet module of a principal series representation of $G_{2n}$ with respect to  $(N,\psi)$ to that of the inducing data $\rho$ with respect to $(N_k',\psi'_k)$ where $k$ satisfies $\max \{0,r-n\} \leq k \leq  \lfloor \frac{r}{2} \rfloor$.  
\subsubsection{}
The second main result describes the structure of the twisted Jacquet module of a principal series representation of $G_{2n}$ in the case of a non-archimedean local field.  
\begin{theorem}\label{intromaintheorem}
Let $r$ be an integer such that $1\leq r <2n.$ Let  $\rho_1\in \Alg(G_r)$ and  $\rho_2 \in \Alg(G_{2n-r}).$ Put $\rho:= \rho_1 \otimes \rho_2\in \Alg(M_{r,2n-r}).$ Let $\pi$ denote the normalized parabolically induced representation $i_{P_{r,2n-r}}^{G_{2n}}(\rho).$ Put $\alpha=\max\{0, r-n\}$ and $\beta={\lfloor \frac{r}{2} \rfloor}.$ For $\alpha \leq k \leq \beta,$ define subgroups $N(k,1)$ and $N(k,2)$ of $G_{r}$ and $G_{2n-r}$ respectively by 

\begin{equation*}
N(k,1) := \left\{\begin{pmatrix}
I_k&x&y\\
0&I_k&0\\
0&0&I_{r-2k}
\end{pmatrix}: x\in \mathcal{M}_k, y\in \mathcal{M}_{k,r-2k} \right\}, \end{equation*}
and
\begin{equation*}
N(k,2) := \left\{\begin{pmatrix}
I_{r-2k}&0&u\\
0&I_{n-r+k}&v\\
0&0&I_{n-r+k}
\end{pmatrix}: u\in \mathcal{M}_{r-2k,n-r+k}, v\in \mathcal{M}_{n-r+k}\right\}.
\end{equation*}

Let $\psi_{k,1}:N(k,1) \to \C^\times$ and $\psi_{k,2}:N(k,2) \to \C^\times$ be the characters of $N(k,1)$ and $N(k,2)$ defined by
\begin{equation*}
\psi_{k,1} \left(\begin{pmatrix}
I_k&x&y\\
0&I_k&0\\
0&0&I_{r-2k}
\end{pmatrix} \right)
= \psi_0(tr(x)) \ ,  \text{ and }\ \psi_{k,2}\left(\begin{pmatrix}
I_{r-2k}&0&u\\
0&I_{n-r+k}&v\\
0&0&I_{n-r+k}
\end{pmatrix}\right) = \psi_0(tr(v)).
\end{equation*}	
	
Then, the twisted Jacquet module $\pi_{N,\psi}$ of $\pi$ has a filtration 

\[\{0\}\subset V_{\alpha}\subset \dots \subset V_{\beta}=\pi_{N,\psi}\]
such that  for $\alpha\leq k \leq \beta $ we have an isomorphism of $\Delta G_n$-modules given by
\begin{equation}
V_k/V_{k-1}\cong  i_{\Delta P_{k,r-2k,n-r+k}}^{\Delta G_n}\left(\tau_k \otimes[ \delta_{P_{r,2n-r}}^{1/2}]^{w_{k}}\otimes \delta_{P_{k,r-2k,n-r+k}}^{-3/2}\right),
\end{equation}
where 

\begin{equation}
\tau_k\left(\begin{pmatrix}
	a&d&e & \VR 0&0&0\\
0&b&f& \VR 0&0&0\\
0&0&c& \VR 0&0&0\\
\hline
0&0&0& \VR a&d&e\\
0&0&0& \VR 0&b&f\\
0&0&0& \VR 0&0&c
\end{pmatrix}\right)
\end{equation}
\begin{equation*}
= \nu(a)^{\frac{r-2k}{2}}\nu(b)^{-\frac{k}{2}}r_{N(k,1),\psi_{k,1}}(\rho_1)\left(\begin{pmatrix}
a&0&0\\0&a&d\\0&0&b
\end{pmatrix}\right) \otimes \nu(b)^{\frac{n-r+k}{2}}\nu(c)^{\frac{2k-r}{2}}r_{N(k,2),\psi_{k,2}}(\rho_2)\left(\begin{pmatrix}
b&f&0\\0&c&0\\0&0&c
\end{pmatrix}\right),
\end{equation*} 
and 
\begin{equation*}
w_{k}=\begin{pmatrix}
I_k & 0& 0& 0 &0&0 \\
0&0 & 0& I_{r-2k} & 0&0\\
0&0&0&0&I_{n-r+k}&0\\
0& I_k &0 &0  &0&0 \\
0& 0& I_{r-2k}&0  &0&0 \\
0& 0&0 &0 &0 & I_{n-r+k}
\end{pmatrix}.
\end{equation*}
\end{theorem}

\subsubsection{} 
We invite the reader to compare the pair $(N(k,1),\psi_{k,1})$ appearing in Theorem \ref{intromaintheorem} with the pair $(H,\psi)$ appearing in \cite[1.2.1]{Naor} in the context of uneven Shalika models. We observe that the pairs $(N(k,1),\psi_{k,1})$ and   $(H,\psi)$ are closely related in the sense that $N(k,1)$ is a subgroup of $H$ (considered as a subgroup of $G_r$ associated to the partition $(k,k,r-2k)$ of $r$) and the character $\psi_{k,1}$ is the restriction of $\psi$ to $N(k,1).$ 

\subsubsection{}
In view of Theorem \ref{intromaintheorem}, we can refine Theorem \ref{necsuffintro} when the inducing representation $\rho$ is of the form $\rho_1\otimes \rho_2$ where $\rho_1$ and $\rho_2$ are  smooth representations of $G_r$  and $G_{n-r}$ respectively, which is stated in Corollary \ref{nonzeroTJMcomponentwise}. The finite field analogue of Theorem \ref{intromaintheorem} is simpler and considered in Theorem \ref{mainfinite}.

\subsubsection{}
We illustrate two applications of Theorem \ref{intromaintheorem} obtained  by considering the twisted Jacquet module of a product of two characters. If $\pi_j\in \Alg(G_{n_j})$ for $1\leq j \leq m$, $\pi_1\times \dots \times \pi_m$ denotes the Bernstein-Zelevinsky product (see  \cite{BZ2, Zelevinsky}). 

\begin{theorem}\label{productoftwocharcatersintro}
	Let $\chi$ and $\mu$ be characters of $G_{r}$ and $G_{2n-r}$ respectively. Then, 
	\begin{equation*}
		(\chi\times \mu)_{N,\psi}\simeq \begin{cases}
			0 & \mbox{ if } r\neq n\\
			\chi\otimes \mu & \mbox{ if } r=n
		\end{cases}
	\end{equation*}
	as $\Delta G_n$-modules.	
\end{theorem}

\begin{theorem}\label{LQintro}
Let $\chi$ be a character of $F^{\times}.$	For $1< \alpha \leq n,$  let $L_{\chi, \alpha}$ denote the Langlands quotient of the representation 
	\begin{equation*}\chi\nu^{\frac{n-1}{2}+\alpha}\times \cdots \times \chi\nu^{\frac{n-1}{2}+2}\times St_2\chi\nu^{\frac{n}{2}} \times \cdots \times St_2\chi\nu^{-\frac{n}{2}+\alpha}  \times \chi\nu^{-\frac{n-1}{2}+\alpha-2}\times \cdots \times \chi\nu^{-\frac{n-1}{2}}
	\end{equation*} 
	of $G_{2n}.$
	Also, let $L_{\chi,1}$ denote the Langlands quotient of the representation 
	\begin{equation*}
	St_2\chi\nu^{\frac{n}{2}} \times St_2\chi\nu^{\frac{n}{2}-1} \times \cdots \times St_2\chi\nu^{-\frac{n}{2}+1}
		\end{equation*}
	of $G_{2n}.$ Then, for each $\alpha\in \{1,\dots, n\}$, when have $(L_{\chi,\alpha})_{N,\psi}\simeq \chi\otimes \chi\nu^{\alpha}$ as $\Delta G_n$-modules.
\end{theorem}

\subsubsection{}
As a corollary to Theorem \ref{LQintro}, we deduce the following result on existence of Shalika model for $L_{\nu^{-\frac{\alpha}{2}}, \alpha}$ which might be of interest to some readers.
 
\begin{corollary}\label{SMintro}
	For each $\alpha\in \{1,\dots, n\},$ the non-generic irreducible smooth representation $L_{\nu^{-\frac{\alpha}{2}}, \alpha}$ of $G_{2n}$ possesses a Shalika model. Consequently, it has a non-zero $G_n\times G_n$-invariant linear form.
\end{corollary} 
We can guarantee  that the twisted Jacquet module of a principal series representation of $G_{2n}$ is  non-zero in some situations.  One instance when the twisted Jacquet module taken with respect to $(N,\psi)$  is non-zero is when the parabolic induction is considered  from the subgroup $P$ itself.
\begin{theorem}\label{nonzeroTJMintro}
Suppose $\rho_1,\rho_2\in \Alg(G_{n}).$ Then, $\rho_1\otimes \rho_2\subset  (\rho_1\times \rho_2)_{N,\psi}$ and consequently, $ (\rho_1\times \rho_2)_{N,\psi}$  is non-zero.
\end{theorem}

\subsubsection{}
If the inducing maximal standard parabolic subgroup is not $P,$  one cannot guarantee that the twisted Jacquet module is non-zero as Theorem \ref{productoftwocharcatersintro} has shown. If one starts with a generic representation of $G_{n+1},$ we can prove the following result.  
\begin{theorem}\label{nonzeroTJMgenericcomponentintro}
Let $\rho\in \Irr(G_{n+1})$ be generic and $\eta\in \Irr(G_{n-1})$ be any representation. Then, 
\begin{enumerate}
	\item $(\rho\times \eta) _{N,\psi}\neq 0$ and,
	\item $(\eta\times \rho) _{N,\psi} \neq 0.$
\end{enumerate}
\end{theorem}

\subsubsection{Conjecture on non-vanishing of $\pi_{N,\psi}$}
We conclude the article  with  the following conjecture by Dipendra Prasad classifying  the smooth irreducible representations of $GL_{2n}(F)$ with a non-zero twisted Jacquet module. 
\begin{conjecture}[D. Prasad]\label{DPConjectureintro}
Let $\pi$ be an irreducible smooth representation of $G_{2n}$ and let $\rho=\rec_F^{2n}(\pi)$ denote its Langlands parameter. Then, the twisted Jacquet module $\pi_{N,\psi} = 0$ if and only if the $L$-function $L(s, \pi, Ad)=L(s,\rho\otimes \rho^{\vee})$ has poles of order greater than or equal to $n,n-1,\dots, 1$ at $s=1,2,\dots, n$ respectively.
\end{conjecture}

\subsection{Outline of the article}
\subsubsection{}
 Section \ref{NotationsPrelims} contains notations and some preliminary results that are used in this article. In Section \ref{MTC}, we recall results necessary for  Mackey theory and its application relevant to our work. The key ingredient to apply Mackey Theory is a suitable set of  double coset representatives  obtained by the authors in \cite{HV1}.  The relevant results of \cite{HV1} are summarized in Theorem \ref{summarydoublecosets}.  Section \ref{Structure} forms the core of the paper where the main results Theorems \ref{necsuffintro} and \ref{intromaintheorem} are proved.  The analogues of the main structure theorem in the case of a finite field is stated in Theorem \ref{mainfinite}.  
 
 \subsubsection{}
 In  Section \ref{Applications}, we give a few applications of our main results.  We derive the structure of twisted Jacquet modules of certain Bernstein-Zelevinsky products and prove Theorem \ref{productoftwocharcatersintro}.  As a consequence, we show the existence of Shalika model 
(Corollary \ref{SMintro}) for certain non-generic representations denoted $L_{\nu^{-\frac{\alpha}{2}},\alpha},$ by showing that their twisted Jacquet module is trivial via our computations. We also show that in  some instances  (including Theorems \ref{nonzeroTJMintro}, \ref{nonzeroTJMgenericcomponentintro}) the twisted Jacquet module is non-zero and obtain few other results on twisted Jacquet modules of subquotients of certain principal series representations. In Section 6, we present a conjecture of D. Prasad (Conjecture \ref{DPConjectureintro}) on non-vanishing  twisted Jacquet modules and verify that the conjecture holds in few special cases (Section \ref{generic-section} \& \ref{verify}).

\subsection{Few Remarks}
\subsubsection{}
We end this introduction with a few remarks on our strategy to the proof of Theorem \ref{intromaintheorem}. Let $S_{\psi}$ denote the stabilizer of the character $\psi$ in $P.$ As a first step, we have a complete set of double coset representatives $w_{k,l}$ (Theorem \ref{summarydoublecosets}) of $(S_{\psi},P_{r,2n-r})$ in $G_{2n}$ which were obtained in \cite{HV1}. The $w_{k,l}$'s are indexed by a pair of integers $k,l$ such that $\max\{0,r-n\}\leq k \leq \min\{r,n\}$ and $\max\{0,r-n\}\leq l \leq \min\{k,r-k\}.$  
 \subsubsection{}
 We note that the Geometric Lemma of Bernstein-Zelevinsky  \cite[\S 5.2 Theorem]{BZ1} gives a general recipe to compute the twisted Jacquet module of a principal series representation. To apply the Geometric Lemma,  the double coset representatives must satisfy a crucial property, namely decomposability conditions (\S 5.1(4) \cite{BZ1}). However, some of the $w_{k,l}$'s do not satisfy certain decomposability conditions (see Remark \ref{decomposability}) for $n\geq 3.$   While dealing with the case where $n=r=2$ (\cite{DPDegTIFR}, \cite[Proposition 7.1]{GS}), Prasad has adopted a strategy to compute the twisted Jacquet module without resorting to the Geometric Lemma, even though the double coset representatives  obtained there satisfy all the decomposability conditions. In view of this, our approach is inspired by the proof of Proposition 7.1 in \cite{GS}. 

 \subsubsection{}
We use Mackey theory to observe that the restriction to $S_{\psi}$ of a principal series representation $\pi=i_{P_{r,2n-r}}^{G_{2n}}(\rho)$ is glued from representations $\pi_{k,l}$ which are parameterized by $w_{k,l}$'s. At this stage, one may appeal to the exactness of the twisted  Jacquet functor to conclude that  $\pi_{N,\psi}$  is glued from $(\pi_{k,l})_{N,\psi}$'s.  The key ingredient to calculate $(\pi_{k,l})_{N,\psi}$  is Proposition \ref{TJM-reduction} which generalizes Lemma 7.2 and Lemma 7.3 of Prasad \cite{GS}. We remark here that a direct proof of the finite field analogue of Proposition \ref{TJM-reduction} is provided in Proposition \ref{TJM-reduction-analoguefinite}. We note that Proposition \ref{TJM-reduction}  and its finite group analogue Proposition \ref{TJM-reduction-analoguefinite} may be of independent interest to some readers.  Using Proposition \ref{TJM-reduction}, we obtain a  necessary and sufficient condition (Theorem \ref{necsuffintro}) for $\pi_{N,\psi}$ to be non-zero.

\subsubsection{}
However, to give the structure of $\pi_{N,\psi}$ as a filtration as in Theorem \ref{intromaintheorem}, we make an appeal to Theorem \ref{summarydoublecosets} (1) and (2). The key observation is that the  restriction of $\pi$ to $P$ is given by a filtration,  where the representations in the filtration, say $\pi_k,$ are indexed by an integer $k$ satisfying $\min\{0,r-n\}\leq k\leq \max\{r,n\}.$ But, upon further restriction to $S_{\psi},$ the semi-simplification of each $\pi_{k}$  consists of $\pi_{k,l}$'s where $\max\{0,r-n\}\leq l \leq \min\{k, r-k\} .$ The crucial point we note is that $(\pi_{k,l})_{N,\psi}$ is zero except possibly when $l=k$ and such an equality is possible only when $k\leq \lfloor \frac{r}{2}  \rfloor.$ The important fact to be noted here is that each $\pi_k$ on restriction to $S_{\psi}$ has at most one factor, namely $\pi_{k,k},$ which may contribute to the twisted Jacquet module. In summary, we exploit the precise  structure of the restriction of $\pi$ to $P$ along with the observation that $(\pi_{k,l})_{N,\psi}$ is non-zero possibly only when $l=k$ which also forces $k\leq \lfloor \frac{r}{2}\rfloor$ to deduce the final structure as in Theorem \ref{intromaintheorem}. We remark that this additional difficulty to obtain the filtration for  $\pi_{N,\psi}$ is present only when $n\geq 4.$ 
  
  \subsubsection{}
  Our description of the double coset representatives as $w_{k,l}$ indexed by a pair of integers $(k,l)$ has allowed us to compare our result with that of Prasad obtained when $n=r=2.$ Comparing our results with (\cite{DPDegIMRN} and \cite[Propsoition 7.1]{GS}), out of the four double coset representatives in this case, exactly two contributes to the twisted Jacquet module because they correspond to the double cosets $w_{0,0}$ and $w_{1,1}$ and, the inducing data satisfies the necessary and sufficient condition of Theorem \ref{necsuffintro} for both $k=0$ and $k=1.$

\section{Notations and Preliminaries}\label{NotationsPrelims}
\subsection{Notations}\label{Notations}
We shall collect some notations and terminology (\cite{BZ1}, \cite{BZ2}) in this section which we shall be using in this article. 
\subsubsection{Some generalities}
By an $\ell$-group, we mean a locally compact Hausdorff group which has a fundamental system of neighborhoods at the identity consisting of open compact subgroups. Recall that the category of smooth complex representations of an $\ell$-group $G$ is denoted by $\Alg(G).$ A character of an $\ell$-group $G$ is a smooth one dimensional representation of $G.$ Given a character $\chi$ of $G$ and $(\pi,V)\in \Alg(G),$ we shall sometimes denote  $\pi \otimes \chi$ by $\pi\chi.$ If $H$ is a closed subgroup of $G$ and $g\in G,$ $H^g$ shall denote the conjugate $gHg^{-1}$ of $H.$ Moreover, if $(\rho,W)\in \Alg(H)$ we have the representation $(\rho^{g},W)$ of $H^g$ defined for each $g\in G$ by $\rho^g(x)=\rho(g^{-1}xg)$ for $x\in H^g.$

\subsubsection{Module of an automorphism} Let $G$ be an $\ell$-group and $U$ be a closed subgroup of $G$ such that $U$ is a union of its compact open subgroups. Let $du$ be a left Haar measure on $U$ and $\sigma$ be an automorphism of $U.$ The module of the automorphism $\sigma,$  \cite[\S 1.7]{BZ2} denoted by ${\rm mod}(\sigma),$ is defined by the formula
\begin{equation*}
\int_U f(\sigma^{-1}(u)) du={\rm mod}(\sigma)\int_Uf(u) du
\end{equation*}
 for all locally constant and compactly supported functions $f$ on $U.$ Let $g\in G$ normalize $U.$ The module of the automorphism $u\mapsto gug^{-1}$ is denoted by ${\rm mod}_U(g).$ Note that ${\rm mod}_U$ is a character of the normalizer of $U$ in $G.$ The character ${\rm mod}_G$ will be denoted by $\delta_G$ and is called the modular character of $G.$

\subsection{Induction functors}\label{Inductionfunctors}
 \subsubsection{}
Let $H$ be a closed subgroup of an $\ell$-group $G$ and  $(\rho,W)\in \Alg(H).$ A function $f:G \to W$ is said to be smooth if there exists a compact open subgroup $K_f$ of $G$ such that $f(gk)=f(g)$ for all $k \in K_f$ and all $ g \in G.$ Define
\begin{equation*}
\Ind_H^G(W)=\{f:G\to W : f\mbox{ is smooth and }  f(hg)=\rho(h)(f(g)), \forall h\in H, \forall g\in G\}.
\end{equation*}

The action of  $G$ on $\Ind_H^G(W)$ by right translation will be denoted by $\Ind_H^G(\rho).$ The representation $(\Ind_H^G(\rho), \Ind_H^G(W))$ is called an unnormalized induction. Let 
\begin{equation*}
\ind_H^G(W)=\{f\in \Ind_H^G(W): f \mbox{ is compactly supported modulo } H\}.\end{equation*}
The restriction of the action of $G$ to the subspace $\ind_H^G(W)$ will be denoted by $\ind_H^{G}(\rho).$ The pair $(\ind_H^G(\rho), \ind_H^G(W))$ is called an unnormalized compactly induced representation.

 \subsubsection{}
For several purposes, one also uses normalized versions of the induction functors $\Ind_H^G$ and $\ind_H^G$. Let $I_H^G(W)=\{f:G\to W: f\mbox{ is smooth}, f(hg)=\delta_G^{-\frac{1}{2}}(h)\delta_H^{\frac{1}{2}}(h)\rho(h)(f(g)) \mbox{ for all } h\in H, g\in G\}.$ The corresponding representation $(I_H^G(\rho), I_H^G(W))$
obtained by the action of $G$ by right translation is called a normalized induced representation. Let $i_H^G(W)$ denote the subspace of $I_H^G(W)$ consisting of functions which are compactly supported modulo $H.$ The representation $(i_H^G(\rho), i_H^G(W))$ is called a normalized compactly induced representation.We note that $i_H^G(\rho)=\ind_H^G(\delta_H^{\frac{1}{2}} \delta_G^{-\frac{1}{2}}\rho).$ All the induction functors are exact \cite{BZ1}.

 \subsubsection{}
Let  $Q=LU$ be any parabolic subgroup of $G_n$ and $\rho\in \Alg(L).$ Consider $\rho$ as a representation of $Q$ by declaring it to be trivial on $U.$ The representation $\ind_Q^{G_n}(\rho)$ is a called a parabolically induced representation or a principal series representation. By a normalized parabolically induced representation $i_Q^{G_n}(\rho),$ we mean the parabolically induced representation $\ind_Q^{G_n}(\delta_Q^{1/2}\rho).$

\subsection{Twisted Jacquet module}\label{TJMdefn} 
 \subsubsection{}
 Suppose $U$ is a closed subgroup of an $\ell$-group $G$ such that $U$ is a union of its compact open subgroups. For $(\pi,V)\in \Alg(G)$ and a character $\theta$ of $U,$ let $V(U,\theta)$  denote the subspace  $\langle \pi(u)v-\theta(u)v: u\in U, v\in V\rangle$ and let $V_{U,\theta}$ denote the quotient $V/V(U,\theta).$ Let $N_G(U,\theta)=\{g\in G: gug^{-1}\in U \mbox{ and } \theta(gug^{-1})=\theta(u) \mbox { for all } u\in U\}.$ Let $L$ be a closed subgroup  of $G$ which is contained in $N_G(U,\theta).$ The group $L$ preserves $V(U,\theta)$ and therefore acts on the quotient $V_{U,\theta}$ via 
$$m\cdot u=\pi(m)(v)+V(U,\theta)$$ 
for $m\in L, v\in V.$ We shall denote this action of $L$ on $V_{U,\theta}$ by $\pi_{U,\theta}.$ The representation $(\pi_{U,\theta}, V_{U,\theta})$ is called the (unnormalized) twisted Jacquet module of $\pi$ taken with respect to $(U,\theta).$ 

 \subsubsection{}
 As with induction, we also have the notion of a normalized  twisted Jacquet module. The normalized twisted Jacquet module is the representation $({\rm mod}_U^{-\frac{1}{2}} \pi_{_{U,\theta}}, V_{_{U,\theta}}).$ We shall denote ${\rm mod}_U^{-\frac{1}{2}} \pi_{U,\theta}$ by $r_{U,\theta}(\pi).$ It is well known \cite[\S Proposition 2.35]{BZ1} that the functor $\pi\mapsto \pi_{U,\theta}$ is exact. For  $\pi\in \Alg(G_{2n})$, we note that the twisted Jacquet module $\pi_{N,\psi}$ ($\psi$ as in \eqref{definitionofpsi}) is isomorphic to the normalized twisted Jacquet module $r_{N,\psi}(\pi)$ as ${\rm mod}_N=\delta_P$ acts trivially on $\Delta G_n.$

\subsection{Some standard subgroups of $G_n$}
 \subsubsection{}
Let $F$ be a non-archimedean local field or a finite field. If $H$ is a subgroup of $G_n, \Delta H$ shall denote the subgroup $\left\{ \begin{pmatrix}
g&0\\0&g
\end{pmatrix}: g\in H\right\}$ of $G_{2n}.$ Let $F^{n}$ denote the $n$-dimensional vector space over $F$  with its standard basis $\{e_1,  \ldots , e_{n}\}.$ Let $\Omega_k(F^{n})$ denote the set of all $k$-dimensional subspaces of $F^{n}.$
The group $G_{n}$ acts transitively on $\Omega_k(F^{n})$ and the stabilizer in $ G_{n}$ of $\langle e_1, \ldots , e_k \rangle $ under this action is the maximal standard parabolic subgroup given by \begin{equation*}
	P_{k,n-k} = \left\{ \begin{pmatrix}  g_1 & x\\ 0 & g_2 \end{pmatrix} : g_1 \in G_{k} , g_2 \in G_{n-k}, x \in \mathcal{M}_{k,n-k} \right\}.
\end{equation*}

\subsubsection{}
Let \begin{equation*}
M_{k,n-k} = \left\{ \begin{pmatrix}  g_1 & 0\\ 0 & g_2 \end{pmatrix} : g_1 \in G_{k} , g_2 \in G_{n-k} \right\} \text{ and } N_{k,n-k} = \left\{ \begin{pmatrix}  I_k & x\\ 0 & I_{n-k} \end{pmatrix} :  x \in \mathcal{M}_
{k,n-k}\right\}.
\end{equation*} 
One has the Levi decomposition $P_{k,n-k} = M_{k,n-k} N_{k,n-k},$ where  $M_{k,n-k} $ and $N_{k,n-k}$ are respectively called the Levi subgroup and  the unipotent radical of $P_{k,n-k}.$ More generally, if $n_j(1\leq j\leq k) $ are positive integers such that $\sum n_j=n,$ one has the standard parabolic subgroup $P_{n_1,\dots,n_k},$  its Levi subgroup $M_{n_1,\dots,n_k}\simeq G_{n_1}\times \dots G_{n_k}$ and its unipotent radical is denoted by $N_{n_1,\dots,n_k}$ (see \cite[\S 3]{BZ1}).

\subsection{The character $\psi$ of $N$ and associated subgroups}
 \subsubsection{}
Recall that we denote the maximal standard parabolic subgroup $P_{n,n}$ of $G_{2n}$ by $P,$ its Levi subgroup $M_{n,n}$ by $M$ and its unipotent radical $N_{n,n}$ by $N$ respectively.  We also recall from Section \ref{summaryofresults} that throughout this article, we fix a non-trivial additive character $\psi_0$ of $F$  and  $\psi:N\to \C^{\times}$ is the character defined by \eqref{definitionofpsi}. Since $N$ is normal in $P,$ $P$ acts by conjugation on $N$ and hence on  $\hat{N},$ the group of characters of $N.$ We will denote the stabilizer of $\psi$ in $P$ and $M$ by $S_{\psi}$ and $M_{\psi}$ respectively. We observe that

\begin{equation}\label{definitionSpsi}
	S_{\psi} = \left\{ \begin{pmatrix}
		g & x\\ 0 & g 
	\end{pmatrix} : g \in G_n, x \in \mathcal{M}_n\right\},	
\end{equation}  
and 

\begin{equation}\label{definitionMpsi}
	M_{\psi} = \left\{ \begin{pmatrix}
		g & 0\\ 0 & g 
	\end{pmatrix} : g \in G_n\right\}.	
\end{equation}  
The subgroup $S_{\psi}$ is known as the Shalika subgroup of $G_{2n}.$ We note that $S_{\psi}=M_{\psi}\ltimes N$ and $M_{\psi}=\Delta G_n,$ the diagonal copy of $G_n$ embedded in $G_{2n}.$
 
\subsection{Preliminary Results}
 \subsubsection{}
We collect a few preliminary results that we shall need in the sequel.
\begin{lemma}\label{ConjTJM}
Suppose $G$ is an $\ell$-group. Let $U$ be a closed subgroup of $G$ which is a union of its compact open subgroups and $\theta:U\to\C^{\times}$ be a character. Assume that $L$ is a closed subgroup of $G$ such that $L\cap U=\{e\}$ and $L$ normalizes $U$ and $\theta.$   The following statements hold:

\begin{enumerate}
	\item For $g\in G$ and $(\rho,V)\in \Alg(L),$ we have an isomorphism $(\rho_{_{U,\theta}})^{g}\simeq (\rho^{g})_{U^g,\theta^g}$
	of $L^g$-modules.
	\item If $\chi$ is a character of $LU$ which is trivial on $U$ and $(\rho,V)\in \Alg(LU),$ we have an isomorphism 
	$(\chi\otimes\rho )_{_{U,\theta}}\simeq \chi\otimes \rho_{_{U,\theta}} $  of $L$-modules.
\end{enumerate}

\end{lemma} 
\begin{proof}
To prove (1), it is enough to observe that $V(U^g,\theta^g)=V(U,\theta).$ For proving (2), let $V_{\rho}(U,\theta)$ and $V_{\chi\otimes \rho}(U,\theta)$ denote the space of $(U,\theta)$-invariants of the representations $(\rho,V)$ and $(\chi\otimes\rho, V)$ respectively. It is easy to see that $V_{\chi\otimes\rho}(U,\theta)=V_{\rho}(U,\theta),$ from which (2) follows. 
\end{proof}

\begin{lemma}\label{TJM split in tensor product}
For $i\in \{1,2\},$ let $G_i$ be an $\ell$-group with closed subgroups $L_i,U_i$  such that $G_i = L_i \ltimes U_i$ and $U_i$  is a union of its compact open subgroups.   Suppose also that $\theta_i$ is a character of $U_i$ such that $L_i$ normalizes $U_i$ and $\theta_i$ for $i\in \{1,2\}.$ Put $U=U_1 \times U_2$ and $\theta = \theta_1 \times \theta_2,$ i.e., $\theta(u_1,u_2)=\theta_1(u_1)\theta_2(u_2)$ for $(u_1,u_2)\in U_1\times U_2.$ Then, we have the following:
	\begin{enumerate}
		\item If $\pi \in \Alg(G_1 \times G_2),$ then $ (\pi_{U_2, \theta_2})_{U_1, \theta_1} = \pi_{U,\theta}.$
		\item If $\pi_i \in \Alg(G_i)$ for $i\in \{1,2\}$ then $(\pi_1)_{U_1,\theta_1} \otimes \pi_2 \cong (\pi_1\otimes \pi_2)_{U_1 \times \{1\}, \theta_1 \times 1}$ and $\pi_1 \otimes (\pi_2)_{U_2,\theta_2}\cong (\pi_1\otimes \pi_2)_{\{1\}\times U_2, 1\times \theta_2}$. 
	\end{enumerate}
\end{lemma}
The proof of Lemma \ref{TJM split in tensor product} follows from \cite[\S 1.9, Proposition (c) and (g)]{BZ2}.
 \subsubsection{}
 Next, we recall  the following result from \cite[(VI.5.1.3), pp. 191]{renard} which shall be needed later. 
\begin{theorem}\label{Renard-general}
Suppose $J$ is an $\ell$-group and $L,U$ are closed subgroups of $J$ such that $L$ and $U$ are closed in $J,$ $U$ is normal in $J$ and $U$ is a union of its compact open subgroups. 	Assume that $LU$ is closed in $J.$ Let $(\rho, W)\in \Alg(L).$ Consider the Jacquet module $\rho_{_{L\cap U, 1}}$ as a representation of the group $L/L \cap U\cong LU/U \subset J/U.$ Then, there is a natural isomorphism of $J/U$-modules:
	\begin{equation}\label{GeneralJM}
		(\ind_L^J(\rho))_{_{U,1}} \cong \ind_{LU/U}^{J/U}((\rho)_{_{L\cap U,1}}\otimes \delta^{-1})	
	\end{equation}
	where $\delta$ is the modular character of the action of $L$ on $L\cap U\backslash U.$
\end{theorem}
For a proof of Theorem \ref{Renard-general}, see \cite[pp. 191-195]{renard}.

\section{Double cosets and Applications}\label{MTC}
Let $r$ be an integer such that $1\leq r<2n$ and let $\rho\in \Alg(M_{r,2n-r}).$ Let $\pi$ denote the normalized parabolically induced representation $i_{P_{r,2n-r}}^{G_{2n}}(\rho).$
The first step in determining the structure of $\pi_{N,\psi}$ is to determine its restriction to $S_{\psi}.$ Naturally, this brings us to the problem of determining a complete set of  $(S_{\psi},P_{r,2n-r})$-double coset representatives in $G_{2n}.$  In Section \ref{DC}, we present the double coset decomposition for the space $S_{\psi}\backslash G_{2n}/P_{r,2n-r}.$ Some subgroups relevant for Mackey theory are introduced in Section \ref{Subgroups} .  The restriction of $i_{P_{r,2n-r}}^{G_{2n}}(\rho)$ to  the subgroups $S_{\psi}$ and $P$ are determined in Section \ref{restrictions}.
\subsection{Double coset representatives}\label{DC}
\subsubsection{}
We summarize the main results of \cite{HV1} obtained by the authors  regarding  the double coset space $S_{\psi}\backslash G_{2n}/P_{r,2n-r}.$

\begin{theorem}\label{summarydoublecosets} 	
	Let $n$ and  $r$ be integers such that $1\leq r<2n.$ 
Put $\alpha=\max\{0, r-n\}$ and $\gamma=\min\{r,n\}.$  Let $k$ and $l$ be integers such that $\alpha\leq k \leq \gamma$ and $\alpha \leq l\leq \min\{k, r-k\}.$  Put
	\begin{equation}\label{definitionwkl}
	w_{k,l}:=\begin{bmatrix}
		I_k & 0& 0& 0&0 &0 \\
		0&0 & 0& I_{n-k}&0 & 0\\
		0& I_l &0 &0 &0 &0 \\
		0&0 &0 &0 & I_{k-l}&0 \\
		0& 0& I_{r-(k+l)}&0 &0 &0 \\
		0& 0&0 &0 &0 & I_{n-r+l}
	\end{bmatrix},
\end{equation}

\begin{equation}\label{Jkalphadefn}
	J_{k,\alpha}:= w_{k,\alpha}P_{r,2n-r}w_{k,\alpha}^{-1} \cap P,
\end{equation}
and, 
\begin{equation}
	\sigma_{k,l}:=\begin{pmatrix}
		I_k& 0& 0& 0& 0 &0 &0 &0\\
		0&I_{n-k}&0&0&0&0&0&0\\
		0& 0& I_{\alpha}&0&0&0&0&0\\
		0&0&0&0&0&I_{l-\alpha}&0&0\\
		0&0&0&I_{k-l}&0&0&0&0\\
		0&0&0&0&0&0&I_{r-(k+l)}&0\\
		0&0&0&0&I_{l-\alpha}&0&0&0\\
		0&0&0&0&0&0&0&I_{n-r+\alpha}
	\end{pmatrix}.
\end{equation}

Then, the following statements hold:
\begin{enumerate}
	\item For each $r$ such that $1\leq r<2n,$ $\{w_{k,\alpha}: \alpha\leq k \leq  \gamma\}$ is a complete set of double coset representatives for $P\backslash G_{2n}/ P_{r,2n-r}.$
	\item For each $k$ such that $\alpha\leq k \leq \gamma,$ $\{\sigma_{k,l}: \alpha\leq l\leq \min\{k, r-k\}\}$ is a complete set of double coset representatives for $S_{\psi}\backslash P/ J_{k,\alpha}.$
	\item  $\sigma_{k,l} \circ w_{k,\alpha}=w_{k,l}$ for each $k$ and $l$ satisfying $\alpha\leq k \leq \gamma, \alpha\leq l\leq \min\{k,r-k\}.$
	\item $\{w_{k,l}: \alpha\leq k \leq \gamma, \alpha\leq l \leq \min\{k,r-k\}\}$ is a complete set of double coset representatives for $S_{\psi}\backslash G_{2n}/ P_{r,2n-r}.$
\end{enumerate}		
\end{theorem}
	
\subsection{Some Subgroups}\label{Subgroups}
\subsubsection{} 
Let $n$ and  $r$ be integers such that $1\leq r<2n.$ 
Put $\alpha=\max\{0, r-n\}$ and $\gamma=\min\{r,n\}.$  Let $k$ and $l$ be integers such that $\alpha\leq k \leq \gamma$ and $\alpha \leq l\leq \min\{k, r-k\}.$  
  For such $k$ and $l,$ we determine the subgroup $w_{k,l}P_{r,2n-r}w_{k,l}^{-1}\cap S_{\psi}.$ In a majority of our calculations, we choose to write a block matrix $\begin{pmatrix} 
	a & b\\
	c & d
\end{pmatrix}$ where $a\in \mathcal{M}_{n_1,n_3}, b\in \mathcal{M}_{n_1,n_4}, c\in \mathcal{M}_{n_2,n_3}$ and $d\in \mathcal{M}_{n_2,n_4}$ by  

\renewcommand{\kbldelim}{(}
\renewcommand{\kbrdelim}{)}
\[
\kbordermatrix{
	&n_3& n_4\\
	n_1& a & b\\
	n_2& c & d}
,\]
specifically indicating the sizes of the blocks which play a crucial role nevertheless.
Following this notation, we write $p\in P_{r,2n-r}$  as 
\renewcommand{\kbldelim}{(}
\renewcommand{\kbrdelim}{)}
\[
p = \kbordermatrix{
	&k&l& r-k-l& & n-k &k-l&n-r+l\\
	k&g_1&g_2& g_3 & \VR x_1 & x_2&x_3\\
	l&g_4&g_5& g_6 & \VR x_4&x_5& x_6\\
	r-k-l& g_7 & g_8 & g_9& \VR x_7 & x_8 & x_9\\
	\hline
	n-k & 0&0&0& \VR h_1 & h_2 & h_3 \\
	k-l & 0& 0& 0& \VR h_4 & h_5 & h_6\\
	n-r+l & 0& 0& 0& \VR h_7 & h_8 & h_9
}.
\]
\subsubsection{}
We then have,
\renewcommand{\kbldelim}{(}
\renewcommand{\kbrdelim}{)}
\begin{equation}\label{wklP-conjugate}
	w_{k,l}p  w_{k,l}^{-1}= \kbordermatrix{
		&k&n-k& & l&k-l&r-k-l&n-r+l\\
		k&g_1&x_1& \VR g_2&x_2&g_3&x_3\\
		n-k&0&h_1& \VR 0&h_2&0&h_3\\
		\hline
		l&g_4&x_4& \VR g_5&x_5&g_6&x_6\\
		k-l&0&h_4& \VR 0&h_5&0&h_6\\
		r-k-l&g_7&x_7& \VR g_8&x_8&g_9&x_9\\
		n-r+l&0&h_7& \VR 0&h_8&0&h_9\\
	}.
\end{equation}

\subsubsection{}
To describe succinctly certain subgroups that shall frequently arise in our computations, we define subsets $A_{k,l}, B_{k,l}\subset \mathcal{M}_n$ as follows:

\renewcommand{\kbldelim}{(}
\renewcommand{\kbrdelim}{)}
\begin{equation}\label{AklandBkl}
	A_{k,l}=\left\{ \kbordermatrix{
		&l&k-l&r-k-l&n-r+l \\
		l&a&r&d& e\\
		k-l&0&r'&0&e'\\
		r-k-l&0&0&b&f\\
		n-r+l&0&0&0&c 
	}\right\},	 \ \ \
	\renewcommand{\kbldelim}{(}
	\renewcommand{\kbrdelim}{)}
	B_{k,l}= \left\{\kbordermatrix{
		&l&k-l&r-k-l&n-r+l\\
		l&x&m&y&u\\
		k-l& x'&m'&y'&u'\\
		r-k-l&0&w'&0&w\\
		n-r+l &0&z'&0&z
	}\right\}.
\end{equation}

\subsubsection{}

Put $S_{k,l}:= w_{k,l} P_{r,2n-r} w_{k,l}^{-1} \cap S_{\psi}.$ Also, put $H_{k,l}:=S_{k,l}\cap M =w_{k,l} P_{r,2n-r} w_{k,l}^{-1} \cap M_{\psi}$ and $N^{k,l}:= S_{k,l}\cap N=w_{k,l} P_{r,2n-r} w_{k,l}^{-1} \cap N.$
Then, 
\begin{equation}\label{Skldefinition}
	S_{k,l}= \left\{\begin{pmatrix}
		g&X\\
		0&g
	\end{pmatrix}: g\in A_{k,l}\cap G_n, X\in B_{k,l}\right\},
\end{equation}

\begin{equation}
	H_{k,l}=\left\{\begin{pmatrix}
		g&0\\
		0&g
	\end{pmatrix}:g\in A_{k,l}\cap G_n\right\},
\end{equation}
and 
\begin{equation}\label{HklcapN}
	N^{k,l}= \left\{\begin{pmatrix}
		I_n&X\\
		0&I_n
	\end{pmatrix}: X\in B_{k,l}\right\}.
\end{equation}

\subsubsection{}
Note that for each $k,$  $l$ is an integer such that  $\alpha \leq l \leq \min\{k,r-k\}$ which in particular implies $l\leq k.$  
If $l<k,$ we define a certain subgroup $U^{k,l}$ of $N^{k,l}$  by
\begin{equation}\label{Ukl-definition}
	U^{k,l}:=\left\{\kbordermatrix{
		&l&k-l&r-k-l&n-r+l & & l&k-l&r-k-l&n-r+l\\
		l&I_l&0&0&0 & \VR 0&0&0&0\\
		k-l&0&I_{k-l}&0&0 & \VR 0 &x&0&0\\
		r-k-l&0&0&I_{r-k-l}&0& \VR 0&0&0&0\\
		n-r+l&0&0&0&I_{n-r+l}& \VR 0&0&0&0\\
		\hline
		l&0&0&0&0& \VR I_l&0&0&0\\
		k-l&0&0&0&0& \VR 0&I_{k-l}&0&0\\
		r-k-l&0&0&0&0& \VR 0&0&I_{r-k-l}&0\\
		n-r+l&0&0&0&0& \VR 0&0&0&I_{n-r+l}
	} : x\in \mathcal{M}_{k-l}\right\}.
\end{equation}
The subgroup $U^{k,l}$ and its conjugate $w_{k,l}^{-1}U^{k,l} w_{k,l}$ will play a crucial role while determining  $\pi_{N,\psi}.$ 

\subsubsection{} We determine $w_{k,l}^{-1}U^{k,l} w_{k,l}.$  Let us write $w_{k,l}$ as
\[w_{k,l}=\kbordermatrix{
	&l&k-l&l&r-k-l & & r-k-l&n-r+l&k-l&n-r+l\\
	l&I_l&0&0&0 & \VR 0&0&0&0\\
	k-l&0&I_{k-l}&0&0 & \VR 0 &0&0&0\\
	r-k-l&0&0&0&0& \VR I_{r-k-l}&0&0&0\\
	n-r+l&0&0&0&0& \VR 0&I_{n-r+l}&0&0\\
	\hline
	l&0&0&I_l&0& \VR 0&0&0&0\\
	k-l&0&0&0&0& \VR 0&0&I_{k-l}&0\\
	r-k-l&0&0&0&I_{r-k-l}& \VR 0&0&0&0\\
	n-r+l&0&0&0&0& \VR 0&0&0&I_{n-r+l}
}.\]
It can be verified that 
\begin{equation}\label{Ukl-conjugate}
	w_{k,l}^{-1}U^{k,l}w_{k,l}=\left\{\kbordermatrix{
		&l&k-l&l&r-k-l & & r-k-l&n-r+l&k-l&n-r+l\\
		l&I_l&0&0&0 & \VR 0&0&0&0\\
		k-l&0&I_{k-l}&0&0 & \VR 0 &0&x&0\\
		l&0&0&I_{l}&0& \VR 0&0&0&0\\
		r-k-l&0&0&0&I_{r-k-l}& \VR 0&0&0&0\\
		\hline
		r-k-l&0&0&0&0& \VR I_{r-k-l}&0&0&0\\
		n-r+l&0&0&0&0& \VR 0&I_{n-r+l}&0&0\\
		k-l&0&0&0&0& \VR 0&0&I_{k-l}&0\\
		n-r+l&0&0&0&0& \VR 0&0&0&I_{n-r+l}
	}: x\in \mathcal{M}_{k-l}\right\}.
\end{equation}

\subsubsection{}
We shall need the following lemma which records a few properties of the subgroups defined in this section so far and the character $\psi_{|_{N^{k,l}}}.$
\begin{lemma}\label{structureofSkl} 
	Let $n$ and  $r$ be integers such that $1\leq r<2n$ and put $\alpha=\max\{0, r-n\}$ and $\gamma=\min\{r,n\}.$
	For each $k,l$ satisfying $\alpha\leq k \leq \gamma, \alpha\leq l \leq \min\{k,r-k\},$ we have the following:
	\begin{enumerate}
		\item The group $N^{k,l}$ is a normal subgroup of $S_{k,l}$ and $S_{k,l}=H_{k,l}\ltimes N^{k,l}.$
		\item $S_{k,l}$ normalizes the character $\psi_{|_{N^{k,l}}}.$
		\item If $l<k,$ the character $\psi$ is a non-trivial character of $U^{k,l}.$
	\end{enumerate}  
\end{lemma}
\begin{proof}
	It is trivial to see that $H_{k,l}\cap N^{k,l}=\{e\}.$ Let $A=\begin{pmatrix}
		g &X\\0&g
	\end{pmatrix}\in S_{k,l}.$ From \eqref{Skldefinition}, $g\in A_{k,l}\cap G_n$ and $X\in B_{k,l}.$  It can be verified by  a direct calculation  that $g^{-1}X\in B_{k,l}$ and therefore,  we can write $A=\begin{pmatrix}
		g &0\\0&g
	\end{pmatrix}\begin{pmatrix}
		I_n &g^{-1}X\\0&I_n
	\end{pmatrix}\in H_{k,l}N^{k,l}.$  Also, by a similar calculation we can see that $gXg^{-1}\in B_{k,l}.$  This proves that $H_{k,l}$ normalizes $N^{k,l}$  and  consequently $N^{k,l}$ is normal in $S_{k,l},$ proving (1). Since $S_{\psi}$ normalizes $\psi,$ the statement (2) is obvious. Finally, (3) follows from \eqref{Ukl-definition} since $\psi_0$ is non-trivial. 
\end{proof}

\begin{remark}\label{decomposability}
	In the terminology of \cite[\S 5.1]{BZ2}, statement (1) in Lemma \ref{structureofSkl} means that $S_{k,l}$ is decomposable with respect to the pair $(M_{\psi},N).$ 	We note that for $n=2$ and $r=2,$  there are $4$ double coset representatives, namely, $\{w_{0,0}, w_{1,0}, w_{1,1}, w_{2,0}\}.$ We refer the reader to \cite[Propsoition 7.1]{GS} to compare with our double coset representatives. We note that the double coset representatives appearing in 
	\cite[Propsoition 7.1]{GS} satisfy the decomposability conditions \cite[\S 5.1 (4)]{BZ2} in the hypothesis of Bernstein-Zelevinsky's Geometric Lemma. However, this is not the case when $n\geq 3.$ For example, if we consider the representative, $w_{2,1}$ when $n=r\geq 3$  the decomposability condition $w_{k,l}^{-1}S_{\psi}w_{k,l}\cap P=(w_{k,l}^{-1}S_{\psi}w_{k,l}\cap M) \cdot (w_{k,l}^{-1}S_{\psi}w_{k,l} \cap N)$  fails to hold.
\end{remark}

\subsection{Restriction of $i_{P_{r,2n-r}}^{G_{2n}}(\rho)$ to $S_{\psi}$ and $P$}\label{restrictions}
\subsubsection{}
We have the following theorem which describes the restriction of a parabolically induced representation to the subgroup $S_{\psi}.$
\begin{theorem}\label{RestrictiontoSpsi}
Let $n$ and  $r$ be integers such that $1\leq r<2n.$ 
		Put $\alpha=\max\{0, r-n\}$ and $\gamma=\min\{r,n\}.$
	Let $\pi=i_{P_{r,2n-r}}^{G_{2n}}(\rho)$ where $\rho\in \Alg(M_{r,2n-r}).$ For each $k$ and $l$ satisfying $\alpha \leq k \leq \gamma$ and $\alpha\leq l \leq \min\{k,r-k\},$ let $w_{k,l}$ denote the double coset representative of $S_{\psi}\backslash G_{2n}/P_{r,2n-r}$ defined by \eqref{definitionwkl} and let $S_{k,l}$ be as in \eqref{Skldefinition}.  Then, the restriction of $\pi$ to the subgroup $S_{\psi}$ is glued from the $S_{\psi}$-modules $\ind_{S_{k,l}}^{S_{\psi}}([\rho\otimes \delta_{P_{r,2n-r}}^{1/2}]^{w_{k,l}}).$ 
\end{theorem}

The proof of Theorem \ref{RestrictiontoSpsi} follows from Theorem \ref{summarydoublecosets}(4) and  Mackey theory for which we refer the reader to \cite[\S 3.2.1]{Omer} and \cite[Proposition 1.17]{BD}. 

\subsubsection{}
 In the proof of our main structure theorem, apart from Theorem \ref{RestrictiontoSpsi}, we shall use the following result which gives a filtration of a principal series when restricted to the subgroup $P.$ In view of Theorem \ref{summarydoublecosets}(1), it is natural that the restriction under consideration is parameterized by $w_{k,\alpha}.$
\begin{theorem}\label{RestrictiontoP}
Let $n$ and  $r$ be integers such that $1\leq r<2n.$ 
	Put $\alpha=\max\{0, r-n\}$ and $\gamma=\min\{r,n\}.$	Let $\pi=i_{P_{r,2n-r}}^{G_{2n}}(\rho)$ where $\rho\in \Alg(M_{r,2n-r}).$ Put $J_{k,\alpha}=w_{k,\alpha}P_{r,2n-r}w_{k,\alpha}^{-1}\cap P.$ Then, $\pi_{|_P}$ is given by a filtration 
\begin{equation}\label{filtrationofrestoP}
	0\subset W_{\alpha}\subset \dots \subset W_{\gamma}=\pi_{|_{P}}
\end{equation}
where for each $k$ with $\alpha \leq k \leq \gamma$, we have an isomorphism $W_{k}/W_{k-1}\cong \ind_{J_{k,\alpha}}^{P}([\rho\otimes \delta_{P_{r,2n-r}}^{1/2}]^{w_{k,\alpha}})$ of $P$-modules.
\end{theorem}

\section{Structure of twisted Jacquet modules}\label{Structure}
This section forms the core of our article.  In Section \ref{Non-zero criteria}, we obtain a necessary and sufficient condition for the twisted Jacquet module of a principal series representation to be non-zero (Theorem \ref{nonzeroTJM}). We prove a first version of the main structure theorem in Section \ref{structurep-adic} (Theorem \ref{mainpadic}) in the non-archimedean case. In Section \ref{componentwise}, we prove a refined version of Theorem \ref{mainpadic} suitable for applications (Theorem \ref{maincomponentwise} \& Corollary \ref{normalisedcomponentwisecor}).  The finite field analogue of the main structure theorem (Theorem \ref{maincomponentwise} ) obtained in the non-archimedean case is established in Section \ref{structure-finite} (Theorem \ref{mainfinite}).

\subsection{Criteria for the twisted Jacquet module to be non-zero}\label{Non-zero criteria}
\subsubsection{}
In this subsection, we will give a necessary and sufficient condition for  $(i_{P_{r,2n-r}}^{G_{2n}}(\rho))_{N,\psi}$ to be non zero where $\rho\in \Alg(M_{r,2n-r}).$ Write $\pi=i_{P_{r,2n-r}}^{G_{2n}}(\rho).$ By Theorem \ref{RestrictiontoSpsi} and exactness of the twisted Jacquet functor, to determine whether $\pi_{N,\psi}$ is non-zero or not, it is sufficient to calculate $(\ind_{S_{k,l}}^{S_{\psi}}([\rho\otimes \delta_{P_{r,2n-r}}^{1/2}]^{w_{k,l}}))_{N,\psi}$ for each $k$ and $l.$ As a first step of our calculation, we shall show that  $(\ind_{S_{k,l}}^{S_{\psi}}([\rho\otimes \delta_{P_{r,2n-r}}^{1/2}]^{w_{k,l}}))_{N,\psi}=0$ whenever $l\neq  k,$ i.e., when $l<k.$ To achieve this, we require the following variant of Theorem \ref{Renard-general} which holds the key to our further calculation.

\begin{proposition}\label{TJM-reduction}
	Assume that $G_0$ is an $\ell$-group, $G,H,U, U_0$ are closed subgroups of $G_0$ such that $H$ is a closed subgroup of $G$ and $U_0$ is a closed subgroup of $U.$ Suppose also that $G$  normalizes $U,$ $G\cap U=\{e\},$ $U$ is a union of its compact open subgroups and $\theta:U\to \C^{\times}$ is a character which is normalized by $G.$  Let $\theta_0$ denote the restriction of $\theta$ to $U_0$ and let $H$ normalize $U_0$ and $\theta_0.$ Assume further that $GU,HU_0$ are closed subgroups of $G_0$ such that $HU_0$ is a closed subgroup of $GU.$  Then, for any  $(\tau,W)\in \Alg(HU_0)$ we have the following isomorphism of $G$-modules:
	
	\begin{equation}\label{GeneralTJM}
		(\ind_{HU_0}^{GU}(\tau))_{_{U,\theta}} \cong \ind_{H}^{G}(\tau_{_{U_{_0},\theta_{_0}}}\otimes \delta^{-1})	
	\end{equation}
	where $\delta$ is the inverse of the module of the automorphism induced by the action of $H$ on $U_0 \backslash U.$
\end{proposition}
\begin{proof}
Since $G$ normalizes $U$ and $\theta,$ we can consider $\theta:GU\to \C^{\times}$ as a character by defining $\theta(gu)=\theta(u)$ for $g\in G, u\in U.$ Similarly, we can consider $\theta_0:HU_0\to \C^{\times}$ as a character. Also, for any $\pi\in \Alg(GU),$ by \cite[Remark 2.31]{BZ1}, there is an isomorphism $\pi_{_{U,\theta}} \simeq (\pi\otimes \theta^{-1})_{_{U,1}}$ of $G$-modules. In view of this, we obtain an  isomorphism	$(\ind_{HU_0}^{GU}(\tau))_{_{U,\theta}} \simeq (\ind_{HU_0}^{GU}(\tau)\otimes \theta^{-1})_{_{U,1}}$ of $G$-modules.	We recall the generality that for a character $\chi$ of an $\ell$-group $G'$ and a smooth representation $\sigma$ of a closed subgroup $H'$ of $G',$ one has $\chi\otimes \ind_{H'}^{G'}(\sigma)=\ind_{H'}^{G'}(\chi_{{|_{H'}}}\otimes \sigma).$ Applying this generality, $(\ind_{HU_0}^{GU}(\tau)\otimes \theta^{-1})_{_{U,1}}\simeq (\ind_{HU_0}^{GU}(\tau\otimes \theta_0^{-1}))_{_{U,1}}.$ Now, we apply  Theorem \ref{Renard-general} with $J=GU,L=HU_0,U=U$ and $\rho=\tau\otimes \theta_{0}^{-1}.$ It is easily verified that $J/U=GU/U\simeq G, LU/U=(HU_0)U/U=HU/H\simeq H$ and $ L\cap U=HU_0\cap U=U_0.$ Applying \eqref{GeneralJM}, we have an isomorphism of $G$-modules:
\begin{equation*}
(\ind_{HU_0}^{GU}(\tau\otimes \theta_0^{-1}))_{_{U,1}}\simeq \ind_{H}^{G}((\tau\otimes \theta_0^{-1})_{_{U_0,1}}\otimes \delta^{-1}),
\end{equation*}
where $\delta$ is the inverse of the module of automorphism of induced by the action of $H$ on $U_0 \backslash U.$ On the other hand, again by \cite[Remark 2.31]{BZ1} , $(\tau\otimes \theta_0^{-1})_{_{U_0,1}}\simeq \tau_{U_0,\theta_0}.$ Thus, clubbing all isomorphisms together, the proposition is proved.\end{proof}

\begin{remark}\label{DPLemmas}
We note that Proposition \ref{GeneralTJM} is a combined generalization of Prasad's lemmas stated in \cite[Lemmas 7.2 \& 7.3]{GS}  while considering the twisted Jacquet module of a representation of the form $i_{P_{2,2}}^{G_4}(\rho_1\otimes \rho_2).$
\end{remark}

\subsubsection{} We obtain the following corollary to Proposition \ref{TJM-reduction}, which says that the only double cosets $w_{k,l}$ which can  possibly contribute to $\pi_{N,\psi}$ are $w_{k,k}.$
\begin{corollary}\label{knotldoesnotcontribute}
Let $n$ and  $r$ be integers such that $1\leq r<2n.$ Put $\alpha=\max\{0, r-n\}$ and $\gamma=\min\{r,n\}.$ For each  integer $k$ and $l$ such that $\alpha \leq k \leq \gamma$ and $\alpha\leq l \leq \min\{k,r-k\},$ let $w_{k,l}$ (see \eqref{definitionwkl}) denote the corresponding double coset representative of $S_{\psi}\backslash G_{2n}/P_{r,2n-r}.$  If $l\neq k$ then  for any  $\rho \in \Alg(M_{r,2n-r})$ considered as a representation of $P_{r,2n-r}$ by inflating trivially across $N_{r,2n-r},$ one has 
\begin{equation*}\left(\ind_{S_{k,l}}^{S_{\psi}}([\rho\otimes \delta_{P_{r,2n-r}}^{1/2}]^{w_{k,l}})\right)_{N,\psi}=0.
\end{equation*}
\end{corollary}

\begin{proof}
Put $G=\Delta G_n, U= N, H=H_{k,l}, U_0=N^{k,l}$ and $\theta=\psi$ in Proposition \ref{TJM-reduction}. By Lemma \ref{structureofSkl}, all the hypothesis in Proposition \ref{TJM-reduction} are verified for these choices of $G,H,U, U_0$ and $\theta.$  Denote the character $\psi_{|_{N^{k,l}}}$ by $\psi_{k,l}.$ Applying Proposition \ref{TJM-reduction}, we obtain the following isomorphism of $\Delta G_n$-modules:

\begin{equation}\label{mainisomorpshim}
\left(\ind_{S_{k,l}}^{S_{\psi}}([\rho\otimes \delta_{P_{r,2n-r}}^{1/2}]^{w_{k,l}})\right)_{N,\psi}\simeq \ind_{H_{k,l}}^{\Delta G_n}\left(\left([\rho\otimes \delta_{P_{r,2n-r}}^{1/2}]^{w_{k,l}}\right)_{N^{k,l},\psi_{k,l}}	\otimes \delta^{-1}\right),
\end{equation}
where $\delta$ is the inverse of the module of automorphism induced by the action of $H_{k,l}$ on $N^{k,l}\backslash N.$

To prove the statement of the corollary, we shall show that if $l<k$, $\left([\rho\otimes \delta_{P_{r,2n-r}}^{1/2}]^{w_{k,l}}\right)_{N^{k,l},\psi_{k,l}}=0$ or equivalently that $\Hom_{N^{k,l}}([\rho\otimes \delta_{P_{r,2n-r}}^{1/2}]^{w_{k,l}}, \psi_{k,l}) =0.$ To this end,  as $U^{k,l}\subset N^{k,l}$, we note that 
\begin{equation}\label{lnotkreduction}
\Hom_{N^{k,l}}([\rho\otimes \delta_{P_{r,2n-r}}^{1/2}]^{w_{k,l}}, \psi_{k,l})\subset \Hom_{U^{k,l}}([\rho\otimes \delta_{P_{r,2n-r}}^{1/2}]^{w_{k,l}}, {\psi_{k,l}}_{|U^{k,l}}).
\end{equation}
Let $V_{\rho}$ denote the representation space of $\rho$ and assume that $\ell$ is a linear form in the right hand side  of \eqref{lnotkreduction}. Note that   the conjugate $\delta_{P_{r,2n-r}}^{w_{k,l}}$ is trivial on $U^{k,l}.$ Then, $\ell$ satisfies 
\begin{equation*}
\ell(\rho (w_{k,l}^{-1}uw_{k,l})v )=\psi(u)\ell(v)
\end{equation*}
for all $u\in U^{k,l}$ and $v\in V_{\rho}.$
By \eqref{Ukl-conjugate}, $w_{k,l}^{-1}uw_{k,l}\subset N_{r,2n-r}$ which forces $\rho(w_{k,l}^{-1}uw_{k,l})=1.$
So, $(\psi(u)-1)\ell(v)=0$ for all $u\in U^{k,l}$ and $v\in V_{\rho}.$ Since $l\neq k,$ by Lemma \ref{structureofSkl} (3), $\psi$ is nontrivial yielding $\ell=0.$  \end{proof}

\begin{remark}\label{mainisom}
We note that \eqref{mainisomorpshim} holds irrespective of whether $l<k$ or not, In the proof of Corollary \ref{knotldoesnotcontribute}, the fact that $l\neq k$ is only used while invoking Lemma \ref{structureofSkl} (3).
\end{remark}

\begin{remark}\label{reductiontoequalcase} To determine $(i_{P_{r,2n-r}}^{G_{2n}}(\rho))_{N,\psi}$ up to  semi-simplification, it suffices to compute  $(\ind_{S_{k,k}}^{S_{\psi}}([\rho\otimes \delta_{P_{r,2n-r}}^{1/2}]^{w_{k,k}}))_{N,\psi}$  in view of Corollary \ref{knotldoesnotcontribute}.
	Since  $ l \leq \min\{k,r-k\},$ we see that $l=k$ implies $k\leq \lfloor \frac{r}{2} \rfloor,$ where  $\lfloor \frac{r}{2} \rfloor$ denotes the greatest integer not exceeding $\frac{r}{2}.$ 
\end{remark}

\subsubsection{The $l=k $ case}
By Remark \ref{reductiontoequalcase}, in our further computation, we need only consider the case when $l=k.$ In this case, the elements, $w_{k,k}$ and the subgroups $S_{k,k}, H_{k,k}$ and $N^{k,k}$ have a simpler structure which we note below.  For simplicity of notation, we shall put $w_k:=w_{k,k},S_k:=S_{k,k}, H_k:=H_{k,k}$ and $N_k:=N^{k,k}.$  In a similar vein, we shall denote the subsets $A_{k,k}$ and $B_{k,k}$ of $\mathcal{M}_n$ (see  \eqref{AklandBkl}) by $A_{k}$ and $B_k$ respectively.
\subsubsection{}
 We write  $w_k$ as 
\begin{equation}\label{definitionwkk-alternate}
	w_{k}=\begin{bmatrix}
		I_k & 0& 0& 0 &0&0 \\
		0&0 & 0& I_{r-2k} & 0&0\\
		0&0&0&0&I_{n-r+k}&0\\
		0& I_k &0 &0  &0&0 \\
		0& 0& I_{r-2k}&0  &0&0 \\
		0& 0&0 &0 &0 & I_{n-r+k}
	\end{bmatrix}.
\end{equation}

\subsubsection{$A_k$ and $B_k$}
If $l=k,$ we note that
\renewcommand{\kbldelim}{(}
\renewcommand{\kbrdelim}{)}
\begin{equation}\label{AkBk}
	A{_k}= \left\{\kbordermatrix{
		&k&r-2k&n-r+k \\
		k&a&d&e \\
		r-2k&0&b&f\\
		n-r+k&0&0&c 
	}\right\},  \ \ \
	\renewcommand{\kbldelim}{(}
	\renewcommand{\kbrdelim}{)}
	B{_k}= \left\{\kbordermatrix{
		& k&r-2k&n-r+k\\
		k& x&y&u\\
		r-2k& 0&0&w\\
		n-r+k& 0&0&z
	}\right\}.
\end{equation}

\subsubsection{$S_k,H_k$ and $N_k$}
Finally, when $l=k$ one has
\renewcommand{\kbldelim}{(}
\renewcommand{\kbrdelim}{)}
\begin{equation}\label{Skdefinition}
	S_k= \left\{\begin{pmatrix}
		g&X\\
		0&g
	\end{pmatrix}: g\in A_k\cap  G_n, X\in B_k\right\},
	\end{equation}
	
	\begin{equation}\label{Hkdefinition}
	H_k=\left\{\begin{pmatrix}
		g&0\\
		0&g
	\end{pmatrix}: g \in A_k\cap  G_n\right\}= \Delta P_{k,r-2k,n-r+k},
\end{equation}
 and 
\begin{equation}\label{Nkdefinition}
	N_k=\left\{\begin{pmatrix}
		I_n&X\\
		0&I_n
	\end{pmatrix}: X\in B_k\right\}.
\end{equation}
By Lemma \ref{structureofSkl}, $S_{k}=\Delta P_{k,r-2k,n-r+k}\ltimes N_k.$ 

\subsubsection{The character $\psi_k$ of $N_k$}
 Let $\psi_{k}$ denote the restriction of $\psi$  to $N_k.$ Explicitly, for each $n=\begin{pmatrix} I_n & X\\
	0 & I_n\end{pmatrix}\in N_k$ one has
\begin{equation}\label{Definitionofpsik}
\psi_k(n) = \psi_0(tr(X))=\psi_0(tr(x)+tr(z)).
\end{equation}

\subsubsection{The conjugates of $N_k$ and $\psi_k$}
Put $N_k''=w_{k}^{-1}N_k w_{k}.$ Then,

\renewcommand{\kbldelim}{(}
\renewcommand{\kbrdelim}{)}
\begin{equation}\label{DefinitionNkk''}
	N_k''=\left\{\kbordermatrix{
		&k&k&r-2k & & r-2k &n-r+k&n-r+k\\
		k&I_k&x&y & \VR 0&0&u\\
		k&0&I_{k}&0& \VR 0&0&0\\
		r-2k&0&0&I_{r-2k}& \VR 0&0&0\\
		\hline
		r-2k&0&0&0& \VR I_{r-2k}&0&w\\
		n-r+k&0&0&0& \VR 0&I_{n-r+k}&z\\
		n-r+k&0&0&0& \VR 0&0&I_{n-r+k}
	}\right\}\subset P_{r,2n-r}.
\end{equation}
Put $\psi''_k=\psi_k^{(w_{k}^{-1})}.$ Then, $\psi''_{k}:N_k''\to \C^{\times}$  is the character given by  
 $\psi''_k(n)=\psi_k(w_{k}nw_{k}^{-1})$ for $n\in N''_k.$ If $n''\in N_k'',$ then 
\begin{equation}\label{Definitionpsikdoubleprime}
	\psi''_k(n'')
	=\psi_0(\tr(x)+\tr(z))
\end{equation}
where $x\in \mathcal{M}_{k}$ and $z\in \mathcal{M}_{n-r+k}.$ 
\subsubsection{The subgroup $N_k'$ and its character $\psi_k'$}
Next, we define $N_k'=w_{k}^{-1}N_k w_{k}\cap M_{r,2n-r},$ i.e., $N_k'=N_k''\cap M_{r,2n-r}.$ Then,
\renewcommand{\kbldelim}{(}
\renewcommand{\kbrdelim}{)}
\begin{equation}\label{DefinitionNkk'}
	N_k'=\left\{\kbordermatrix{
		&k&k&r-2k & & r-2k &n-r+k&n-r+k\\
		k&I_k&x&y & \VR 0&0&0\\
		k&0&I_{k}&0& \VR 0&0&0\\
		r-2k&0&0&I_{r-2k}& \VR 0&0&0\\
		\hline
		r-2k&0&0&0& \VR I_{r-2k}&0&w\\
		n-r+k&0&0&0& \VR 0&I_{n-r+k}&z\\
		n-r+k&0&0&0& \VR 0&0&I_{n-r+k}
	}\right\}\subset M_{r,2n-r}.
\end{equation}
We also define the character $\psi'_{k}:N_k'\to \C^{\times}$ by  $\psi'_k=\psi_k''{_{|_{N_k'}}}$ for  $n\in N_k'.$ For $n'\in N_k'$
\begin{equation}\label{Definitionpsikprime}
	\psi'_k(n')
	=\psi_0(\tr(x))\psi_0(\tr(z))
\end{equation}
where $x\in \mathcal{M}_{k}$ and $z\in \mathcal{M}_{n-r+k}.$

\subsubsection{Necessary and sufficient condition} With all the notations in place, we present the following theorem which provides a necessary and sufficient condition for the twisted Jacquet module of a principal series representation of $G_{2n}$ to be non-zero. 

\begin{theorem}\label{nonzeroTJM} 
Let $n$ and  $r$ be  integers such that $1\leq r<2n.$ Put $\alpha=\max\{0,r-n\} $ and $\beta= \lfloor \frac{r}{2} \rfloor.$
Let $\rho\in \Alg(M_{r,2n-r})$ and  put $\pi=i_{P_{r,2n-r}}^{G_{2n}}(\rho).$	Let $\psi$ be the character  of $N$ defined by \eqref{definitionofpsi} and $\psi'_k$ be the character of the subgroup $N_k'$ of $M_{r,2n-r}$ defined by \eqref{Definitionpsikprime}. Then,  $\pi_{N,\psi}$  is non-zero if and only if there exists an  integer $k$ with $\alpha\leq k \leq \beta$ such that $\rho_{N'_{k},\psi'_k}\neq 0.$ 
\end{theorem}

\begin{proof}
In view of Corollary \ref{knotldoesnotcontribute} and Remark \ref{reductiontoequalcase}, $\pi_{N,\psi}$ is non-zero if and only if there exists an  integer $k$ with $\alpha \leq k \leq  \lfloor \frac{r}{2} \rfloor$ such that  $(\ind_{S_k}^{S_{\psi}}([\rho\otimes \delta_{P_{r,2n-r}}^{1/2}]^{w_{k}}))_{N,\psi}\neq 0.$ By \eqref{mainisomorpshim} (also see Remark \ref{mainisom}), we have
\begin{equation}\label{mainisomorphism-equal}
\left(\ind_{S_k}^{S_{\psi}}([\rho\otimes \delta_{P_{r,2n-r}}^{1/2}]^{w_{k}})\right)_{N,\psi}\simeq \ind_{\Delta P_{k,r-2k,n-r+k}}^{\Delta G_n}\left(\left([\rho\otimes \delta_{P_{r,2n-r}}^{1/2}]^{w_{k}}\right)_{N_k,\psi_{k}}	\otimes \delta^{-1}\right),	
\end{equation}
where $\delta$ is the inverse of the module of the automorphism induced by the action of $\Delta P_{k,r-2k,n-r+k}$ on $N_k\backslash N.$ Observe that
$\left([\rho\otimes\delta_{P_{r,2n-r}}^{1/2}]^{w_{k}}\right)_{N_k,\psi_{k}}\neq 0$ if and only if  $\Hom_{N_k}([\rho\otimes \delta_{P_{r,2n-r}}^{1/2}]^{w_{k}},\psi_k)\neq 0.$ On the other hand,
\begin{align*}
\Hom_{N_k}([\rho\otimes \delta_{P_{r,2n-r}}^{1/2}]^{w_{k}},\psi_k)&=\Hom_{N_k''}(\rho\otimes\delta_{P_{r,2n-r}}^{1/2}, \psi_k'' )\\
&=\Hom_{N_k''}(\rho, \psi_k'' ) ( \mbox{as } \delta_{P_{r,2n-r}}^{1/2} \mbox{is trivial on } N_k'')\\
&=\Hom_{N_k'}(\rho, \psi_k'),\end{align*}
since both $\rho$ and $\psi_k''$ act trivially on $N_{r,2n-r}.$ This completes the proof of the theorem. \end{proof}

\subsection{Structure in the non-archimedean case I}\label{structurep-adic}
\subsubsection{}
Let $F$ be a non-archimedean local field and $\rho \in \Alg(M_{r,2n-r}).$  Our next aim is to determine  the structure of  $(i_{P_{r,2n-r}}^{G_{2n}}(\rho))_{N,\psi}$ in a general form. Before we embark on doing that, as observed in the proof of Theorem \ref{nonzeroTJM},  it would be of importance to us to establish the transition between the twisted Jacquet modules $(\rho^{w_k})_{N_k,\psi_k}$ and $\rho_{N_k'',\psi_k''},$ which shall turn out to be useful.
\begin{lemma}\label{transition}
Let $n$ and  $r$ be integers such that $1\leq r<2n$ and  let $\rho\in \Alg(M_{r,2n-r}). $ Let $k$ be an integer such that $\max\{0,r-n\}\leq k \leq \lfloor \frac{r}{2} \rfloor.$ Then,  $([\rho \otimes \delta_{P_{r,2n-r}}^{1/2}]^{w_{k}})_{N_{k},\psi_{k}}$ is isomorphic to $[\rho_{N_{k}'',\psi_{k}''} \otimes \delta_{P_{r,2n-r}}^{1/2}]^{w_{k}}$ as $\Delta P_{k,r-2k, n-r+k}$-modules.
\end{lemma}
\begin{proof}
Recall from \eqref{DefinitionNkk''} and \eqref{Definitionpsikdoubleprime} that $(N_{k}'')^{w_{k}} = N_{k}$ and $(\psi_{k}'')^{w_{k}}= \psi_{k}$ to get
\begin{align*}
	([\rho \otimes \delta_{P_{r,2n-r}}^{1/2}]^{w_{k}})_{N_{k},\psi_{k}} & = ([\rho \otimes \delta_{P_{r,2n-r}}^{1/2}]^{w_{k}})_{(N_{k}'')^{w_{k}},(\psi_{k})^{w_{k}}}\\
	& = ((\rho \otimes \delta_{P_{r,2n-r}}^{1/2})_{N_{k}'',\psi_{k}''})^{w_{k}}(\mbox{By Lemma \ref{ConjTJM} (1)})\\
	& = [\rho_{N_{k}'',\psi_{k}''} \otimes \delta_{P_{r,2n-r}}^{1/2}]^{w_{k}}(\mbox{By Lemma \ref{ConjTJM} (2)}) .\qedhere
\end{align*} 
\end{proof}

\subsubsection{Module of the automorphism induced by the action of $\Delta P_{k,r-2k,n-r+k}$ on $N_k\backslash N$} 
We  need to unmask the character $\delta$ appearing in \eqref{mainisomorphism-equal}, which shall be achieved via the following computation.
Let $k$ be an integer such that $\max\{0,r-n\}\leq k \leq \lfloor \frac{r}{2} \rfloor.$ The action of $\Delta P_{k,r-2k,n-r+k}$ on $N_k\backslash N$ induced by the conjugation action of $\Delta P_{k,r-2k,n-r+k}$ on $N$ is via $p \cdot N_kn:= N_kpnp^{-1}.$ It is easy to check that this is an action. In this paragraph, we shall compute the module of the automorphism induced by the action of $\Delta P_{k,r-2k,n-r+k}$ on $N_k\backslash N$ which we denote by $\delta'.$ Let us put

\begin{equation*}U':=\left\{\renewcommand{\kbldelim}{(}
	\renewcommand{\kbrdelim}{)}
	\kbordermatrix{
		&k&r-2k&n-r+k \\
		k&0&0&0 \\
		r-2k&x&y&0 \\
		n-r+k&z&w&0}\right\}, \text{ and } U=\left\{\begin{pmatrix}
		I_n&u\\0&I_n
	\end{pmatrix}:u\in U'\right\}.
\end{equation*}
Note that $U'\subset \mathcal{M}_n$ and $U$ is a subgroup of $N.$ Define a map $\phi:N \to U$ by 

\begin{equation*}	
	\phi\left(\begin{pmatrix}
		&&&a_{11}&a_{12}&a_{13}\\&I_n&&a_{21}&a_{22}&a_{23}\\&&&a_{31}&a_{32}&a_{33}\\&&&&&\\&&&&I_n&\\&&&&&
	\end{pmatrix}\right)= \begin{pmatrix}
		&&&0&0&0\\&I_n&&a_{21}&a_{22}&0\\&&&a_{31}&a_{32}&0\\&&&&&\\&&&&I_n&\\&&&&&
	\end{pmatrix}.\end{equation*}
Clearly, $\phi $ is group homomorphism with $\Ker(\phi)=N_k$ and it induces an isomorphism $\bar{\phi}:N_k\backslash N \to U$ given by $N_kn \mapsto \phi(n).$ For $p\in \Delta P_{k,r-2k, n-r+k}$ and $u \in U,$ observe that 

\begin{equation}\label{DeltaP-action}
	p\cdot N_ku=N_kpup^{-1}=N_k\phi(pup^{-1}).
\end{equation}
Using this, we can transfer the action of $\Delta P_{k,r-2k,n-r+k}$ on $N_k\backslash N$ to $U$ by defining $p \cdot u=\phi(pup^{-1}).$

\begin{lemma}\label{thedeltacomputation}
	Let $\delta'$ denote the module of the automorphism given by the action \eqref{DeltaP-action} of $\Delta P_{k,r-2k,n-r+k}$ on $N_k\backslash N$ induced by the conjugation action of $\Delta P_{k,r-2k,n-r+k}$ on $N.$ Then, for $p\in \Delta P_{k,r-2k,n-r+k},$ we have
	\begin{equation*}
		\delta'(p)=\delta^{-1}_{\Delta P_{k,r-2k,n-r+k}}(p).
	\end{equation*}
\end{lemma}
\begin{proof}
	Write $p\in \Delta P_{k,r-2k,n-r+k}$ as $p=\begin{pmatrix}
		t&0\\
		0&t
	\end{pmatrix}$ where $t =\begin{pmatrix}
		a & d & e\\
		0& b & f\\
		0& 0& c\end{pmatrix}\in P_{k,r-2k, n-r+k}.$ For $u'=\begin{pmatrix}
		I_n&u\\0&I_n
	\end{pmatrix}\in U$ and a locally constant and compactly supported function $f_0$ on $N_k\backslash N,$ which can be viewed as a function on $U$ via $\bar{\phi},$ we have	
	
	$\displaystyle\bigintsss_{U}f_0(p^{-1}\cdot u')du'=\displaystyle\bigintsss_{U} f_0\left(\phi\left(\begin{pmatrix}
		t&0\\0&t
	\end{pmatrix}^{-1}\begin{pmatrix}
		I_n&u\\0&I_n
	\end{pmatrix}\begin{pmatrix}
		t&0\\0&t
	\end{pmatrix}\right)\right)du$\\
	
	$\hspace{-.5cm} =  \displaystyle\bigint_{U'} f_0 \left(\begin{pmatrix}
		&&&0&0&0\\&I_n&&b^{-1}xa-b^{-1}fc^{-1}za&b^{-1}yb+b^{-1}xd-b^{-1}fc^{-1}zd-b^{-1}fc^{-1}wb&0\\&&&c^{-1}za&c^{-1}wb+c^{-1}zd&0\\&&&&&\\&0&&&I_n&\\&&&&&
	\end{pmatrix}\right) dx dy dz dw.$
By a succession of change of variables in the following order, $x\mapsto bxa^{-1}+fc^{-1}z,$ $z\mapsto cza^{-1}$, $w\mapsto cwb^{-1}+cza^{-1}db^{-1}$ and  $y\mapsto byb^{-1}-bxa^{-1}db^{-1}+fwb^{-1}+fza^{-1}db^{-1},$  we get
	\begin{align*}
		\displaystyle\bigintsss_{U'}f_0(p^{-1}\cdot u')du'&
		=\nu(a^{k-n})\nu(b^{r-n})\nu(c^{r-k}) \displaystyle\int_{U} f_0\left(\begin{pmatrix}
			I_n&u\\0&I_n
		\end{pmatrix}\right)du \\
		&=\nu(a^{k-n})\nu(b^{r-n})\nu(c^{r-k}) \displaystyle\int_{U'} f_0(u')du'.\end{align*}
Noting that $\delta_{\Delta P_{k,r-2k,n-r+k}}(p)=\nu(a^{k-n})\nu(b^{r-n})\nu(c^{r-k}) ,$ proves our claim.\end{proof}
\subsubsection{First form of Structure Theorem}
We are ready to present the first form of the main structure theorem for the twisted Jacquet module  $(i_{P_{r,2n-r}}^{G_{2n}}(\rho))_{N,\psi}.$
\begin{theorem}\label{mainpadic}
	Fix an integer $r$ such that $1\leq r <2n.$
	Let $\rho\in \Alg(M_{r,2n-r})$ and put $\pi=i_{P_{r,2n-r}}^{G_{2n}}(\rho).$  Put $\alpha=\max\{0, r-n\}$ and $\beta={\lfloor \frac{r}{2} \rfloor}.$ 	
	Then,  $\pi_{N,\psi}$  has a filtration 
	\begin{equation}\label{filtration1}
		\{0\}\subset V_{\alpha}\subset \dots \subset V_{\beta}=\pi_{N,\psi},
	\end{equation}
such that for $\alpha \leq k \leq \beta $ we have an isomorphism of $\Delta G_n$-modules given by
	\begin{equation}
	V_k/V_{k-1}\cong  \ind_{\Delta P_{k,r-2k,n-r+k}}^{\Delta G_n}\left([\rho_{_{N_k'',\psi_{k}''}}\otimes \delta_{P_{r,2n-r}}^{1/2}]^{w_{k}}\otimes \delta_{\Delta P_{k,r-2k,n-r+k}}^{-1}\right),
	\end{equation}
	where $N_k''$ and $\psi_k''$ are as in \eqref{DefinitionNkk''} and \eqref{Definitionpsikdoubleprime}. 	
\end{theorem}

\begin{proof} Put $\gamma=\min\{r,n\}.$
By Theorem \ref{RestrictiontoSpsi},  $\pi_{|_{S_{\psi}}}$ is glued from the representations
$\ind_{S_{k,l}}^{S_{\psi}}([\rho\otimes \delta_{_{P_{r,2n-r}}}^{1/2}]^{w_{k,l}})$  where $k$ and $l$ vary in the range $\alpha \leq k \leq \gamma$ and  $\alpha \leq l\leq \min\{k,r-k\}.$
If  $k\neq l,$ $(\ind_{S_{k,l}}^{S_{\psi}}([\rho\otimes \delta_{_{P_{r,2n-r}}}^{1/2}]^{w_{k,l}}))_{N,\psi}=0$ by Corollary \ref{knotldoesnotcontribute}. By exactness of the twisted Jacquet functor, $\pi_{N,\psi}$ is glued from the representations $\left(\ind_{S_k}^{S_{\psi}}([\rho\otimes \delta_{P_{r,2n-r}}^{1/2}]^{w_{k}})\right)_{N,\psi}$ where $k$ varies in the range $\alpha\leq k\leq \beta.$ But, to obtain the filtration of $\pi_{N,\psi}$ as promised by the theorem, we need to invoke Theorem \ref{summarydoublecosets} and Theorem \ref{RestrictiontoP}.  By Theorem \ref{RestrictiontoP},  $\pi_{|_P}$ is given by a filtration  (see \eqref{filtrationofrestoP})
\begin{equation}\label{filtrationofrestoPrecall}
	0\subset W_{\alpha}\subset \dots \subset W_{\gamma}=\pi_{|_{P}}, 
	\end{equation}
where  for each $k$ such that $\alpha\leq k \leq \gamma$ one has, $W_{k}/W_{k-1}\cong \ind_{J_{k,\alpha}}^{P}([\rho\otimes \delta_{P_{r,2n-r}}^{1/2}]^{w_{k,\alpha}}).$ Take the twisted Jacquet functor of the spaces in \eqref{filtrationofrestoPrecall} to get 

\begin{equation}\label{TJMfiltrationofrestoP}
0\subset (W_{\alpha})_{N,\psi}\subset \dots \subset (W_{\gamma})_{N,\psi}=\pi_{N,\psi}.
\end{equation}
Fix a $k$ such that $\alpha\leq k \leq \gamma.$ Note that,

\begin{equation}
	(W_{k})_{N,\psi}/(W_{k-1})_{N,\psi} \simeq (\ind_{J_{k,\alpha}}^{P}([\rho\otimes \delta_{P_{r,2n-r}}^{1/2}]^{w_{k,\alpha}}))_{N,\psi}.
\end{equation}

 By Theorem \ref{summarydoublecosets} (1) \&(2) and Mackey theory, it is easy to see that for each fixed $k$, $\ind_{J_{k,\alpha}}^{P}([\rho\otimes \delta_{_{P_{r,2n-r}}}^{1/2}]^{w_{k,\alpha}})_{|_{S_{\psi}}}$ is glued from $\ind_{S_{k,l}}^{S_{\psi}}([\rho\otimes \delta_{P_{r,2n-r}}^{1/2}]^{w_{k,l}})$ where $\alpha\leq l \leq \min\{k,r-k\}.$ Observe that, by Corollary \ref{knotldoesnotcontribute} and Remark \ref{reductiontoequalcase}, the representation $\ind_{J_{k,\alpha}}^{P}([\rho\otimes \delta_{_{P_{r,2n-r}}}^{1/2}]^{w_{k,\alpha}})$ does not contribute to $\pi_{N,\psi}$ if $k>\beta.$ On the other hand, if $k\leq \beta,$ there exists at most one $S_{\psi}$-orbit in $P/J_{k,\alpha}$ corresponding to the double coset representative $w_k:=w_{k,k}$ i.e., when $l=k,$ (recall that $w_{k,k}$ is denoted by $w_k$ as in \eqref{definitionwkk-alternate}) which may contribute to  $\pi_{N,\psi}.$
	
 Thus, $(\ind_{J_{k,\alpha}}^{P}([\rho\otimes \delta_{P_{r,2n-r}}^{1/2}]^{w_{k,\alpha}}))_{N,\psi}=0$ if $k>\beta$ yields $(W_{\beta})_{N,\psi}=(W_{\beta+1})_{N,\psi}=\cdots =(W_{\gamma})_{N,\psi}=\pi_{N,\psi}.$ Also, it follows that if $\alpha \leq k\leq \beta,$ 
\begin{equation}
((\ind_{J_{k,\alpha}}^{P}([\rho\otimes \delta_{P_{r,2n-r}}^{1/2}]^{w_{k,\alpha}}))_{|_{S_{\psi}}})_{N,\psi}\simeq (\ind_{S_k}^{S_{\psi}}([\rho\otimes \delta_{P_{r,2n-r}}^{1/2}]^{w_{k}}))_{N,\psi}.	
\end{equation}
Now, put $V_k=(W_{k})_{N,\psi}$ for $\alpha\leq k \leq \beta.$ In conclusion, $\pi_{N,\psi}$ is given by a filtration of $\Delta G_n$-modules as follows:
\begin{equation}\label{reduction1TJM}
0\subset V_{\alpha}\subset \dots 	\subset V_{\beta}=\pi_{N,\psi},
\end{equation}
where for $\alpha \leq k\leq \beta$ we have  ${V_k}/{V_{k-1}}\cong (\ind_{S_k}^{S_{\psi}}([\rho\otimes \delta_{_{P_{r,2n-r}}}^{1/2}]^{w_k}))_{N,\psi}.$ The statement of the theorem now follows from this isomorphism and an application of  \eqref{mainisomorphism-equal} along with  Lemma \ref{transition} and Lemma \ref{thedeltacomputation}.
 \end{proof}
 
 \begin{remark}\label{contributionandfiber}
 We note that the fiber of the map $ S_{\psi}\backslash G_{2n}/P_{r,2n-r}\to P\backslash G_{2n}/P_{r,2n-r}$ lying over $w_{k,\alpha}$ is precisely $\{w_{k,l}: \alpha \leq l \leq \min\{k,r-k\}\}.$ The strategy of the proof of Theorem \ref{mainpadic} was to observe that in the filtration of $\pi_{|_P},$ for the contribution from the component corresponding to the representative $w_{k,\alpha}$  to $\pi_{N,\psi}$ to be possibly non-zero, there must exist a representative $w_{k,l}$  in its fiber with $l=k.$ But that is possible only if $k\leq \lfloor \frac{r}{2}\rfloor.$ 
 \end{remark}

\subsection{Structure in the non-archimedean case II}\label{componentwise}
\subsubsection{}
Our next aim is to obtain a simplified form of Theorem \ref{mainpadic}, assuming that the inducing data $\rho\in \Alg(M_{r,2n-r})$ is of the form $\rho=\rho_1\otimes \rho_2$ where $\rho_1\in \Alg(G_{r})$ and $\rho_2\in \Alg(G_{2n-r}).$ To this end, we need to introduce few more subgroups which play a role in this description. We put $\alpha=\max\{0,r-n\},$ $\beta=\lfloor \frac{r}{2}\rfloor$ and assume that $k$ is an integer such that $\alpha\leq k\leq \beta.$  

\subsubsection{}
Recall that, $S_k = \Delta P_{k,r-2k, n-r+k}\ltimes N_{k}.$ Put $$S_k'' = w_{k}^{-1}S_kw_{k},  \text{ and } \Delta P_{k,r-2k, n-r+k}''= w_{k}^{-1}\Delta P_{k,r-2k, n-r+k}w_{k}.$$ Then,

\renewcommand{\kbldelim}{(}
\renewcommand{\kbrdelim}{)}
\begin{equation}\label{Sk''description}
S_k''= \left\{\kbordermatrix{
	&k&k&r-2k&&r-2k&n-r+k&n-r+k \\
	k&a&x&y& \VR d&e&z \\
	k&0&a&d& \VR 0&0&e \\
	r-2k & 0 & 0 &b& \VR0&0&f\\
	\hline
r-2k&0&0&0& \VR b&f&u\\
n-r+k&0&0&0& \VR 0&c&v \\
n-r+k &0&0&0&\VR 0 & 0 &c }\right\}\cap G_{2n},
\end{equation}
and 
\renewcommand{\kbldelim}{(}
\renewcommand{\kbrdelim}{)}
\begin{equation}\label{DeltaPk''description}
\Delta P_{k,r-2k, n-r+k}''= \left\{\kbordermatrix{
	&k&k&r-2k&&r-2k&n-r+k&n-r+k \\
	k&a&0&0& \VR d&e&0 \\
	k&0&a&d&\VR 0&0&e \\
	r-2k & 0 & 0 &b& \VR 0&0&f\\
	\hline
	r-2k&0&0&0& \VR b&f&0\\
	n-r+k&0&0&0& \VR 0&c&0 \\
	n-r+k &0&0&0&\VR 0 & 0 &c }\right\}\cap G_{2n}.
\end{equation}
We observe that $S_k'' = \Delta P_{k,r-2k, n-r+k}''\ltimes N_k''.$ 
\subsubsection{}
Let
\renewcommand{\kbldelim}{(}
\renewcommand{\kbrdelim}{)}
\begin{equation}
	S_k'= \left\{\kbordermatrix{
		&k&k&r-2k&&r-2k&n-r+k&n-r+k \\
		k&a&x&y& \VR 0&0&0 \\
		k&0&a&d& \VR 0&0&0 \\
		r-2k & 0 & 0 &b& \VR0&0&0\\
		\hline
		r-2k&0&0&0& \VR b&f&u\\
		n-r+k&0&0&0& \VR 0&c&v \\
		n-r+k &0&0&0&\VR 0 & 0 &c }\right\}\cap G_{2n},
\end{equation}
and 
\renewcommand{\kbldelim}{(}
\renewcommand{\kbrdelim}{)}
\begin{equation}
	M_{\psi_k}'= \left\{\kbordermatrix{
		&k&k&r-2k&&r-2k&n-r+k&n-r+k \\
		k&a&0&0& \VR 0&0&0 \\
		k&0&a&d&\VR 0&0&0 \\
		r-2k & 0 & 0 &b& \VR 0&0&0\\
		\hline
		r-2k&0&0&0& \VR b&f&0\\
		n-r+k&0&0&0& \VR 0&c&0 \\
		n-r+k &0&0&0&\VR 0 & 0 &c }\right\}\cap G_{2n}.
\end{equation}
Note that $S_k'$ and $M_{\psi_k}'$ are subgroups of $M_{r,2n-r}$ and moreover $S_k' = M_{\psi_k'} \ltimes N_k'.$ 	

\subsubsection{The subgroups $N(k,1)$ and $N(k,2)$}
Finally, let $N(k,1)$ and $N(k,2)$ be subgroups of $G_{r}$ and $G_{2n-r}$ respectively defined by 
\begin{equation}\label{N1definition}N(k,1)=\left\{\left[\begin{array}{ccc}
I_k&x&y \\
0&I_{k}&0 \\
0&0&I_{r-2k}	
\end{array}
\right]: x\in \mathcal{M}_{k}, y\in \mathcal{M}_{k,r-2k} \right\}\end{equation} and 

\begin{equation}\label{N2definition}
N(k,2)=\left\{\left[\begin{array}{ccc}
I_{r-2k}&0&w \\
0&I_{n-r+k}& z\\
0&0&I_{n-r+k}	
\end{array}
\right]: w\in \mathcal{M}_{r-2k,n-r+k}, z\in \mathcal{M}_{n-r+k}\right\}.\end{equation}

\subsubsection{The characters $\psi_{k,1}$ and $\psi_{k,2}$}
Let $\psi_{k,1}$ be the character of $N(k,1)$ defined by
\begin{equation}\label{Defnpsik1}
 \psi_{k,1}(n_1) =\psi_0(\tr(x)).
 \end{equation}
 
Similarly, let $\psi_{k,2}$ be the character of $N(k,2)$ defined by 
 \begin{equation}\label{Defnpsik2}
 \psi_{k,2}(n_2) =\psi_0(\tr(z)).
 \end{equation}
 Then, we have $N_k'=N(k,1)\times N(k,2) \mbox{ and } \psi'_k(\diag(n_1,n_2))=\psi_{k,1}(n_1)\psi_{k,2}(n_2),$ i.e., $\psi'_k=\psi_{k,1}\times \psi_{k,2}.$  
 \begin{remark}
We note that the subgroup $N(k,1)$  is indeed the subgroup $N_{k,r-k},$ the unipotent radical of the parabolic subgroup $P_{k,r-k}$ of $G_r.$ However, we have preferred to write it in steps of three blocks to highlight the square block   $x\in \mathcal{M}_k$ which contributes to the character $\psi_{k,1}.$ A similar remark applies to the subgroup $N(k,2)$ as well, which is the unipotent radical $N_{n-k,n-r+k}$ of  the parabolic subgroup $P_{n-k,n-r+k}$ of $G_{2n-r}.$
 \end{remark}
 
 \subsubsection{} 
 We need the following simple lemma.
\begin{lemma}\label{componentwiseaction}
Suppose $(\rho,V)\in \Alg(P_{r,2n-r})$ acts trivially on $N_{r,2n-r}.$ Let $N_{k}'', \psi_{k}'', N_{k}', \psi_{k}'$ be defined as in Section \ref{Non-zero criteria}. The following statements hold:
\begin{enumerate}
	\item The subgroup $M_{\psi_{k}'} $ of $M_{r,2n-r}$ normalizes $N_k'$ and the character $\psi_k'.$
	\item $V_{N_{k}'',\psi_k''}=  V_{N_{k}',\psi_k'}$ as $\C$-vector spaces.
	\item For $v\in V$ and $\begin{pmatrix}
		A' &C'\\ 0 & B'
	\end{pmatrix}\in \Delta P_{k,r-2k, n-r+k}'' ,$  

	\begin{equation}
	\rho_{N_{k}'',\psi_k''} \begin{pmatrix}
		A' & C'\\ 0 & B'
	\end{pmatrix} (v + V(N_{k}'',\psi_k''))=\rho_{N_{k}',\psi_k'} \begin{pmatrix}
		A' &0\\ 0 & B'
	\end{pmatrix}  (v + V(N_{k}',\psi_k')).
\end{equation}
\end{enumerate}  
\end{lemma}
\begin{proof}
The statement (1) is easily verified.  As $V(N_{k}'',\psi_k'')= V(N_{k}',\psi_k')$ as $\C$-vector spaces, (2) follows. Next, we compare the action of  $\rho_{N_{k}'',\psi_k''} \in \Alg (\Delta P_{k, r-2k, n-r+k})$ with that of $\rho_{N_{k}',\psi_k'} \in \Alg (M_{\psi_k'}).$  For $v \in V,$ noting that the representation $\rho$ acts trivially on $N_{r,2n-r},$ we get
\begin{align*}
	\rho_{N_{k}'',\psi_k''} \begin{pmatrix}
		A' & C'\\ 0 & B'
	\end{pmatrix} (v + V(N_{k}'',\psi_k''))
	&= \rho \begin{pmatrix}
		A' & C'\\ 0 & B'
	\end{pmatrix}
	 (v) + V(N_{k}'',\psi_k'')\\
	& = \rho \begin{pmatrix}
		A' &0\\ 0 & B'
	\end{pmatrix} (v) + V(N_{k}',\psi_k')\\
	& = \rho_{N_{k}',\psi_k'} \begin{pmatrix}
		A' &0\\ 0 & B'
	\end{pmatrix} (v + V(N_{k}',\psi_k')).\qedhere
\end{align*}
\end{proof}

\subsubsection{}
We are  ready to state the structure theorem in its final form.
\begin{theorem}\label{maincomponentwise}
Fix an integer $r$ such that $1\leq r <2n.$
For $\rho_1 \in \Alg(G_{r})$ and  $\rho_2 \in \Alg(G_{2n-r}) $ put $\rho:= \rho_1 \otimes \rho_2 $  and  $\pi=i_{P_{r,2n-r}}^{G_{2n}}(\rho).$ Let  $\psi$ be the character $N$ defined by \eqref{definitionofpsi}.  Put $\alpha=\max\{0, r-n\}$ and $\beta={\lfloor \frac{r}{2} \rfloor}.$ 	
Then,  $\pi_{N,\psi}$  has a filtration 

\begin{equation}\label{filtrationmain2}
\{0\}\subset V_{\alpha}\subset \dots \subset V_{\beta}=\pi_{N,\psi},
\end{equation}
such that for $\alpha \leq k \leq \beta $ we have an isomorphism of $\Delta G_n$-modules given by
\begin{equation}\label{componentwisegluepieces}
	V_k/V_{k-1}\cong  \ind_{\Delta P_{k,r-2k,n-r+k}}^{\Delta G_n}\left(\tau_k \otimes[ \delta_{P_{r,2n-r}}^{1/2}]^{w_{k}}\otimes \delta_{\Delta  P_{k,r-2k,n-r+k}}^{-1}\right)
\end{equation}
where 
\renewcommand{\kbldelim}{(}
\renewcommand{\kbrdelim}{)}
\begin{equation}\label{definitionoftau}
\tau_k\left(\kbordermatrix{
	&k&r-2k&n-r+k &  & k&r-2k&n-r+k\\
	k&a&d&e & \VR 0&0&0\\
	r-2k&0&b&f& \VR 0&0&0\\
	n-r+k&0&0&c& \VR 0&0&0\\
	\hline
	k&0&0&0& \VR a&d&e\\
	r-2k&0&0&0& \VR 0&b&f\\
	n-r+k&0&0&0& \VR 0&0&c
}\right) 
\end{equation}
\begin{equation*}
= (\rho_1)_{N(k,1),\psi_{k,1}}\left(\kbordermatrix{
&k&k&r-2k \\
k&a&0&0 \\
k&0&a&d \\
r-2k & 0 & 0 &b }\right) \otimes (\rho_2)_{N(k,2),\psi_{k,2}}\left(\kbordermatrix{
&r-2k&n-r+k&n-r+k \\
r-2k&b&f&0 \\
n-r+k&0&c&0 \\
n-r+k & 0 & 0 &c }\right).
\end{equation*}
\end{theorem}

\begin{proof}
By Theorem \ref{mainpadic}, 
\begin{equation*}
V_k/V_{k-1}\cong  \ind_{\Delta P_{k,r-2k,n-r+k}}^{\Delta G_n}\left([\rho_{_{N_{k}'',\psi_{k}''}}\otimes \delta_{P_{r,2n-r}}^{1/2}]^{w_{k}}\otimes \delta_{\Delta P_{k,r-2k,n-r+k}}^{-1}\right).
\end{equation*}
It is trivial that  $[\rho_{_{N_{k}'',\psi_{k}''}}\otimes \delta_{P_{r,2n-r}}^{1/2}]^{w_{k}}=[\rho_{_{N_{k}'',\psi_{k}''}}]^{w_{k}}\otimes [\delta_{P_{r,2n-r}}^{1/2}]^{w_{k}}.$ Write an element $Z\in P_{k,r-2k,n-r+k}$ as 
\begin{equation*}
Z=\kbordermatrix{
&k&r-2k&n-r+k \\
k&a&d&e \\
r-2k&0&b&f \\
n-r+k & 0 & 0 &c}.
\end{equation*}
Then, by \eqref{DeltaPk''description}
\begin{equation*}[\rho_{_{N_k'',\psi_{k}''}}]^{w_{k}}\begin{pmatrix}
		Z & 0\\ 0 &Z
	\end{pmatrix} (v + V(N_k'',\psi_k'')) = \rho_{N_k'',\psi_k''} \begin{pmatrix}
		A' & C'\\ 0 & B'
	\end{pmatrix} (v + V(N_k'',\psi_k'')),
\end{equation*}
where 
$A'=\begin{pmatrix}
	a & 0 &0\\
	0& a & d\\
	0 & 0 & b
	\end{pmatrix}\in  G_{r}, 
B'=\begin{pmatrix}
	b & f &0\\
	0& c & 0\\
	0 & 0 & c
	\end{pmatrix}\in  G_{2n-r}$ and 
$C'=\begin{pmatrix}
	d & e &0\\
	0& 0 & e\\
	0 & 0 & f
	\end{pmatrix}\in \mathcal{M}_{r,2n-r}.$ 
Appealing to Lemma \ref{componentwiseaction}, one has
\begin{equation*}
\rho_{N_k'',\psi_k''} 
\begin{pmatrix}
	A' & C'\\ 0 & B'
\end{pmatrix} (v + V(N_k'',\psi_k'')) = 
\rho_{N_k',\psi_k'}
 \left(\begin{pmatrix}
A' &0\\ 0 & B'
\end{pmatrix}\right) (v + V(N_k',\psi_k')).
\end{equation*}
Applying Lemma \ref{TJM split in tensor product}, 
we obtain 
$$\rho_{N_k',\psi_k'}= (\rho_1 \otimes \rho_2)_{N_k',\psi_k'} = (\rho_1)_{N(k,1),\psi_{k,1}} \otimes (\rho_2)_{N(k,2),\psi_{k,2}}$$
as ${\rm Norm}_{G_r}(N(k,1),\psi_{k,1})\times {\rm Norm}_{G_{2n-r}}(N(k,2),\psi_{k,2})$-modules  and hence as $M_{\psi_{k}'}$-modules  by Lemma \ref{componentwiseaction} (1).
Therefore, 
\begin{align*}
[\rho_{_{N_k'',\psi_{k}''}}]^{w_{k}}\begin{pmatrix}
Z & 0\\ 0 &Z
\end{pmatrix}& = 
\rho_{N_k',\psi_k'} \begin{pmatrix}
A' &0\\ 0 & B'
\end{pmatrix}\\
& = 
(\rho_1)_{N(k,1),\psi_{k,1}}(A')\otimes (\rho_2)_{N(k,2),\psi_{k,2}}(B')\\
& 
= \tau_k\begin{pmatrix}
Z & 0\\ 0 &Z
\end{pmatrix}.
\end{align*}	
This completes the proof of the theorem.	
\end{proof}	

\subsubsection{} One immediately has the following result which is a `component wise' version of Theorem \ref{nonzeroTJM}. For the following corollary, we shall keep the notations as in Theorem \ref{maincomponentwise}.
\begin{corollary}\label{nonzeroTJMcomponentwise}
Let $\rho_1\in \Alg(G_r), \rho_2\in \Alg(G_{2n-r})$ and put $\pi=i_{P_{r,2n-r}}^{G_{2n}}(\rho_1\otimes \rho_2).$  Then,  $\pi_{N,\psi}$ is non-zero if and only if there exists an integer $k$ such that $\alpha \leq k \leq \beta$ for which both the twisted Jacquet modules $(\rho_1)_{N(k,1),\psi_{k,1}}$ and $(\rho_2)_{N(k,2),\psi_{k,2}}$ are simultaneously non-zero. 
\end{corollary}

\subsubsection{} 
We shall also state below a normalized version  of Theorem \ref{maincomponentwise}, (i.e., where all the inductions and twisted Jacquet functors are normalized) which may be more useful in applications.
\begin{corollary}\label{normalisedcomponentwisecor}
	Keeping notations and assumptions as in Theorem \ref{maincomponentwise},  $\pi_{{N,\psi}}$ has a filtration 
	
	\[\{0\}\subset V_{\alpha}\subset \dots \subset V_{\beta}=\pi{_{N,\psi}}\]
	where for $\alpha \leq k \leq \beta,$ we have an isomorphism of $\Delta G_n$-modules given by
	\begin{equation}\label{componentwisegluepiecesnormalized}
	V_k/V_{k-1}\cong  i_{\Delta P_{k,r-2k,n-r+k}}^{\Delta G_n}\left(\tau_k' \otimes[ \delta_{P_{r,2n-r}}^{1/2}]^{w_{k}}\otimes \delta_{\Delta P_{k,r-2k,n-r+k}}^{-3/2}\right)
	\end{equation}
	where 
	\renewcommand{\kbldelim}{(}
	\renewcommand{\kbrdelim}{)}
	\begin{equation}\label{definitionoftau'}
	\tau_k'\left(\kbordermatrix{
		&k&r-2k&n-r+k &  & k&r-2k&n-r+k\\
		k&a&d&e & \VR 0&0&0\\
		r-2k&0&b&f& \VR 0&0&0\\
		n-r+k&0&0&c& \VR 0&0&0\\
		\hline
		k&0&0&0& \VR a&d&e\\
		r-2k&0&0&0& \VR 0&b&f\\
		n-r+k&0&0&0& \VR 0&0&c
	}\right) 
	\end{equation}
	\begin{equation*}
	= \nu^{\frac{r-2k}{2}}(a)\nu^{-\frac{k}{2}}(b) r_{N(k,1),\psi_{k,1}}(\rho_1)\left(\begin{pmatrix}
	a&0&0\\0&a&d\\0&0&b
	\end{pmatrix}\right) \otimes \nu^{\frac{n-r+k}{2}}(b)\nu^{\frac{2k-r}{2}}(c) r_{N(k,2),\psi_{k,2}}(\rho_2)\left(\begin{pmatrix}
	b&f&0\\0&c&0\\0&0&c
	\end{pmatrix}\right).
	\end{equation*} 
\end{corollary}
\begin{proof}
	By Theorem \ref{maincomponentwise}, one has
	\begin{equation}
	V_k/V_{k-1}\cong  \ind_{\Delta P_{k,r-2k,n-r+k}}^{\Delta G_n}\left(\tau_k \otimes[ \delta_{P_{r,2n-r}}^{1/2}]^{w_{k}}\otimes \delta_{\Delta P_{k,r-2k,n-r+k}}^{-1}\right)
	\end{equation} where $\tau_k$ is as in \eqref{definitionoftau}.
	Since, for any $\ell$-group $G,$ a closed subgroup $H$ of $G$ and $\eta\in\Alg(H),$ $i_H^G(\eta)=\ind_H^G(\delta_H^{\frac{1}{2}} \delta_G^{-\frac{1}{2}}\eta)$, we get
	\begin{equation}\label{normalisedcomponentwise}
	V_k/V_{k-1}\cong  i_{\Delta P_{k,r-2k,n-r+k}}^{\Delta G_n}\left(\tau_k \otimes[ \delta_{P_{r,2n-r}}^{1/2}]^{w_{k}}\otimes \delta_{\Delta P_{k,r-2k,n-r+k}}^{-3/2}\right).
	\end{equation}
	Note that $(\rho_i)_{N(k,i),\psi_{k,i}}\simeq {\rm mod}_{N(k,i)}^{\frac{1}{2}} \otimes  r_{N(k,i),\psi_{k,i}}(\rho_i)$ for $i\in \{1,2\}.$ To establish the corollary, we need to compute the action of the modular characters ${\rm mod}_{N(k,1)}$ and ${\rm mod}_{N(k,2)}$ on the subgroups $M_{\psi_k}'\cap G_r$ and $M_{\psi_k}'\cap G_{2n-r}$ respectively. 
	Let $f_0$ denote a locally constant compactly supported function on $N(k,1).$ 	For 
	$A=\kbordermatrix{
		&k&k&r-2k \\
		k&a&0&0 \\
		k&0&a&d \\
		r-2k & 0 & 0 &b}\in M_{\psi_k'}\cap G_{r}$ we have 
		
	\begin{align*}
	\bigintsss_{N(k,1)}f_0(A^{-1}n_1A)dn_1
	 & = \bigintsss_{N(k,1)}\begin{pmatrix}
	I_k&a^{-1}xa&a^{-1}yb+a^{-1}xd\\0&I_k&0\\0&0&I_{r-2k}
	\end{pmatrix}dxdy.
	\end{align*}
Applying a succession of change of variables in the order $x\to axa^{-1}, y\to ayb^{-1}-xdb^{-1},$ we get 
	\begin{align*}
	\bigintsss_{N(k,1)}f_0(A^{-1}n_1A)dn_1&= \nu(a)^{r-2k}\nu^{-k}(b)\bigintsss_{N(k,1)}\begin{pmatrix}
	I_k&x&y\\0&I_k&0\\0&0&I_{r-2k}
	\end{pmatrix}dxdy\\
	\end{align*}
	Therefore, ${\rm mod}_{N(k,1)} \left(A\right)= \nu^{2k-r}(a)\nu^{-k}(b).$ By a similar calculation,
\begin{equation*}{\rm mod}_{N(k,2)}\left(\begin{pmatrix}
	b&f&0\\0&c&0\\0&0&c
	\end{pmatrix}\right)= \nu^{n-r+k}(b)\nu^{2k-r}(c).
\end{equation*}
	The corollary follows.
\end{proof}

\subsection{Structure of the twisted Jacquet module in the finite field case}\label{structure-finite}
\subsubsection{}
In this section, we shall prove our main structure theorem for a parabolically induced representation of $GL_{2n}(\F_q).$ To begin with, we note some standard facts on representations of finite groups. Let $G$ be a finite group and $(\pi, V)$ be a finite dimensional complex representation of $G.$ For a subgroup $H$ of $G$ and a character $\chi$ of $H,$  we have
\begin{equation*}
V(H,\chi) := {\rm Span}\{\pi(h)v-\chi(h)v: v\in V, h\in H\}= \{v \in V : \sum_{h \in H} \chi^{-1}(h)\pi(h)v = 0\}.
\end{equation*}
Let
\begin{equation*}
V^{H,\chi} := \{v \in V : \pi(h)v = \chi(h)v\: \text{for all} \: h \in H \}.
\end{equation*}
\subsubsection{}
Let $S$ be a subgroup of $G$ which normalizes both $H$ and $\chi.$ Then, $V^{H,\chi}$ is an $S$-invariant subspace of $V$ and becomes a representation of $S$ which we shall denote by  $\pi^{H,\chi}.$  Also, denote the restriction of $\pi$ to $V(H,\chi)$ by $\pi(H,\chi).$ We then have $\pi= \pi(H,\chi) \bigoplus \pi^{H,\chi}$ and consequently, obtain an isomorphism of $S$-modules:
\begin{equation}\label{equivalence-finite}
	\pi_{H,\chi}\simeq\pi^{H,\chi}.
\end{equation}
	
\subsubsection{}
One of the main ingredients in the proof of Theorem \ref{mainpadic} was Proposition \ref{TJM-reduction} which was proved using Theorem \ref{Renard-general}. Here, we prove the analogue of Proposition \ref{TJM-reduction} in the finite group case directly giving a standard argument.

\begin{proposition}\label{TJM-reduction-analoguefinite}
	Let $G_0$ be a finite group. Let $G,U$ be subgroups of $G_0$  and let $\theta:U\to \C^{\times}$ be a character such that $G$ normalizes both $U$ and $\theta.$ Let $H$ be a  subgroup of $G$ and $U_0$ a subgroup of $U.$ Let $\theta_0$ denote the restriction of $\theta$ to $U_0.$ Assume further that $H$ normalizes $U_0$ and $\theta_0.$ For any representation $(\rho, V)$ of the group $HU_0,$  we have an isomorphism of $G$-modules: 
\begin{equation}\label{isomorphism-finite}
	(\Ind_{HU_0}^{GU}(\rho))_{U,\theta} \simeq \Ind_{H}^{G}(\rho_{U_0,\theta_0})
	\end{equation}
\end{proposition}
\begin{proof}
In view of the isomorphism 	\eqref{equivalence-finite}, to prove \eqref{isomorphism-finite}, it suffices to show that 	$(\Ind_{HU_0}^{GU}(\rho))^{U,\theta}$ is equivalent to $ \Ind_{H}^{G}(\rho^{U_0,\theta_0})$ as $G$-modules.	To this end, let $\phi:(\Ind_{HU_0}^{GU}(\rho))^{U,\theta} \to \Ind_{H}^{G}(\rho^{U_0,\theta_0})$  denote the map obtained by restricting a function $f \in (\Ind_{HU_0}^{GU}(\rho))^{U,\theta}$ to $G.$ For the purposes of this proof, we shall denote the action of $\Ind_{HU_0}^{GU}(\rho)$ by $\sigma.$ We will show that $\phi$ is well defined and is a $G$-equivariant linear bijection.
	
To show that $\phi$ is well defined, we show first that $\phi(f)(g) \in \rho^{U_0,\theta_0}$ for every $g \in G.$  For $u_0 \in U_0,$ $\rho(u_0)(\phi(f)(g))= \rho(u_0)f(g)=f(u_0g).$ Since $G$ normalizes $U,$ let $u_1\in U$ be such that $u_0g=gu_1.$ Then, $f(u_0g)=f(gu_1)=(\sigma(u_1)f)(g)=(\theta(u_1)f)(g)=\theta(u_1)f(g)=\theta(u_0)f(g).$ Thus, $\phi(f)(g) \in \rho^{U_0,\theta_0}.$ Next, take $h\in H$ and $g \in G.$ Note that $\phi(f)(hg)=f(hg)=\rho(h)(f(g)).$ So, $\phi(f)\in \Ind_{H}^{G}(\rho^{U_0,\theta_0}).$  It is obvious that $\phi$ is linear and straightforward to check that $\phi$ is a $G$-equivariant map. 
	
To show injectivity of $\phi,$ suppose that $\phi(f)=0,$ i.e., $f(g)=0$ for every $g\in G.$ Take $u \in U,g \in G.$ Since $f \in  (\Ind_{HU_0}^{GU}(\rho))^{U,\theta},$ we see that $f(gu)=(\sigma(u)f)(g)=\theta(u)f(g)=0,$ proving $\phi$ is injective. 
	
	Let $f'\in \Ind_{H}^{G}(\rho^{U_0,\theta_0}).$ Define $f:GU \to V$ by $f(gu)=\theta(u)f'(g).$ Then, for $h\in H, u_0\in U_0,g \in G, u\in U$ we obtain $f(hu_0gu)=f(hgu_0'u)$ where $u_0'$ is such that $u_0g=gu_0'.$ Then, $\theta(u_0)=\theta(u_0')$ as $G$ normalizes $\theta.$ In view of this,
	\begin{equation*}f(hgu_0'u)=\theta(u_0'u)f'(hg)=\theta(u_0'u)\rho(h)(f'(g))=\theta(u_0)\theta(u)\rho(h)(f'(g)).
	\end{equation*} 

 Now, $\theta(u_0)\theta(u)\rho(h)(f'(g))=\theta(u)\rho(h)(\theta(u_0)f'(g))=\theta(u)\rho(h)(\rho(u_0)(f'(g)))=\rho(hu_0)f(gu).$ Also, if $u\in U,u'\in U$ and $g\in G,$ we have  $[\sigma(u)f](gu')=f(gu'u)=\theta(u'u)f'(g)=\theta(u)\theta(u')f'(g)=\theta(u)f(gu').$ We have shown that $f\in \Ind_{HU_0}^{GU}(\rho))^{U,\theta}$ and $\phi(f)=f'$ proving $\phi$ is surjective as well, completing the proof of the theorem. 
\end{proof}

\begin{remark}\label{transitiontofinite}
With Proposition \ref{TJM-reduction-analoguefinite} in place, it is straightforward to deduce the analogues of Corollary \ref{knotldoesnotcontribute}, Theorem \ref{nonzeroTJM} and Theorem \ref{mainpadic} for $GL_{2n}(\F_q).$ We need to keep in mind that all modular characters are trivial in this case and instead of a filtration as in Theorem \ref{mainpadic}, one has a direct sum in the case of $GL_{2n}(\F_q).$  Note also that to prove the analogue of Theorem \ref{mainpadic}, we need only invoke the finite field analogue of Theorem \ref{RestrictiontoSpsi}.
\end{remark}

\subsubsection{Structure theorem in the finite field case}
We state the main result in the finite field case which is the analogue of Theorem \ref{maincomponentwise}. 
Keeping Remark \ref{transitiontofinite} in mind, the proof of Theorem \ref{mainfinite} which we omit, follows by the same arguments as in the proof of Theorem \ref{maincomponentwise}.  In the next theorem, $G_n=GL_n(F)$ where $F$ is a finite field with $q$ elements.

\begin{theorem}\label{mainfinite}
	Fix an integer $r$ such that $1\leq r < 2n.$
	Let $\rho_1$  and $\rho_2$ be finite dimensional representations of $G_r$ and $ G_{2n-r}$ respectively. Put $\pi=\Ind_{P_{r,2n-r}}^{G_{2n}}(\rho_1 \otimes \rho_2).$ Fix a non trivial additive character $\psi_0$ of $F.$ Let $\psi$ be the character of $N$ defined by \eqref{definitionofpsi}. Let $N(k,1)$ and $N(k,2)$ be the subgroups of $G_{r}$ and $G_{2n-r}$ respectively defined by \eqref{N1definition} and \eqref{N2definition}. Also, let $\psi_{k,1}$ and $\psi_{k,2}$ be the characters of $N(k,1)$ and $N(k,2)$ defined by \eqref{Defnpsik1} and \eqref{Defnpsik2}.
 Then,
	\begin{equation*}
	\pi_{N,\psi}=\displaystyle\bigoplus_{\max\{0,r-n\}\leq k \leq {\lfloor \frac{r}{2} \rfloor}} \Ind_{\Delta P_{k,r-2k,n-r+k}}^{\Delta G_n}(\tau_k)
	\end{equation*}
\end{theorem}	
where $\tau_k\left( \begin{pmatrix}
a& d & e & 0 & 0 & 0\\
0 & b & f & 0 & 0 & 0\\
0 & 0 & c & 0 & 0 & 0\\
0 & 0 & 0 & a & d & e\\
0 & 0 & 0 & 0 & b & f\\
0 & 0 & 0 & 0 & 0 & c
\end{pmatrix} \right) = (\rho_1)_{N(k,1),\psi_{k,1} }  \begin{pmatrix}
a & 0 & 0\\
0 & a & d\\
0 & 0 & b
\end{pmatrix} \otimes (\rho_2)_{N(k,2),\psi_{k,2} }  \begin{pmatrix}
b & f & 0\\
0 & c & 0\\
0 & 0 & c
\end{pmatrix}.$

\section{Some special cases and applications}\label{Applications}
The aim of this section is to consider a few special cases of the computation of the twisted Jacquet module and illustrate a few applications of our results. Throughout the rest of the article, we fix $F$ to be a non-archimedean local field. If $\rho_i \in \Alg(G_{n_i})$ for $i\in \{1,\dots,s\},$  and $n=\sum n_i$ then $\rho_1\times\dots\times\rho_s$ will denote the parabolically induced representation $i_{P_{n_1,\dots,n_s}}^{G_n}(\rho_1\otimes\dots\otimes\rho_s).$ In Section \ref{smooth-v-char},  we compute $\pi_{N,\psi}$ where $\pi$ is a representation of $G_{2n}$ of the form $\chi\times \rho$ or $\rho\times \chi$ where $\rho$ is any smooth representation and $\chi$ is any character. Section \ref{nonvanishing} discusses certain instances where the twisted Jacquet module of a principal series representation is  non-zero.  We  determine the structure of twisted Jacquet modules in Section \ref{mainapplications} of certain irreducible representations of $G_{2n}$ which are subquotients of certain principal series representations. As a consequence, we show the existence of Shalika model for certain non-generic irreducible representation representation of $G_{2n}.$ In Section \ref{Steinberg-v-char}, we present the structure of twisted Jacquet modules of subquotients of $St_r.\chi\times \mu,$ where $\mu$ and $\chi$ are characters. In Sections \ref{example1} and \ref{example2}, we consider two particular reducible principal series representations of $G_4$ and determine precisely for which subquotients of these reducible principal series the twisted Jacquet module is non-vanishing.

\subsection{Product of a character with a smooth representation}\label{smooth-v-char}
\subsubsection{}
In this subsection, we shall consider the parabolic induction of a character with a smooth representation. The structure of the twisted Jacquet module of such a representation is relatively simple. In particular, if $\chi$ and $\mu$ are one characters of $G_r$ and $G_{2n-r},$ then we show that $(\chi\times\mu)_{N,\psi}\neq 0$ if and only if $r=n$ and in which case $(\chi\times\mu)_{N,\psi}=\chi\otimes \mu.$
\begin{proposition}\label{charactervssmooth}
Let $r$ be such that  $1\leq r <2n.$ Let $\chi$ be a character of $G_{r},$ $\rho\in \Alg(G_{2n-r})$ and put $\pi=i_{P_{r,2n-r}}^{G_{2n}}(\chi\otimes\rho).$ 
\begin{enumerate}
	\item If $r>n,$ then $\pi_{N,\psi}=0,$  
	\item  If $r=n,$ $\pi_{N,\psi}=\chi\otimes \rho$ as $\Delta G_n$-modules, and
	\item   If $r< n,$ then $\pi_{N,\psi}\simeq i_{\Delta P_{r,n-r}}^{\Delta G_n}(\tau \otimes[ \delta_{P_{r,2n-r}}^{1/2}]^{w_{0}}\otimes \delta_{P_{r,n-r}}^{-3/2})$ where $\tau$ acts on the group $\Delta P_{r,2n-r}$ as follows:
	\begin{equation}
		\tau\left(\left[\begin{array}{cc}
			\begin{pmatrix}
				b & f\\
				0 & c
			\end{pmatrix}	&  \\
			& \begin{pmatrix}
				b & f\\
				0 & c
			\end{pmatrix}
		\end{array}\right] \right)\simeq 
		\chi(b) \otimes \nu(b)^{\frac{n-r}{2}}\nu(c)^{-\frac{r}{2}}r_{N_2,\psi_2}(\rho_2)\left( \begin{pmatrix}
			b & f & 0\\
			0 & c & 0\\
			0 & 0 & c\end{pmatrix}\right),
	\end{equation}
	where  $N_2=\left\{\left[\begin{array}{ccc}
		I_{r}&0&w \\
		0&I_{n-r}& z\\
		0&0&I_{n-r}	
	\end{array}
	\right]\right\}$ and $\psi_2(n)=\psi_0(tr(z))$ for $n\in N_2.$ 
\end{enumerate}
\end{proposition}
\begin{proof}
Recall that $\alpha=\max\{0,r-n\}$ and $\beta=\lfloor \frac{r}{2}\rfloor.$ By Theorem \ref{maincomponentwise}, $\pi_{N,\psi}$ is glued from the representations $V_k/V_{k-1}$ (see \eqref{componentwisegluepiecesnormalized} and \eqref{definitionoftau'}) where $\alpha\leq k \leq \beta.$ We note that for any $k>0,$ the subgroup  $N(k,1)$ and the character $\psi_{k,1}$ in Corollary \ref{normalisedcomponentwisecor} are non-trivial whereas the character $\chi$ acts trivially on $N(k,1).$ Thus, for any $k>0,$ $r_{N(k,1), \psi_{k,1}}(\chi)=0.$ If $r>n,$ we have $\alpha=r-n>0$ and therefore, the range of values that $k$ assumes is always positive, forcing $\pi_{N,\psi}=0.$ This proves (i).

On the other hand, if $r\leq n,$ we have $\alpha=0.$  By the argument noted above, the only $V_{k}/V_{k-1}$ which can be non-zero is $V_{0}$ i.e., when the subgroup $N(k,1)$ is trivial. The statements (ii) and (iii) of the corollary follow by putting $k=0$ in \eqref{componentwisegluepiecesnormalized} and \eqref{definitionoftau'}. \end{proof}

\subsubsection{}
We next obtain a proposition analogous to the previous one by interchanging the roles of $\rho$ and $\chi.$ We omit the proof.
\begin{proposition}\label{smoothvscharacter}
	Fix an $r$ such that $1\leq r <2n.$ Let $\rho\in \Alg(G_{r})$ and $\chi$ be a character of $G_{2n-r}.$  Put $\pi=i_{P_{r,2n-r}}^{G_{2n}}(\rho\otimes\chi).$  Then,
	\begin{enumerate}
		\item
		if $r<n, \pi_{N,\psi}=0,$ 
		\item  if $r=n,$ $\pi_{N,\psi}=\rho\otimes \chi$ as $\Delta G_n$-modules, and
		\item  if $r>n,$ then $\pi_{N,\psi}\simeq i_{\Delta P_{r-n,2n-r}}^{\Delta G_n}(\tau \otimes[ \delta_{P_{r,2n-r}}^{1/2}]^{w_{r-n}}\otimes \delta_{P_{r-n,2n-r}}^{-3/2})$ where $\tau$ acts on the group $\Delta P_{r-n,2n-r}$ as follows:
		
		\begin{equation}
			\tau\left(\left[\begin{array}{cc}
				\begin{pmatrix}
					a & d\\
					0 & b
				\end{pmatrix}	&  \\
				& \begin{pmatrix}
					a & d\\
					0 & b
				\end{pmatrix}
			\end{array}\right] \right)\simeq 
			\nu(a)^{\frac{2n-r}{2}}\nu(b)^{\frac{n-r}{2}}r_{N_1,\psi_1}(\rho_1)\left( \begin{pmatrix}
				a & 0 & 0\\
				0 & a & d\\
				0 & 0 & b\end{pmatrix}\right)\otimes \chi(b),
		\end{equation}
		where $N_1=\left\{\left[\begin{array}{ccc}
			I_{r-n}&x&y \\
			0&I_{r-n}& 0\\
			0&0&I_{2n-r}	
		\end{array}
		\right]\right\}$ and $\psi_1(n)=\psi_0(tr(x))$ for an element $n\in N_1.$ 
	\end{enumerate}	
\end{proposition}

\subsubsection{}
As a simple consequence of the two previous propositions, we single out the following special case.
\begin{theorem}\label{productoftwocharcaters}
Let $\chi$ and $\mu$ be characters of $G_{r}$ and $G_{2n-r}$ respectively. Then, 
\begin{equation}\label{TJMproductofcharacters}
(\chi\times \mu)_{N,\psi}\simeq \begin{cases}
0 & \mbox{ if } r\neq n\\
\chi\otimes \mu & \mbox{ if } r=n
\end{cases}
\end{equation}
as $\Delta G_n$-modules.	
\end{theorem}
\begin{proof}
By Propositions \ref{charactervssmooth} and \ref{smoothvscharacter}, we have $(\chi\times \mu)_{N,\psi}=0$ if $r>n$ and $r<n$ respectively. Hence, $(\chi\times \mu)_{N,\psi}$ is non-zero only when $r=n,$ in which case, the statement of the theorem follows from Proposition \ref{charactervssmooth} (or Proposition \ref{smoothvscharacter}).
\end{proof}

\begin{remark}
The analogues of Propositions \ref{charactervssmooth}, \ref{smoothvscharacter}  hold if the respective groups are considered over $\F_q.$ Thus, Theorem \ref{productoftwocharcaters}  holds over $\F_q.$
\end{remark}

\subsection{Certain Non-vanishing Twisted Jacquet Modules}\label{nonvanishing}
\subsubsection{}
In this subsection, we shall give examples of certain classes of principal series representations with a non-zero twisted Jacquet module. To this end, let us note a generality. Suppose $G$ is an $\ell$-group and $U_1, U_2$ are closed subgroups of $G$ such that they are a union of their compact open subgroups. Assume further that $U_1\subset U_2,$ $\theta_1,\theta_2$ are characters of $U_1$ and $U_2$ such that ${\theta_2}_{|U_1}=\theta_1.$ Then, for any $\pi\in \Alg(G)$ we have 
\begin{equation}\label{subTJM}
	\Hom_{U_2}(\pi, \theta_2)\subset \Hom_{U_1}(\pi,\theta_1),
\end{equation}
and therefore, $\pi_{U_1,\theta_1}$ is non-zero whenever $\pi_{U_2,\theta_2}$ is non-zero.  

\subsubsection{}
We first have the following simple result.
\begin{proposition}\label{equalcasenonvanishing}
	Suppose $\rho_1,\rho_2\in \Alg(G_{n}).$ Then, $(\rho_1\times \rho_2)_{N,\psi}$ is non-zero. 
\end{proposition}
\begin{proof}
	We put 	$r=n$ in Theorem \ref{maincomponentwise} to get $\alpha=0.$ The subgroups $N(0,1)$ and $N(0,2)$ are both $\{I_n\} $ and hence $V_0=\rho_{1}\otimes \rho_2$ as $\Delta G_n$-modules. In particular, $V_0\subset (\rho_1\times \rho_2)_{N,\psi}$ yielding $(\rho_1\times \rho_2)_{N,\psi}$ is non-zero.
\end{proof}

\subsubsection{}
We recall the notion of a generic representation (\cite[\S 5.7]{BZ1},\cite[\S 2.3]{Kud1}) of $G_n.$ Let $U_n$ denote the unipotent radical of the standard Borel subgroup $B_n$ of $G_n.$ One has a character $\theta$ of $U_n$ defined by  $\theta(u)=\psi_0(\sum_{i=1}^{n-1} u_{i,i+1})$ for $u\in U_n.$ An irreducible representation $\pi$ of $G_n$ is said to be generic if $\Hom_{U_n}(\pi, \theta)\neq 0.$ 

\subsubsection{}
Note that the subgroup $N(1,1)$ of $G_{n+1}$ is contained in $U_{n+1}$ and the character $\theta$ of $U_{n+1}$ restricted to $N(1,1)$ coincides with $\psi_{1,1}.$  Similarly, the subgroup $N(0,2)$ of $G_{n+1}$ is a subgroup of $U_{n+1}$ and the character $\theta$ restricted to $N(0,2)$ coincides with $\psi_{0,2}.$ By the generality noted in \eqref{subTJM}, we have proved
\begin{lemma}\label{genericlemma}
Let $\rho$ be any irreducible generic representation of $G_{n+1}.$ Then, both $(\rho)_{N(1,1),\psi_{1,1}}$ and $(\rho)_{N(0,2),\psi_{0,2}}$ are  non-zero.
\end{lemma}
As a consequence, we have the following general result.
\begin{proposition}\label{genericnplus1}
Let $\rho\in \Irr(G_{n+1})$ be generic and $\eta\in \Irr(G_{n-1})$ be any representation. Then, 
\begin{enumerate}
	\item $(\rho\times \eta)_{N,\psi} \neq 0$ and,
	\item $(\eta\times \rho) _{N,\psi}\neq 0.$
\end{enumerate}
\end{proposition}
\begin{proof}
We apply Theorem \ref{maincomponentwise} (or Corollary \ref{nonzeroTJMcomponentwise}) with $r=n+1$ to get $\alpha=1.$  We shall show that $V_1$ is non-zero (see \eqref{filtrationmain2} and \eqref{componentwisegluepieces}) by showing that $(\rho)_{N(1,1),\psi_{1,1}}$ and $(\eta)_{N(1,2),\psi_{1,2}}$ are both non-zero.  We note that $N(1,2)$ is the trivial subgroup $\{I_{n-1}\}$ as $n-r+\alpha=0,$ yielding $(\eta)_{N(1,2),\psi_{1,2}}=\eta.$ On the other hand, as $\rho$ is generic, by Lemma \ref{genericlemma}, $(\rho)_{N(1,1),\psi_{1,1}}$ is non zero, proving (1).  The proof of (2) can be achieved similarly by applying Theorem \ref{maincomponentwise} with $r=n-1$ and proving that $V_0$ is non-zero.
\end{proof}

We next present another case where we can show that the twisted Jacquet module of a principal series is non-zero.
\begin{proposition}\label{hered}
Suppose $m$ is a positive integer such that $m<n$ and let $\rho_1\in \Alg(G_{2m})$ and $\rho_2\in \Alg(G_{2n-2m}).$ Let $\psi_1$ and $\psi_2$ denote the non-degenerate character given by \eqref{definitionofpsi} of the subgroups $N_{m,m}$ of $G_{2m}$ and $N_{n-m,n-m}$ of $G_{2n-2m}$ respectively.  Suppose that the twisted Jacquet modules $(\rho_1)_{N_{m,m},\psi_1}$ and $(\rho_2)_{N_{n-m,n-m},\psi_2}$ are both non-zero.  Then, $(\rho_1\times \rho_2)_{N,\psi}\neq0.$
	\end{proposition}
	\begin{proof}
In the notation of Corollary \ref{nonzeroTJMcomponentwise}, we have  $\beta=m.$ Choose $k=\beta=m,$ to get $N(m,1)=N_{m,m}$ and $N(m,2)=N_{n-m,n-m}.$ Then, $\psi_{m,1}$ and $\psi_{m,2}$  are  $\psi_1$ and $\psi_2$ respectively.  Thus, $(\rho_1)_{N(m,1),\psi_{m,1}}$ and $(\rho_2)_{N(m,2),\psi_{m,2}}$ are both non-zero, proving $(\rho_1\times \rho_2)_{N,\psi}\neq0$ by Corollary \ref{nonzeroTJMcomponentwise}.
\end{proof}

\subsubsection{A general recipe}
In view of Corollary \ref{nonzeroTJMcomponentwise}, in general, to construct a principal series $\pi=\rho_1\times \rho_2$ of $G_{2n}$ such that $\pi_{N,\psi}$ is non-zero, it is sufficient to start with $\rho_1$ and $\rho_2$ such that $(\rho_1)_{N(k,1),\psi_{k,1}}$ and $(\rho_2)_{N(k,2),\psi_{k,2}}$ are simultaneously non-zero for some integer $k$ satisfying $\max\{0,r-n\}\leq k \leq \lfloor \frac{r}{2} \rfloor.$ We end this subsection with such a result for which our starting point is \cite{Naor}. 

\subsubsection{Connection with uneven Shalika Model}
Consider the unipotent radical $U_k$ of the parabolic $P_{k,k,r-2k}$ of $G_r.$ Put $H=\Delta G_k \times G_{r-2k}\subset M_{k,k,r-2k},$ where $\Delta G_k$ is the diagonal copy of $G_k$ in $G_k\times G_k.$  Then, $HU_k$ is a subgroup of $P_{k,k,r-2k}.$ Let $\theta$ be any character of $HU_k$ such that $\theta(u)=\psi_0(tr(x)) $ for each $u=\begin{pmatrix}
	I_k&x&y\\0&I_k&z\\0&0&I_{r-2k}
\end{pmatrix}\in U_k.$  Such a character is called a twisted Shalika character in \cite{Naor}. If $\theta_{|H}=1,$ then $\theta$ is said to be a regular (non-twisted) Shalika character as well.  

\subsubsection{} 
It is shown in \cite{Naor} that for any $\rho\in \Irr(G_r)$, $\Hom_{G_r}(\rho, \Ind _{HU_k}^{G_r}(\theta))$ has dimension at most  one as a $\C$-vector space.  If this space has dimension one, then $\rho$ is said to have an uneven Shalika model. Representations with an uneven Shalika model thus forms a nice class of representations.  

\subsubsection{}
Consider $\theta$ as a character of the subgroup $U_k.$ Note that $N(k,1)$ is a subgroup of $U_k$ and $\theta_{|N(k,1)}=\psi_{k,1}.$ If $\rho\in \Irr(G_r)$ satisfies $\rho_{U_k,\theta}\neq 0,$ then $(\rho)_{N(k,1),\psi_{k,1}}\neq 0.$
In a similar way, it is natural to consider the unipotent radical $U_k'$ of the parabolic subgroup $P_{r-2k, n-r+k,n-r+k}$ of $G_{2n-r}.$ For an element  $u'=\begin{pmatrix}
	I_{r-2k}&u&v\\0&I_{n-r+k}&w\\0&0&I_{n-r+k}
\end{pmatrix}\in U_k',$ we may define $\theta'(u')=\psi_0(tr(w)).$ If any irreducible representation $\eta$ of $G_{2n-r}$ satisfies  $\eta_{U_k', \theta'}\neq 0, $ then $(\eta)_{N(k,2),\psi_{k,2}}\neq 0.$ We have obtained the following result.

\begin{proposition}
Suppose that $\rho\in\Irr(G_r)$ and $\eta\in \Irr(G_{2n-r}).$  If the twisted Jacquet modules $(\rho)_{U_k,\theta}$ and $(\eta)_{U_k',\theta'}$ are both non-zero  for some  $k$ such that  $\max\{0,r-n\}\leq k \leq \lfloor \frac{r}{2} \rfloor,$ then $(\rho\times \eta)_{N,\psi}$ is non-zero.
\end{proposition}

\begin{remark}
We shall return to the question of non-vanishing of twisted Jacquet modules in Sections \ref{example1}, \ref{example2} and in general in Section \ref{DPConjecture}
\end{remark}

\subsection{Applications}\label{mainapplications}
\subsubsection{}
We shall show how our results can be used to compute the structure of twisted Jacquet modules of certain irreducible representations of $G_{2n}$ which are themselves not principal series representations but rather  arise as subquotients of reducible principal series representations. We shall be using the notion of segments and multisegments and related results in our applications. We refer the reader to  \cite{Zelevinsky} for the notion of segments, multisegments and associated results. Also,  see \cite{Kud1} for a summary of the main results.

\subsubsection{}
We briefly recall some notations. If $\Delta$ is a segment \cite[\S 3]{Zelevinsky}, then by $Z(\Delta)$ and $L(\Delta),$ we mean the unique irreducible subrepresentation and the unique irreducible quotient associated to $\Delta.$ In particular, if $\chi$ is a character of $F^{\times}$ and given any positive integer $r,$ let $\Delta$ denote the segment $[\chi\nu^{-\frac{r-1}{2}},\dots, \chi\nu^{\frac{r-1}{2}}].$ Then, $Z(\Delta)$ is the character $\chi$ of $G_r$ and $L(\Delta)$ is the Steinberg representation $St_r$ of $G_r.$ A very useful result (\cite[\S 4.2 Theorem]{Zelevinsky}) which we shall use frequently is that  a representation of the form $Z(\Delta_1)\times \cdots \times Z(\Delta_k)$ is irreducible if and only if $\Delta_i$ and $\Delta_i$ are not linked for any $i,j$. An analogous statement holds (\cite[\S 9.7 Theorem]{Zelevinsky}) for representations of the form $L(\Delta_1)\times\cdots \times L(\Delta_k).$

\subsubsection{}
If $\mfr{m}=\{\Delta_1, \dots, \Delta_r\}$ is a multisegment such that $\Delta_i$ does not precede $\Delta_j$ whenever $i<j,$ by $Z(\mfr{m})$ and $L(\mfr{m})$ we mean the unique irreducible submodule and the unique irreducible quotient of the parabolically induced representations $Z(\Delta_1)\times \cdots \times Z(\Delta_r)$ and $L(\Delta_1)\times \cdots \times L(\Delta_r)$ respectively. We shall also use the Langlands Quotient theorem for $G_n$ (see for instance \cite[Theorem 1.2.5 and Theorem 2.2.2]{Kud1}). Since we shall be using the Moeglin-Waldspurger algorithm associated to Zelevinsky involution multiple times, we briefly recall the algorithm below from  \cite{Badulescu-Renard} and \cite{Moeglin-Waldspurger}.

\subsubsection{Zelevinsky Involution and the Moeglin-Waldspurger algorithm}\label{MW}

Let $\mfr{m}$ be a multisegment (that is a set of segments with multiplicities). The Moeglin-Waldspurger algorithm  associates a multisegment $\mfr{m}^t$ to $\mfr{m}$ such that $Z(\mfr{m})=L(\mfr{m}^t)$ and vice-versa. We refer the reader to \cite[II.1 \& II.2]{Moeglin-Waldspurger}, \cite[\S 1]{Badulescu-Renard} and \cite[\S 2.8]{BLM} for more details.   

In our work, we shall be considering only  segments of the type $[\chi\nu^{a}, \chi\nu^{a+1},\cdots, \chi\nu^{b}]$ where $\chi$ is a character of $F^{\times}.$ Suppose  $\Delta=[\chi\nu^{b}, \chi\nu^{e}]$ and $\Delta'=[\chi\nu^{b'}, \chi\nu^{e'}]$ are segments where $\chi\in F^{\times}.$
We say that $\Delta\geq \Delta'$ if $b>b'$ or $b=b'$ and $e\geq e'.$ To compute $\mfr{m}^t,$ the algorithm is as follows. Let $d$ be the biggest ending of a segment in $\mfr{m}.$ Choose a segment $\Delta_{i_0}$ in $\mfr{m}$ containing $d$ and maximal for this property. Define integers $i_1,i_2, \dots, i_r$ inductively as follows:
$\Delta_{i_s}$ is a segment of $\mfr{m}$ preceding $\Delta_{i_{s-1}}$ with ending $d-s$ maximal with these properties, and $r$ is such that there is no possibility to find such an $i_{r+1}.$ Let $\mfr{m}^{-}=(\Delta_1',\dots, \Delta_{t}')$ where $\Delta_{i}'=\Delta_{i}$ if $i\notin \{i_0,\dots, i_r\}$ and $\Delta_{i}'=\Delta_{i}^{-}$ if $i\in \{i_0,\dots, i_r\}.$ Then $\{d-r,\dots, d\}$ is the first segment of $\mfr{m}^t.$ By repeating the process to $\mfr{m}^{-}$ that was applied to $\mfr{m},$ find the second segment of $\mfr{m}^t.$ So $\mfr{m}^t$ is the union of $\{d-r,\dots, d\}$ and $(\mfr{m}^{-})^{t}.$ 

\subsubsection{} 
As a preparatory step, we summarize in the following lemma how the first step of Moeglin-Waldspurger algorithm works in some special cases so as to use them suited to our context.
\begin{lemma}\label{MWalgorithm}
Let $\chi$ be a character of $F^{\times}.$
\begin{enumerate}
	\item Suppose  $\mfr{m}=\{\Delta\}$ where $\Delta=[\chi\nu^{b},\dots, \chi\nu^{e}].$ Then, $\mfr{m}^t=\{[\chi\nu^{e}],\dots, [\chi\nu^{b}]\}.$
	
	\item Assume that $\mfr{m}=\{\Delta_1,\Delta_2\}$ where $\Delta_1=[\chi\nu^{b},\dots, \chi\nu^{e}]$ and $\Delta_2=[\chi\nu^{b'},\dots,\chi\nu^{e+1}]$ where $b'>b.$ Then,  $\mfr{m}^t$ is the union of $[\chi\nu^{e},\chi\nu^{e+1}]$ and $(\mfr{m}^{-})^t$ where $\mfr{m}^{-}=\{\Delta_1^{-},\Delta_2^{-}\}.$ 
	\item  Let  $\mfr{m}=\{\Delta_1,\Delta_2\}$ where $\Delta_1=[\chi\nu^{b},\dots, \chi\nu^{e}]$ and $\Delta_2=[\chi\nu^{b'},\dots,\chi\nu^{e'}]$ where $e'>e+1$ and $b'>b.$ Then, $\mfr{m}^t$ is the union of $[\chi\nu^{e'}]$ and $(\mfr{m}^{-})^t$ where $\mfr{m}^{-}=\{\Delta_1,\Delta_2^{-}\}.$ 
\end{enumerate}	

\end{lemma}
\begin{proof}
The statement (1) is obvious. To prove (2), using the notations in the algorithm, $d=e+1$ or equivalently $d-1=e.$ Hence, $\Delta_{i_0}=\Delta_2$ and $\Delta_{i_1}=\Delta_1.$ Then, $\Delta_1^{-}=[\chi\nu^{b},\dots, \chi\nu^{e-1}]$ and $\Delta_2^{-}=[\chi\nu^{b'},\dots,\chi\nu^{e}],$ by which  $\mfr{m}^{-}=\{\Delta_1^{-},\Delta_2^{-}\},$ proving (2). In the setting of (3), $d=e'$ and $d-1=e'-1\neq e.$ In this case, we get $\Delta_{i_0}=\Delta_2,$ $\Delta_2^{-}=[\chi\nu^{b'},\dots,\chi\nu^{e'-1}],$ and  $\mfr{m}^{-}=\{\Delta_1,\Delta_2^{-}\}.$ Thus, $\mfr{m}^t$ is the union of $[\chi\nu^{e'}]$ and $(\mfr{m}^{-})^t,$  which proves (3).
\end{proof}

\begin{lemma}\label{Lchialpha}
Let $\chi$ be  a character of $F^{\times}.$ Put $\Delta=[\chi\nu^{-(\frac{n-1}{2})},\dots, \chi\nu^{\frac{n-1}{2}}].$ For each $\alpha\in \{1, \dots, n\}$ put $\Delta_{\alpha}=[\chi\nu^{-(\frac{n-1}{2})+\alpha},\dots, \chi\nu^{\frac{n-1}{2}+\alpha}]$ and $\xi_{\chi,\alpha}=Z(\Delta_{\alpha})\times Z(\Delta).$ Then, $\xi_{\chi, \alpha}=\chi\nu^{\alpha}\times \chi$ sits in the following exact sequence of $G_{2n}$-modules
\begin{equation}\label{Zelexactsequence-Lalpha}
0 \to  L_{\chi,\alpha}\to \xi_{\chi,\alpha} \to Z(\Delta \cup \Delta_{\alpha}) \times Z(\Delta \cap \Delta_{\alpha}) \ \to 0
\end{equation}
where $L_{\chi,\alpha}$ is the unique irreducible submodule of $\xi_{\chi,\alpha}.$ If $1<\alpha \leq n,$ put 
 \begin{equation*}\pi_{\chi,\alpha}:=\chi\nu^{\frac{n-1}{2}+\alpha}\times \cdots \times \chi\nu^{\frac{n-1}{2}+2}\times St_2\chi\nu^{\frac{n}{2}} \times \cdots \times St_2\chi\nu^{-\frac{n}{2}+\alpha}  \times \chi\nu^{-(\frac{n-1}{2})+\alpha-2}\times \cdots \times \chi\nu^{-(\frac{n-1}{2})},
\end{equation*}
and put 
\begin{equation*}
	\pi_{\chi,1}=St_2\chi\nu^{\frac{n}{2}} \times St_2\chi\nu^{\frac{n}{2}-1}\times \cdots \times St_2\chi\nu^{-\frac{n}{2}+1}.
\end{equation*}	
Then, $L_{\chi,\alpha}$ is the unique irreducible quotient, and in particular the Langlands quotient, of the representation $\pi_{\chi,\alpha}$ for $1\leq \alpha\leq n.$ 
\end{lemma}

\begin{proof}
The exact sequence \eqref{Zelexactsequence-Lalpha} follows from \cite[Proposition 4.6]{Zelevinsky} (see also \cite[Lemma 2.2]{CGV}). Fix an $\alpha$ and write $\mfr{m}_{\alpha}=\{\Delta_{\alpha},\Delta\}.$ We note  that $\Delta$ precedes $\Delta_{\alpha}$ for each $\alpha$ and hence $L_{\chi,\alpha}=Z(\mfr{m}_{\alpha}).$   In order to obtain the result, we apply Moeglin-Waldspurger algorithm to $\mfr{m}_{\alpha}$ and obtain a multisegment $\mfr{m}_{\alpha}^t$ so that we may write $L_{\chi,\alpha}=L(\mfr{m}_{\alpha}^t).$  We divide the proof into  two cases.

\noindent\textit{Case 1}: If $\alpha=1,$ we can directly apply  Lemma \ref{MWalgorithm} (2) to get  that $\mfr{m}^t$ is the union of $n$ length $2$ segments 
\begin{equation*} [\chi\nu^{-(\frac{n-1}{2})+n-1},\chi\nu^{-(\frac{n-1}{2})+n}],\dots, [\chi\nu^{-(\frac{n-1}{2})},\chi\nu^{-(\frac{n-1}{2})+1}].
\end{equation*}
Consequently, $L_{\chi,1}$ is the Langlands quotient of $\pi_{\chi,1}.$ \\

\noindent\textit{Case 2}: Suppose $\alpha>1.$ For $1\leq i \leq \alpha,$ put 
\begin{equation*}
\Delta(i)=[\chi\nu^{-(\frac{n-1}{2})+\alpha},\dots, \chi\nu^{\frac{n-1}{2}+i}], \text{ and }\mfr{m}_i=\{\Delta(i),\Delta\}.
\end{equation*}
For $\alpha \leq j \leq  n,$ put \begin{equation*}
\Delta(j)'=[\chi\nu^{-(\frac{n-1}{2})},\dots, \chi\nu^{-(\frac{n-1}{2})+j-1}]  \text{ and } \Delta(j)''= [\chi\nu^{-(\frac{n-1}{2})+\alpha},\dots, \chi\nu^{-(\frac{n-1}{2})+j}].
\end{equation*} 
Also, note that $\alpha\geq 2$ and put $\Delta_{\alpha-1}'= [\chi\nu^{-(\frac{n-1}{2})},\dots, \chi\nu^{-(\frac{n-1}{2})+\alpha-2}]$ and $\Delta_{\alpha-1}''$ be the empty segment. Set
\begin{equation*}
\mfr{m}_j':=\{\Delta(j)',\Delta(j)''\} \text{ when } \alpha-1\leq j \leq  n.
\end{equation*}

 By Lemma \ref{MWalgorithm} (3),  $\mfr{m}_i^t$ is the union of $\{[\chi\nu^{\frac{n-1}{2}+i}]\}$ and $(\mfr{m}_{i}^{-})^t$ for $2 \leq i \leq \alpha.$ One observes that $\mfr{m}_{i}^{-}=\mfr{m}_{i-1}.$ Hence,  $\mfr{m}_{\alpha}^t$ is the union of $\{[\chi\nu^{\frac{n-1}{2}+\alpha}],\dots,[\chi\nu^{\frac{n-1}{2}+2}]\}$ and $\mfr{m}_1^t.$ We note that $\mfr{m}_1=\mfr{m}_{n}'$ and determining $\mfr{m}_1^t$ is the same as determining $(\mfr{m}_{n}')^{t}.$  Applying Lemma \ref{MWalgorithm} (2), for $\alpha\leq j \leq n,$   $(\mfr{m}_j')^t$ is the union of $[\chi\nu^{-(\frac{n-1}{2})+j-1},\chi\nu^{-(\frac{n-1}{2})+j}]$ and $((\mfr{m}_j')^{-})^t.$ One observes that $(\mfr{m}_j')^{-}=\mfr{m}_{j-1}'.$ Consequently, $(\mfr{m}_j')^t$ is the union of $[\chi\nu^{-(\frac{n-1}{2})+j-1},\chi\nu^{-(\frac{n-1}{2})+j}]$ and $(\mfr{m}_{j-1}')^t$ for $\alpha \leq  j \leq n.$ Recursively,  $\mfr{m}_1^t$ is the union of 
\begin{equation*}
\{[\chi\nu^{-(\frac{n-1}{2})+n},\chi\nu^{-(\frac{n-1}{2})+n+1}],\dots,[\chi\nu^{-(\frac{n-1}{2})+\alpha-1},\chi\nu^{-(\frac{n-1}{2})+\alpha}]\}
\end{equation*} and $\mfr{m}_{\alpha-1}'.$ By Lemma \ref{MWalgorithm} (1), $(\mfr{m}_{\alpha-1}')^t = \{[\chi\nu^{-(\frac{n-1}{2})+\alpha-2}],\cdots,[\chi\nu^{-(\frac{n-1}{2})}]\}.$ 
We finally get $\mfr{m}_{\alpha}^t$ consists of the following:
\begin{enumerate}
	\item  $\alpha-1$ singleton segments $[\chi\nu^{\frac{n-1}{2}+\alpha}],\dots,[\chi\nu^{\frac{n-1}{2}+2}],$
	\item $n-\alpha+1$ segments  $[\chi\nu^{-(\frac{n-1}{2})+n-1},\chi\nu^{-(\frac{n-1}{2})+n}],\dots, [\chi\nu^{-(\frac{n-1}{2})+\alpha-1},\chi\nu^{-(\frac{n-1}{2})+\alpha}],$ of length $2$ and
	\item $\alpha-1$ singleton segments $[\chi\nu^{-(\frac{n-1}{2})+\alpha-2}],\dots,[\chi\nu^{-(\frac{n-1}{2})}].$
\end{enumerate}
As $L([\chi\nu^{\beta},\chi\nu^{\beta+1}])=St_2\chi\nu^{\beta+1/2}$ for any $\beta\in \mathbb{R},$ it follows that $L_{\chi,\alpha}=L(\mfr{m}_{\alpha}^t)$ is indeed the Langlands quotient of the representation $\pi_{\chi,\alpha}.$
\end{proof}

\begin{remark}
We refer the reader to \cite[\S 4.1, Lemme 4.2]{BLM} to compare with Lemma \ref{Lchialpha}. Our lemma is a special case of the more general  `ladder representations' appearing there. As per the notations of Lemme 4.2 of \cite{BLM}, in our case we have $N=2$ and there must be $\frac{n-1}{2}+\alpha-(-\frac{n-1}{2})+N-2= n+\alpha-1$ segments in $\mfr{m}^t.$ In our description, if $\alpha>1$ the number of segments in $\mfr{m}^t=2(\alpha-1)+n-\alpha+1=n+\alpha-1.$ If $\alpha=1,$  the same conclusion holds as we have $n-\alpha+1=n.$ 
\end{remark}
\begin{remark}
We note that the representations $L_{\chi,\alpha}$ obtained in Corollary \ref{Lchialpha} are non-generic (see \cite[\S 9.8]{Zelevinsky}).
\end{remark}

\subsubsection{} 
As an immediate application of Lemma \ref{Lchialpha}, we can deduce the following result on the twisted Jacquet module of $L_{\chi,\alpha}.$

\begin{corollary}\label{TJMLchialpha}
Keeping notations as in Lemma \ref{Lchialpha}, 	for $1\leq  \alpha \leq n,$  let $L_{\chi,\alpha}$ denote the Langlands quotient of the representation $\pi_{\chi,\alpha}.$ Then, $(L_{\chi,\alpha})_{N,\psi}\simeq \chi\otimes \chi\nu^{\alpha}$ as $\Delta G_n$-modules.
\end{corollary}
\begin{proof}
For each $\alpha\in\{1,\dots,n\},$ the segments $\Delta$ and $\Delta_{\alpha}$ are linked.  If $\alpha=n,$ $\Delta$ and $\Delta_{\alpha}$ are juxtaposed, so that $Z(\Delta\cup \Delta_{\alpha})$ is the character $\chi\nu^{\frac{n}{2}}$ of $G_{2n}.$ In this case,  $\xi_{\chi,n}$ (see \eqref{Zelexactsequence-Lalpha}) has the unique irreducible quotient $\chi\nu^{\frac{n}{2}}.$ By Theorem \ref{productoftwocharcaters}, $(\xi_{\chi,n})_{N,\psi}=\chi\otimes \chi\nu^{n}.$ Since, $(\chi\nu^{\frac{n}{2}})_{N,\psi}=0$ we get $(\xi_{\chi, n})= (L_{\chi,n})_{N,\psi}=\chi\otimes \chi\nu^{n}.$ 

On the other hand, if $1\leq \alpha <n,$ $\Delta \cap \Delta_{\alpha}$ is non-empty. This yields that $Z(\Delta \cup \Delta_{\alpha})$ and $Z(\Delta \cap \Delta_{\alpha})$ are both one-dimensional. Again, by Theorem \ref{productoftwocharcaters},  $(Z(\Delta \cup \Delta_{\alpha}) \times Z(\Delta \cap \Delta_{\alpha}))_{N,\psi}=0$  and  $(\xi_{\chi,\alpha})_{N,\psi}\simeq\chi\otimes \chi\nu^{\alpha}$  as $\Delta G_n$-modules. It follows from the exact sequence \eqref{Zelexactsequence-Lalpha} and exactness of the twisted Jacquet functor that $(L_{\chi,\alpha})_{N,\psi}\simeq \chi\otimes \chi\nu^{\alpha}$ as $\Delta G_n$-modules.
\end{proof}

\subsubsection{An application to Shalika Model}
As an application of Corollary \ref{TJMLchialpha} , we deduce the following result on existence of Shalika model for $L_{\chi,\alpha}$ when $\chi=\nu^{-\frac{\alpha}{2}}.$ We recall that a smooth representation $\pi$ of $G_{2n}$ has a Shalika model if $\Hom_{G_{2n}}(\pi, \Ind_{S_{\psi}}^{G_{2n}}(\one\otimes\psi))$ is non-zero. By Frobenius reciprocity, this is equivalent to saying that 
	$\Hom_{S_{\psi}}(\pi_{_{|_{S_{\psi}}}}, \one\otimes \psi)\simeq\Hom_{\Delta G_n}(\pi_{N,\psi}, \one)\neq 0,$
 where $\one$ denotes the trivial character of $\Delta G_n.$ 
\begin{corollary}\label{shalikamodel}
For each $\alpha\in \{1,\dots, n\},$ the irreducible representation $L_{\nu^{-\frac{\alpha}{2}}, \alpha}$ possesses a Shalika model. Consequently, it  has a non-zero $G_n\times G_n$-invariant linear form.
\end{corollary}
\begin{proof}
Since, $(L_{\nu^{-\frac{\alpha}{2}}, \alpha})_{N,\psi}\simeq \C,$ it is trivial to see that it possesses a Shalika model. It is a result of Jacquet-Rallis \cite[Remark pp.117 ]{JRCompositio} that any irreducible representation of $G_{2n}$ which has a Shalika model has a non-zero $G_n\times G_n$-invariant linear form.
\end{proof}	

\begin{remark}
We note that $L_{\nu^{-1/2},1}$ is the unique irreducible quotient of the representation
$$ St_2\nu^{\frac{n-1}{2}}\times St_2\nu^{\frac{n-3}{2}}\times \cdots \times St_2\nu^{-\frac{n-1}{2}}.$$
A representation such as  $L_{\nu^{-1/2},1}$ is an example of a Speh representation in the literature.
\end{remark}
For further results on Shalika Models, we refer the reader to the work of N. Matringe \cite{Mat} and the references therein.

\subsection{Subquotients of $St_r\chi \times \mu$}	\label{Steinberg-v-char}
\subsubsection{}
As another application of our results, we investigate the structure of twisted Jacquet module of subquotients of a parabolically induced representation  of the form $St_r\cdot \chi \times \mu$ where $\chi$ and $\mu$ are characters of $G_r$ and  $G_{2n-r}$ respectively. Put $\Delta=[\chi\nu^{-\frac{r-1}{2}},\dots, \chi\nu^{\frac{r-1}{2}}]$ and $\Delta'=[\mu\nu^{-\frac{2n-r-1}{2}},\dots, \mu\nu^{\frac{2n-r-1}{2}}]$ so that we can write $St_r\cdot \chi \times \mu=L(\Delta)\times Z(\Delta').$ As we already know the structure of $(St_r\cdot \chi \times \mu)_{N,\psi}$ via Proposition \ref{smoothvscharacter}, we would be interested in determining the structure of the twisted Jacquet modules of the subquotients of $St_r\cdot \chi \times \mu$  whenever it is reducible. 

\subsubsection{}
We recall \cite[Th\'eor\`eme 3.1]{BLM} that $St_r\cdot \chi \times \mu$ is reducible if and only if the segments $\Delta$ and $\Delta'$ are juxtaposed, i.e., if and only if  $\mu=\chi\nu^{\pm n}.$ Moreover, in the case when $St_r\cdot \chi \times \mu$ reduces, it has length $2$ by \cite[Remarque 3.2]{BLM}.  We will assume that $\mu=\chi\nu^{-n}.$ The case when $\mu=\chi\nu^{n}$ can be treated similarly. As we will be dealing with the character $\nu=|\det(\cdot)|_F$ of $G_r$ as $r$ varies, we choose to write $\nu_r$ instead of $\nu$ sometimes for avoiding confusion. The trivial character of $G_n$ shall be denoted by $\one_n.$

\begin{proposition}\label{Steinbergvscharacter}
	Assume that $r\leq n.$ Suppose $\chi$ is a  character of $F^{\times}.$ Let $Z_{\chi,r}$ and $Q_{\chi,r}$	respectively denote the unique irreducible submodule and the unique irreducible quotient of $St_r\chi\times \chi \nu_{2n-r}^{-n}.$ The following statements hold:
	\begin{enumerate}
		\item If $r<n$ then $(Z_{\chi,r})_{N,\psi}=0$ and $(Q_{\chi,r})_{N,\psi}=0.$
		\item We have $(Z_{\chi,n})_{N,\psi}=St_n\chi \otimes \chi\nu^{-n}$ and $(Q_{\chi,n})_{N,\psi}=0.$
	\end{enumerate}
\end{proposition}
\begin{proof}
	If $r<n,$ by Proposition \ref{smoothvscharacter}(i), $(St_r\chi\times \chi \nu_{2n-r}^{-n})_{N,\psi}=0.$  Hence (1) follows and we need to consider only the case  $r=n.$ Put 
	\begin{equation*}
		\Delta=[\chi\nu^{-\frac{n-1}{2}},\dots, \chi\nu^{\frac{n-1}{2}}] \text{ and } \Delta'=[\chi\nu^{-\frac{n-1}{2}-n},\dots, \chi\nu^{\frac{n-1}{2}-n}],
	\end{equation*}
	so that $St_n\chi\times \chi \nu_{n}^{-n}=L(\Delta)\times Z(\Delta').$
	Then, we have the following exact sequence of $G_{2n}$-modules:
	\begin{equation}\label{Zchin}
		0\to Z_{\chi,n}\to St_n\chi\times \chi \nu_{n}^{-n} \to Q_{\chi,n}\to 0.
	\end{equation}
	By Proposition \ref{smoothvscharacter} (ii), we have  $(St_n\chi\times \chi \nu_{n}^{-n})_{N,\psi}=St_n\chi \otimes \chi\nu^{-n}.$ We claim that $(Q_{\chi,n})_{N,\psi}=0$ which will give us $(Z_{\chi,n})_{N,\psi}=St_n\chi \otimes \chi\nu_n^{-n}.$ Note that $Q_{\chi,n}$ $=L(\mathfrak{m})$ where 
	\begin{equation*}
		\mfr{m}=\{[\chi\nu^{-\frac{n-1}{2}},\dots, \chi\nu^{\frac{n-1}{2}}],[\chi\nu^{\frac{n-1}{2}-n}],[\chi\nu^{\frac{n-1}{2}-n-1}],\dots,[\chi\nu^{-\frac{n-1}{2}-n}] \}.
	\end{equation*}
	By MW algorithm, $\mfr{m}^t=\{[\chi\nu^{\frac{n-1}{2}}],\dots, [\chi\nu^{-\frac{n-3}{2}}],  [\chi\nu^{-\frac{n-1}{2}-n},\dots,\chi\nu^{\frac{n-1}{2}-n}, \chi\nu^{-\frac{n-1}{2}}] \}$ and $Q_{\chi,n}=Z(\mfr{m}^t).$ Thus, $Q_{\chi,n}$ is a subrepresentation of  
	\begin{equation*}
		\rho\times Z([\chi\nu^{-\frac{n-1}{2}-n},\dots, \chi\nu^{\frac{n-1}{2}-n},\chi\nu^{-\frac{n-1}{2}}])
	\end{equation*}
	where $Z([\chi\nu^{-\frac{n-1}{2}-n},\dots, \chi\nu^{\frac{n-1}{2}-n},\chi\nu^{-\frac{n-1}{2}}])=\chi\nu_{n+1}^{\frac{-n+1}{2}}$ is a character of $G_{n+1}$ and  $\rho=\chi\nu^{\frac{n-1}{2}}\times \times\dots\times \chi\nu^{-\frac{n-3}{2}}\in \Alg(G_{n-1}).$ By Proposition \ref{smoothvscharacter} (i), $(\rho \times\nu_{n+1}^{\frac{-n+1}{2}} )_{N,\psi}=0$ and consequently, $(Q_{\chi,n})_{N,\psi}=0.$
\end{proof}

\begin{remark}\label{Steinbergvscharacterdual}
For a smooth representation $\tau$ of $G_n,$ let $\tau^{\vee}$ denote its contragredient. By considering contragredients of the representations in Proposition \ref{Steinbergvscharacter} and the automorphism  $s$ (\cite[\S 1.9 Theorem \& Lemma]{Zelevinsky}) of $G_{2n}$ which maps $P_{r,2n-r}$ to $P_{2n-r,r}$ we get  the exact sequence
\begin{equation*}
0\to Z_{\chi,r}^{\vee} \to \chi^{-1}\nu_{2n-r}^n \times St_r \chi^{-1} \to Q_{\chi,r}^{\vee} \to 0.
\end{equation*}
By arguments as in the proof Proposition \ref{Steinbergvscharacter}, we may conclude that if $r<n,$ then $(Q_{\chi,r}^{\vee})_{N,\psi}=0$ and $(Z_{\chi,r}^{\vee})_{N,\psi}=0.$ Also, if $r=n,$ we have $Q_{\chi,n}^{\vee}$ is  quotient of the principal series $\chi^{-1}\nu_{n+1}^{\frac{n-1}{2}}\times \rho^{\vee}.$ Hence, $(Q_{\chi,n}^{\vee})_{N,\psi}=0$ as $(\chi^{-1}\nu_{n+1}^{\frac{n-1}{2}}\times \rho^{\vee})_{N,\psi}=0$ by Theorem \ref{productoftwocharcaters}. Consequently, $(Z_{\chi,r}^{\vee})_{N,\psi}=\chi^{-1}\nu^n\otimes St_n\chi^{-1}.$
\end{remark}

\begin{remark}\label{rgreaterthann}
In contrast to the conclusions of Proposition \ref{Steinbergvscharacter}, if $r>n,$ both the subquotients of a reducible principal series representation of the type $St_r\chi \times \mu$ with length $2$ can have non-zero twisted Jacquet modules as we shall show in  the next subsection. Note that, if $r=n+1,$ we already know that such a principal series has a non-zero twisted Jacquet module by Proposition \ref{genericnplus1}.
\end{remark}

\subsection{Examples of non-vanishing twisted Jacquet modules I}\label{example1}
\subsubsection{}
In view of the results obtained in Section \ref{smooth-v-char}-\ref{Steinberg-v-char}, given any $\pi\in \Irr(G_{2n}),$ it would be interesting to know when $\pi_{N,\psi}\neq 0.$  To analyze this, we shall consider two reducible principal series representations of $G_4$, one in this section and another one in the next. In general, if $\xi$ is any reducible principal series representation of $G_{2n}$ with a non twisted Jacquet module $\xi_{N,\psi},$ there may be several subquotients $\pi$ with a non-zero $\pi_{N,\psi}.$  

\subsubsection{Structure theorem for $G_4$} For the ease of reference, we note the following special case deduced from our main structure theorems by putting $n=r=2$ in Corollary \ref{normalisedcomponentwisecor}. 
\begin{proposition}\label{DPanalogue}
	Let $\chi, \chi_1, \chi_2$ be characters of $F^{\times}.$ Also, let $\mathcal{G}_2$ denote the collection of all smooth representations of $G_2$ consisting of irreducible cuspidal representations,  twists of the Steinberg representation $St_2\chi$ or a (not necessarily irreducible) principal series representation  $\chi_1\times \chi_2.$  Assume that $\rho_1,\rho_2\in \Alg(G_2).$ Then, the following statements hold:
	\begin{enumerate}
		\item If one of the $\rho_i$'s is one dimensional, $(\rho_1\times \rho_2)_{N,\psi}=\rho_1\otimes \rho_2$ as $\Delta G_2$-modules.
		\item If both $\rho_1, \rho_2\in \mathcal{G}_2,$ we have the following exact sequence of $\Delta G_2$ modules:
		
		\begin{equation}\label{GL(4)exactsequencerefined}
			0\to\rho_1\otimes \rho_2\to (\rho_1\times \rho_2)_{N,\psi} \to i_{\Delta B_2}^{\Delta G_2}(\omega_{\rho_1}\nu^{\frac{1}{2}}\otimes \omega_{\rho_2}\nu^{-\frac{1}{2}})\to 0,
		\end{equation} 
		where $\omega_{\rho_i}$ denotes the central character of the representation $\rho_i$ for $i=1,2.$
	\end{enumerate} 
\end{proposition}	

\begin{remark}
	We note that the exact sequence \eqref{GL(4)exactsequencerefined} is identical to the one of Prasad derived in \cite[Propospition 7.1 ]{GS} except the fact that we allow the $\rho_i$'s to be reducible principal series representations of $GL_2(F)$ as well. 
\end{remark}

\subsubsection{}\label{exactseq1}
Let $B_4$ denote the standard Borel subgroup of $G_4$ and let $\xi:=i_{B_4}^{G_4}(\delta_{B_4}^{1/2})= \nu^{\frac{3}{2}}\times \nu^{\frac{1}{2}}\times \nu^{-\frac{1}{2}}\times \nu^{-\frac{3}{2}}.$  It is well known (see \cite[\S 2.1 Proposition]{Zelevinsky} and \cite[\S 5.1]{Kud1}) that the reducible principal series representation $\xi$ has $8$ irreducible subquotients and each of these appear with multiplicity exactly one in $\xi.$  The set ${\rm JH}^{0}(\xi)$ consisting of all irreducible subquotients of $\xi$ counted with multiplicity is given by  (see \cite[\S 5.1]{Kud1})
$${\rm JH}^{0}(\xi)=\{\one_4, L_{\nu^{-1},2}, Z_{\nu,2}, Q_{\nu_2}, Z_{\nu,2}^{\vee}, Q_{\nu,2}^{\vee}, \tau, St_4\},$$
where $\tau$  is the unique irreducible quotient (and also the Langlands quotient) of the principal series representation $St_2\nu\times St_2\nu^{-1}$ of $G_4.$ We shall determine precisely which subquotients of $\xi$ has a non-zero twisted Jacquet module  as a consequence of our results so far. 
 It is trivial that $(\one_4)_{N,\psi}=0.$ We have  $(Q_{\nu,2})_{N,\psi}=0$ and $(Q_{\nu,2}^{\vee})_{N,\psi}=0$  whereas  $(Z_{\nu,2})_{N,\psi}$ and $(Z_{\nu,2}^{\vee})_{N,\psi}$ are both non-zero  by Proposition \ref{Steinbergvscharacter} and Remark \ref{Steinbergvscharacterdual}. It remains to consider $\tau$ and $St_4.$

  \subsubsection{}\label{TJMtau}
  We need the following lemma. 
\begin{lemma}\label{preSteinbergvscharacternonzero}
The  unique irreducible quotient $\tau$ of the principal series representation $St_2\nu\times St_2\nu^{-1}$ of $G_4$  has $\tau_{N,\psi}= \nu^{\frac{3}{2}}\times \nu^{-\frac{3}{2}}$ as $\Delta G_2$-modules.	
\end{lemma}
\begin{proof}
 The strategy of our proof to compute the structure of $\tau_{N,\psi}$ is by realizing $\tau$ as a subquotient of another principal series representation $\zeta$ such that $\zeta_{N,\psi}$ has a simple structure and all irreducible subquotients $\pi$ of $\zeta$ except $\tau$ satisfy $\pi_{N,\psi}=0.$  This will yield, $\tau_{N,\psi}=\zeta_{N,\psi}.$

Since $\tau$ is the Langlands quotient of $St_2\nu\times St_2\nu^{-1},$ if we 
put $\mfr{m}=\{[\nu^{\frac{1}{2}}, \nu^{\frac{3}{2}}],[\nu^{-\frac{3}{2}}, \nu^{-\frac{1}{2}}]\},$ then $\tau=L(\mfr{m}).$
By MW algorithm, $\mfr{m}^t=\{[\nu^{\frac{3}{2}}],[\nu^{-\frac{1}{2}},\nu^{\frac{1}{2}}],[\nu^{-\frac{3}{2}}]\}.$ We then have $\tau=L(\mfr{m})=Z(\mfr{m}^t).$ Thus, $\tau$ is the unique irreducible submodule of the reducible principal series $\nu^{\frac{3}{2}} \times \one_2 \times \nu^{-\frac{3}{2}}.$
Put $\zeta=\one_2\times \nu^{\frac{3}{2}}\times \nu^{-\frac{3}{2}}.$ Set  $\mfr{n}_1=\{[\nu^{-\frac{3}{2}},\nu^{-\frac{1}{2}},\nu^{\frac{1}{2}},\nu^{\frac{3}{2}}]\}, \mfr{n}_2=\{[\nu^{\frac{3}{2}}],[\nu^{-\frac{3}{2}},\nu^{-\frac{1}{2}},\nu^{\frac{1}{2}}]\}$ and $\mfr{n}_3=\{[\nu^{-\frac{1}{2}},\nu^{\frac{1}{2}},\nu^{\frac{3}{2}}],[\nu^{-\frac{3}{2}}]\}.$  By \cite[\S 7.1 Theorem \& 9.13 Proposition]{Zelevinsky}, the irreducible subquotients of  $\zeta$ and  of $\nu^{\frac{3}{2}} \times \one_2 \times \nu^{-\frac{3}{2}}$ are $\tau=Z(\mfr{m}^t), Z(\mfr{n_1}), Z(\mfr{n}_2)$ and $Z(\mfr{n}_3).$ Moreover, each of these subquotients appear with multiplicity exactly one. One has $Z(\mfr{n}_1)=\one_4.$  By Lemma \ref{MWalgorithm}, we have $\mfr{n}_2^t=\{[\nu^{\frac{1}{2}},\nu^{\frac{3}{2}}],[\nu^{-\frac{1}{2}}], \nu^{-\frac{3}{2}}\}.$ Thus, $Z(\mfr{n}_2)=L(\mfr{n}_2^t)=\mfr{L}(St_2\nu\times \nu^{-\frac{1}{2}}\times \nu^{-\frac{3}{2}})=Q_{\nu,2}.$ Apply Lemma  \ref{MWalgorithm} again to get $\mfr{n}_3^t=\{[\nu^{\frac{3}{2}}], [\nu^{\frac{1}{2}}],[\nu^{-\frac{3}{2}},\nu^{-\frac{1}{2}}]\}$ so that $Z(\mfr{n}_3)=L(\mfr{n}_3^t)=\mfr{L}(\nu^{\frac{3}{2}}\times \nu^{\frac{1}{2}}\times St_2\nu^{-1})=Q_{\nu,2}^{\vee}.$

By Proposition \ref{DPanalogue}(1),  $\zeta_{N,\psi}=\one_2\otimes (\nu^{\frac{3}{2}}\times \nu^{-\frac{3}{2}})= \nu^{\frac{3}{2}}\times \nu^{-\frac{3}{2}}$ as $\Delta G_2$-modules. Since $(\one_4)_{N,\psi}$, $(Q_{\nu,2})_{N,\psi}$ and $(Q_{\nu,2}^{\vee})_{N,\psi}$ are all zero,  by the exactness of the twisted Jacquet functor, $\tau_{N,\psi}=\zeta_{N,\psi}=\nu^{\frac{3}{2}}\times \nu^{-\frac{3}{2}}$ as $\Delta G_2$-modules.
\end{proof}

\begin{corollary}\label{nonzeroTJMsq}
The principal series representation $St_3\nu^{\frac{1}{2}}\times \nu^{-\frac{3}{2}}$ of $G_4$ has length $2$ with unique irreducible quotient $Z_{\nu,2}$ and  unique irreducible subrepresentation $St_4.$  Both the twisted Jacquet modules $(Z_{\nu,2})_{N,\psi}$ and $(St_4)_{N,\psi}$ are non-zero.
 \end{corollary}
\begin{proof}
It is well known that $St_3\nu^{\frac{1}{2}}\times \nu^{-\frac{3}{2}}$ has length 2 and  $St_4$ is its unique irreducible subrepresentation. Let $\eta$ be the unique irreducible quotient of $St_3\nu^{\frac{1}{2}}\times \nu^{-\frac{3}{2}}.$ We shall first show that $\eta=Z_{\nu,2}.$ Put $\mfr{m}=\{[\nu^{-\frac{1}{2}},\nu^{\frac{1}{2}},\nu^{\frac{3}{2}}],[\nu^{-\frac{3}{2}}]\}$ so that  $\mfr{m}^t=\{[\nu^{\frac{3}{2}}],[\nu^{\frac{1}{2}}],[\nu^{-\frac{3}{2}},\nu^{-\frac{1}{2}}]\}$ by MW algorithm.  We have $\eta=L(\mfr{m})=Z(\mfr{m}^t).$
Thus, $\eta$ is the unique irreducible submodule of the principal series representation $\nu^{\frac{3}{2}}\times \nu^{\frac{1}{2}}\times \nu_2^{-1}.$ Note that, $St_2\nu \times \nu_2^{-1} \hookrightarrow \nu^{\frac{3}{2}}\times \nu^{\frac{1}{2}}\times \nu_2^{-1}$ and therefore $\eta$ is the unique irreducible submodule of $St_2\nu \times \nu_2^{-1} $ as well. Hence, $\eta=Z_{\nu,2}.$ By Proposition \ref{Steinbergvscharacter} (2), we have $(Z_{\nu,2})_{N,\psi}=St_2\nu^{-1}\otimes \nu_2^{-1}$ as $\Delta G_2$-modules.
 
The principal series $St_2\nu^{-1} \times St_2\nu $ has length $2.$  Its irreducible subquotients are $St_4$ and the representation $\tau$ of Lemma \ref{preSteinbergvscharacternonzero}.  By Proposition \ref{DPanalogue}(2),  $(St_2\nu \times St_2\nu ^{-1})_{N,\psi}$ is glued from $St_2\nu\otimes St_2\nu^{-1}=St_2\otimes St_2$ and the irreducible principal series representation $ i_{\Delta B_2}^{\Delta GL_2(F)}(\nu^{-\frac{3}{2}} \otimes \nu^{\frac{3}{2}})\cong \nu^{\frac{3}{2}}\times \nu^{-\frac{3}{2}}.$ Of these, $\tau_{N,\psi}$ accounts for $\nu^{\frac{3}{2}}\times \nu^{-\frac{3}{2}}$  by Lemma \ref{preSteinbergvscharacternonzero}, proving $(St_4)_{N,\psi}$ must be non-zero. 
 \end{proof}

\begin{remark}
Corollary \ref{nonzeroTJMsq}  provides an example for the case mentioned in Remark \ref{rgreaterthann}, where both irreducible subquotients of a reducible principal series representation $St_{r}\chi \times \mu$ can have non-zero twisted Jacquet modules if $r>n.$
\end{remark}

\begin{remark}
In view of Theorem \ref{generic}, it is indeed true that all generic irreducible representations of $G_{2n}$ have a non-zero twisted Jacquet module. However, in this subsection and the next, we have chosen to make the calculations independent of Theorem \ref{generic}.
\end{remark}

\subsection{Examples of non-vanishing twisted Jacquet modules II}\label{example2}
\subsubsection{}
 Consider the principal series $\sigma= \nu\times 1 \times 1 \times \nu^{-1}$ of $G_4.$ To contrast with the principal series $\xi$ considered in Section \ref{example1}, the cuspidal support of $\xi$ is  distinct where as the cuspidal support of $\sigma$ has the character $1$ appearing with multiplicity $2.$
 The subquotients of $\sigma$ are well-known as it is a special case of \cite[Example 11.4]{Zelevinsky}. By \cite{Zelevinsky}, there are five distinct irreducible representations of $\sigma$ of which exactly one subquotient appears with multiplicity $2$ in $\sigma$ and all the remaining four subquotients appear with multiplicity one. The irreducible subquotients of $\sigma$ are 
 $St_2\nu^{\frac{1}{2}}\times \nu_2^{-\frac{1}{2}}, \nu_2^{\frac{1}{2}} \times St_2\nu^{-\frac{1}{2}},\one_3\times 1, St_3\times 1 \mbox{ and } L_{\nu^{-1/2},1}.$ Of these, $L_{\nu^{-1/2},1}$ appears with multiplicity $2.$
 
 \subsubsection{} It follows that  $(St_2\nu^{\frac{1}{2}}\times \nu_2^{-\frac{1}{2}})_{N,\psi}$, $(\nu_2^{\frac{1}{2}} \times St_2\nu^{-\frac{1}{2}})_{N,\psi}$ are non-zero (Proposition \ref{DPanalogue}) as well as  $(St_3\times 1)_{N,\psi}\neq 0$ (Proposition \ref{genericnplus1}).   By Corollary \ref{shalikamodel}, $(L_{\nu^{-1/2},1})_{N,\psi}$ is also non-zero. Finally, by Theorem \ref{productoftwocharcaters}, $(\one_3\times 1)_{N,\psi}=0.$ Thus, the only subquotient of $\sigma$ whose twisted Jacquet module vanishes is $\one_3\times 1.$ 
\section{Non vanishing of twisted Jacquet modules: a conjecture}\label{PrasadsConjecture}
 Based on the results of this work, D. Prasad has proposed  Conjecture \ref{DPConjecture} on non-vanishing of twisted Jacquet module $\pi_{N,\psi}$ for irreducible admissible representations $\pi$ of $G_{2n}$  where $(N,\psi)$ is as in \eqref{definitionofpsi} and everywhere else in this work. This conjecture is based especially on the assertions contained in this work. 

In Section \ref{summaryL}, we summarize some facts on Langlands parameters and $L$-functions necessary to state Prasad's conjecture. In Section \ref{verify}, we verify Conjecture \ref{DPConjecture} holds in few cases using the results we have proved in previous sections. Section \ref{DPConjecture-othercases} lists some cases where we have not verified Conjecture \ref{DPConjecture} in full generality.

\subsection{Short summary on Langlands parameters and $L$-functions}\label{summaryL}
We shall use the notion of $L$-functions in this section, especially the adjoint $L$-function of an irreducible admissible representation of $G_n.$ 
Before we proceed further, we shall summarize a few facts especially from \cite[Section 3]{Kud1} for ease of reference.  For notations and details, we refer the reader to \cite{Kud1}.

	\subsubsection{}  Suppose $\pi$ is an irreducible smooth representation of $G_n$ and $\rho={\rm rec}_F^{n}(\pi)$ is its Langlands parameter, then the adjoint $L$-function  of $\pi$ is $L(s, Ad\circ  \pi)=L(s,\rho\otimes \rho^{\vee}).$ 
		\subsubsection{} For a one dimensional representation, $\chi$ of $F^{\times},$ we shall identify $\chi$ with $rec_F^1(\chi)$ which is a character of $W_F^{ab}.$  The $L$-function of $\chi$ is $(1-\chi(\omega_F)q^{-s})^{-1}$ where $\omega_F$ denotes the uniformizer of $F.$ In particular, if $\chi=|\cdot|_F^{t},$ where $t\in \mathbb{R},$ the  only pole on the real line of the  $L$-function of $\chi$ is a simple pole at $s=-t.$
		\subsubsection{} The Steinberg representation $St_n$ has as its $L$-function $L(s, |\cdot|_F^{\frac{n-1}{2}}).$ For any twist $St_n\chi$ of the Steinberg representation, one has $L(s,St_n\chi)=L(s,\chi|\cdot|_F^{\frac{n-1}{2}}).$ In view of this, any real pole of $L(s, St_n\chi)$ is a simple pole and for $t\in \mathbb{R}$, $L(s,St_n|\cdot|_F^t)$ has a real pole  in the region $s>0$ if and only if $t<-\frac{n-1}{2}.$
		\subsubsection{} If $\rho_2$ denotes the Langlands parameter of $St_2,$ then  $\rho_2\otimes \rho_2=\one_1\oplus \rho_3$ where $\rho_3$ is the Langlands parameter of  the Steinberg representation $St_3$ of $G_3.$   By abuse of notation, we shall write this as $St_2\otimes St_2=\one_1\oplus St_3$ in our calculations. More generally, if $\rho_r$ denotes the Langlands parameter of the Steinberg representation $St_r$ of $G_r,$ one has $\rho_r\otimes \rho_r=\one_1\oplus \rho_3\oplus \cdots \oplus \rho_{2r-1}.$
	
		\subsubsection{} Assume that $\tau_j=L(\Delta_j)$ are essentially square integrable irreducible smooth representations of $G_{n_j}$ for $1\leq j \leq k$ such that $\Delta_i$ does not precede $\Delta_j$ whenever $i<j.$ Then,  $\tau_1\times \cdots \times \tau_k$ has a unique irreducible quotient $\pi$ which is also the Langlands quotient and $L(s,\pi)=\prod_{j=1,\dots,k} L(s,\tau_j)$ ( \cite[(3.1.3)]{Kud1}).
		
		\subsubsection{}	\label{LP}  Suppose $\pi_1$ and $\pi_2$ are irreducible smooth representations of $G_m$ and $G_n$ respectively with Langlands parameters $\rho_1$ and $\rho_2$ such that $\pi_1\times \pi_2$ is irreducible, then (see \cite[A.6 Corollary(ii)]{TadicU}) the Langlands parameter of $\pi_1\times \pi_2$ is $\rho_1\oplus \rho_2.$ We shall use this repeatedly in our calculations.
	
		\subsubsection{} For any $k\in \N,$ the trivial representation $\one_{k}$ of $G_k$ has Langlands parameter $\nu^{\frac{k-1}{2}}\oplus \cdots \oplus \nu^{-(\frac{k-1}{2})}.$ By \ref{LP}, we have
	\begin{equation*}\one_k\otimes \one_k= (\nu^{\frac{k-1}{2}}\oplus \cdots\oplus \nu^{-\frac{k-1}{2}})\otimes (\nu^{\frac{k-1}{2}}\oplus \cdots \oplus\nu^{-\frac{k-1}{2}}),
	\end{equation*}
and therefore 
	\begin{equation}\label{triv_k}
		 L(s, \one_k\otimes \one_k)= \nu^{k-1}\oplus 2 \nu^{k-2}\oplus \cdots \oplus (k-1) \nu\oplus k\one_1\oplus (k-1)\nu^{-1}\oplus \cdots \oplus \nu^{-(k-1)}.
	\end{equation}
	Consequently, 
	\begin{center}
	$L(s, Ad, \one_k)$ has  poles of order $1,2,\dots, k-1$ at $s=k-1,k-2, \dots, 1$ respectively.
	\end{center}
\subsection{D. Prasad's Conjecture}\label{Conjectureformulation}
We recall the following result from \cite[Proposition 5.2.2]{Kud1} which is a motivation for Conjecture \ref{DPConjecture}. We recall that $Ad: GL_{n}(\C)\to GL(M_{n}(\C))$ denotes the adjoint representation of $GL_n(\C).$ 

\begin{theorem}[Gross, D. Prasad, Rallis]\label{GPR}
	Let $\pi$ be an irreducible smooth representation of $G_n$ and let $\rho=\rec_F^n(\pi)$ denote its Langlands parameter.
	Then $\pi$ is generic if and only if $L(s, \pi, Ad )=L(s, Ad\circ \rho)=L(s, \rho\otimes \rho^{\vee})$ has no pole at $s=1.$
\end{theorem}
In other words, Theorem \ref{GPR} says that the twisted Jacquet module $\pi_{U_n,\theta}$ (see Section \ref{nonvanishing}) is zero if and only if $L(s, \rho\otimes \rho^{\vee})$ has a pole at $s=1.$ Conjecture \ref{DPConjecture} formulated below can be viewed as an analogue of Theorem \ref{GPR} for $\pi_{N,\psi}.$ 

\begin{remark}
	According to \cite[Conjecture 2.6]{DPGross1992}, a statement as in Theorem \ref{GPR} is expected to be true for more general  reductive groups than $G_n.$ 
\end{remark}

\begin{conjecture}[D. Prasad]\label{DPConjecture} 	Let $\pi$ be an irreducible smooth representation of $G_{2n}$ and let $\rho=\rec_F^{2n}(\pi)$ denote its Langlands parameter. Then, the twisted Jacquet module $\pi_{N,\psi} = 0$ if and only if the $L$-function $L(s, \pi, Ad)=L(s,\rho\otimes \rho^{\vee})$ has poles of order greater than or equal to $n,n-1,\dots, 1$ at $s=1,2,\dots, n$ respectively. 	
\end{conjecture}

\subsection{Generic Representations}\label{generic-section}
As a first case, we have the following proposition which verifies Prasad's conjecture for generic representations.

\begin{theorem}[D. Prasad]\label{generic}
	Let $\pi\in \Irr(G_{2n})$ be generic. 	Then, $\pi_{N,\psi}\neq 0.$
\end{theorem}
\begin{proof}
	Let $\pi$ be an irreducible generic representation of $G_{2n}.$ Let $P_{2n}$ denote the mirabolic subgroup of $G_{2n}$ consisting of matrices whose last row equals $(0, \dots,0, 1).$ By Bernstein-Zelevinsky theory of derivatives (see \cite[\S 3.5]{BZ2} or \cite[Proposition 1, pp. 171]{DPDuke93} ), $\pi_{|_{P_{2n}}}$ has a filtration of $P_{2n}$-modules 
	\begin{equation}
		\{0\}\subset V_1\subset \cdots \subset V_{2n}=\pi_{|_{P_{2n}}},
	\end{equation}
	where $V_1=\ind_{U(2n)}^{P_{2n}}(\theta_{2n}).$ Observe that $N\subset P_{2n}.$ By the exactness of the twisted Jacquet functor, $(V_1)_{N,\psi}\subset \pi_{N,\psi} $ and to show that $\pi_{N,\psi}\neq 0,$ it is sufficient to show that $(V_1)_{N,\psi}\neq 0.$
	
	By Gelfand-Kazhdan theorem (\cite[\S 5.18]{BZ1}, \cite[Theorem E, pp. 97]{GK}), given any irreducible cuspidal representation  $\sigma$ of $G_{2n},$ one has $\sigma_{|_{P_{2n}}}= \ind_{U(2n)}^{P_{2n}}(\theta_{2n}).$ It should be possible to deduce the non-vanishing of the twisted Jacquet module from this structure theorem about $\sigma$ restricted to $P_{2n}$ by direct calculations. Not finding such a proof, we resort to a more indirect proof. As $\sigma$ restricted to $P_{2n}$ is independent of $\sigma$, we are free to choose $\sigma$, which we will, and prove that one such $\sigma$ has a nonzero Shalika model,  which will in particular imply that $V_1$ has a nonzero twisted Jacquet module.
  Recall (\cite[Theorem 1.1]{JNQ} or \cite[\S 1]{Mat}) that, an irreducible cuspidal representation $\sigma$ of $G_{2n}$ has a Shalika model if and only if its Langlands parameter $\rho$ is symplectic, and such a cuspidal representation exists  by Proposition \ref{Weilgrprepn} proved below, which completes the proof.
\end{proof}

\subsubsection{} To prove Proposition \ref{Weilgrprepn}, we recall some notations for which we shall refer to mainly \cite{BH-book} and \cite{Knapp}. 
	For a non-archimedean local field $F$,  let $O_F, O_F^{\times}$ and $\varpi_F $ denote respectively the ring of integers, units in $O_F$ and a uniformizer. Also, let $O_F^{\times}(1)$ denote the subgroup $1+\varpi_F O_F$ consisting of principal units. Then, $O_{F}^{\times}/O_F^{\times}(1)=\F_{q}^{\times},$ if the residue field $k_F$ of $F$ is isomorphic to $\F_q.$ If $K/F$ is an unramified extension of degree $n$ then $\Gal(K/F)$ is isomorphic to $\Z/n\Z.$
	
In the following proposition, we shall be constructing a regular character  $\chi$  of $\F^{\times}_{q^{2n}}$ which will be used in the proof of Proposition \ref{Weilgrprepn}.
\begin{proposition}\label{cuspidal-selfdual-finite}
There exists an irreducible self-dual cuspidal representation of $GL_{2n}(\F_q).$ 
	\end{proposition}
	\begin{proof}Throughout this proof, let  $G_n=GL_{n}(\F_q).$ Recall (see for instance \cite[Theorem 4.2]{DP-Finite-self-dual}) that irreducible cuspidal representations of $G_{2n}$ correspond to regular characters $\chi:\F_{q^{2n}}^{\times}\to \C^{\times},$ i.e., characters $\chi$  such that $\chi^{\sigma}\neq \chi$ for any nontrivial $\sigma\in \Gal(\F_{q^{2n}}/\F_q).$ Also, the self-duality is equivalent to the existence of $\sigma_{0}\in \Gal(\F_{q^{2n}}/\F_q)$ such that  $\chi^{\sigma_0}=\chi^{-1}.$
	
	Take any $\chi$ on $\F^{\times}_{q^{2n}}/\F^{\times}_{q^n}$ whose order is the order of the cyclic group $\F^{\times}_{q^{2n}}/\F^{\times}_{q^n}.$ We claim that such a $\chi$ is regular, producing a self-dual irreducible cuspidal representation of $G_{2n}.$ If $\chi^{\sigma}=\chi$ for a nontrivial  $\sigma\in \Gal(\F_{q^{2n}}/\F_q),$ put $H=\langle \sigma \rangle$ and let $L=\F^H_{q^{2n}},$ the fixed subfield of $H.$   Let $\Nm$ denote the norm map from $\F^{\times}_{q^{2n}}\to L^{\times}$, which is surjective.  The condition that $\chi(\sigma(x))=\chi(x)$ for all $x\in \F^{\times}_{q^{2n}}$ is equivalent to  saying that $\chi$ factors through $L^{\times}$ via $\Ker(\Nm).$   Since $[\F_{q^{2n}}:L]=|H|\geq 2,$  and $\chi$ factors through $L^{\times},$ $\chi$ has order less than or equal to $|L^{\times}|\leq q^n-1.$ But the order of $\chi$ is $q^n+1,$ a contradiction. This establishes the regularity of $\chi.$ Choose $\sigma_0$ to be the unique element of order $2$ in $\Gal(\F_{q^{2n}}/\F_q).$ Since $\chi $ is trivial
	on $\F^{\times}_{q^n}$,  $\chi(x \sigma_0(x))= 1$ for all $x\in \F^{\times}_{q^{2n}},$ that is,  $ \chi^{-1} = \chi^{\sigma_0}.$
	\end{proof}

\begin{proposition}\label{Weilgrprepn}
Let $F$ be a non-archimedean field. Then, the Weil group $W_F$ has an irreducible symplectic representation of dimension $2n.$
\end{proposition}
\begin{proof}
We will consider the quotient $W_{E/F}$ (see \cite[(2.6) pp. 252]{Knapp} of $W_F$ which sits in the following  exact sequence
	\begin{equation*}
		1\to E^{\times} \to W_{E/F} \to {\rm Gal}(E/F) \to 1,
	\end{equation*}
	where $E/F$ is an unramified extension of degree $2n.$ Let $\sigma_0$ be the element of order $2$ in ${\rm Gal}(E/F)=\Z/2n\Z$ with fixed subfield $K.$ Then, $E/K$ is a quadratic unramified extension. As $E^{\times}=O_E^{\times}\varpi_F^{\Z},$ and $O_E^{\times}/O_E^{\times}(1)\simeq \F^{\times}_{q^{2n}},$ 
	we define a character $\chi'$ of $E^{\times}$ where $\chi'$ acts as the character $\chi$ of Proposition \ref{cuspidal-selfdual-finite} on $O_E^{\times}$ and as the character $\theta$ with $\theta(\varpi_F)=-1.$ It is easy to see that $(\chi')^{\sigma_0}=(\chi')^{-1}$ and that the character $\chi'$ is regular since $\chi$ is regular.
	
	Put $\pi=\Ind_{E^{\times}}^{W_{E/F}}(\chi')$ and $\tau=\Ind_{E^{\times}}^{W_{E/K}}(\chi').$
	By regularity of $\chi'$, both $\pi$ and $\tau$ are irreducible representations of dimensions $2n$ and $2$ respectively. As $\chi'^{\sigma_0}=\chi'^{-1},$ both $\pi$ and $\tau$ are self-dual as well. Moreover,  it is easy to see that (see for instance \cite[\S 29.2, Proposition, pp.188 ]{BH-book})
	$\det(\tau)= (\theta \otimes \chi')_{|_{K^{\times}}}=1,$
	proving  $\tau$ is symplectic. Since an irreducible representation induced from an irreducible symplectic representation is itself symplectic and  by transitivity of induction $\pi=\Ind_{W_{E/K}}^{W_{E/F}}(\tau),$ $\pi$ is symplectic as well.
\end{proof}

\begin{remark}
In view of 	Theorem \ref{generic}, Conjecture \ref{DPConjecture} holds for all generic irreducible representations of $G_{2n}$ which further agrees with Theorem \ref{GPR}.
\end{remark}

 In the following Sections \ref{verify} and \ref{DPConjecture-othercases}, by abuse of notation, we shall identify a smooth irreducible representation $\pi$ of $G_n$ with its Langlands parameter  $rec_F^n(\pi)$ for the purposes of calculation of the adjoint $L$-function. We hope this eases notation and avoids some tedious book keeping.

\subsection{Special cases of the conjecture}\label{verify}

\subsubsection{}\label{trivnplus1vsirred}
Assume that $\pi$ is an irreducible smooth representation of $G_{2n}$ of the form $\pi=\one_{n+1} \times \sigma$ where $\one_{n+1}$ denotes the trivial representation of $G_{n+1}$ and $\sigma$ is any irreducible representation of $G_{n-1}.$ By \ref{LP}, $\pi \otimes \pi^{\vee}$ contains $	\one_{n+1} \otimes \one_{n+1}$ and by \eqref{triv_k},  $L(s,\one_{n+1}\otimes \one_{n+1})$ has poles at $s=1, 2,\dots, n $ of order $n,n-1,\dots , 1$ respectively.  Thus, by the Conjecture \ref{DPConjecture}, $\pi_{N,\psi}=0.$ Also, by Proposition \ref{charactervssmooth} (1), $\pi_{N,\psi}=0.$

	\subsubsection{} \label{trivnminus1vsgeneric}
	Assume that $\pi$ is a smooth representation of $G_{2n}$ of the form $\pi=\one_{n-1} \times \sigma$ where $\one_{n-1}$ denotes the trivial representation of $G_{n-1}$ and $\sigma$ is any irreducible generic representation of $G_{n+1}$ which we also assume to be tempered.   In particular, $\pi$ is irreducible. Then, by \ref{LP}
	\begin{eqnarray*}
		\pi\otimes \pi^{\vee} = \one_{n-1} \otimes \one_{n-1} \oplus \sigma\otimes \sigma^{\vee} \oplus  \one_{n-1}\otimes \sigma \oplus \one_{n-1}\otimes \sigma^{\vee}.
	\end{eqnarray*}
	Since, $\sigma$ is assumed to be tempered neither $\one\otimes\sigma$ and $\one\otimes \sigma^{\vee}$ nor $\sigma \otimes \sigma^{\vee}$  have a pole in the region $s>0.$ Therefore, all the poles of $L(s,\pi\otimes \pi^{\vee})$ are contributed by $L(s,\one_{n-1}\otimes \one_{n-1}).$ By \eqref{triv_k}, $L(s,\one_{n-1}\otimes \one_{n-1})$ has poles at $s=1,2,\dots, n-1$ of order $n-1,n-2, \dots, 1$ respectively. Therefore, $\pi_{N,\psi}\neq 0$ by Conjecture \ref{DPConjecture}, and this is the case by Proposition \ref{genericnplus1} (2).
	\subsubsection{} \label{tau}
	The representation $\tau$ of $G_4$ constructed in Section \ref{example1} with Langlands parameter $St_2\nu \oplus St_2\nu^{-1}.$ In this case, we have 
	\begin{flalign*}
		\tau\otimes \tau^{\vee}=& (St_2\nu\oplus \St_2\nu^{-1}) \otimes (St_2\nu\oplus St_2\nu^{-1})\\
		=& (St_2\otimes St_2) \otimes (\nu^{2}\oplus \one_1\oplus \one_1\oplus \nu^{-2})\\
		=& (\one_1\oplus St_3) \otimes (\nu^{2}\oplus \one_1\oplus \one_1\oplus \nu^{-2})\\
		=&\nu^{2}\oplus \one_1\oplus  \one_1\oplus \nu^{-2}\oplus St_3\nu^{2}\oplus St_3\oplus St_3\oplus St_3\nu^{-2}
	\end{flalign*}
	In this case, for $L(s, \tau\otimes \tau^{\vee})$, the poles in the region $s>0$ are at $s=2$ of order 1 contributed by $L(s,\nu^{-2})$ and at $s=1$ of order 1 (and not 2) contributed by $L(s,St_3\nu^{-2}).$ Thus, by Conjecture \ref{DPConjecture}, the representation $\tau$ of $G_4$ has non-zero twisted Jacquet module which is also the case by Lemma \ref{preSteinbergvscharacternonzero}. In fact, Conjecture \ref{DPConjecture} is inspired by this particular example.
	
	\subsubsection{} The representations $Q_{\nu,2}$ and $Q_{\nu,2}^{\vee}$ which are Langlands quotients of $St_2\nu\times \nu^{-\frac{1}{2}} \times \nu^{-\frac{3}{2}}$ and $\nu^{\frac{3}{2}}\times \nu^{\frac{1}{2}}\times St_2\nu^{-1}$ respectively have twisted Jacquet module equals zero by  Section \ref{example1}. The adjoint representation of these representations decomposes into 
	$$ St_3\oplus 3\cdot \one _1 \oplus St_2\otimes (\nu^{\frac{5}{2}} \oplus \nu^{\frac{3}{2}}\oplus \nu^{\frac{-3}{2}}\oplus \nu^{\frac{-5}{2}})\oplus \nu \oplus \nu^{-1}.$$
	The poles in the region $s>0$ of the adjoint $L$-function are therefore at $s=1$ and $s=2$ of order $2$ and $1$ respectively.  By Conjecture \ref{DPConjecture}, the twisted Jacquet modules under consideration are zero as well.

\subsubsection{}
Consider the representation $Q_{1,r}$ appearing in Proposition \ref{Steinbergvscharacter}. We write $Q_{1,r}=St_r \oplus \nu_{2n-r}^n.$ Then, one has 
$Q_{1,r}\otimes Q_{1,r}^{\vee}=St_r\otimes St_r \oplus St_r\otimes\nu_{2n-r}^n \oplus St_r\otimes \nu_n^{-n} \oplus \one_{2n-r}\otimes \one_{2n-r}.$
Of these, $L(s, St_r\otimes St_r)$ does not contribute any poles in $s>0.$ Also, it is not difficult to see that $L(s, St_r\otimes \nu_{2n-r}^n)$ does not contribute any pole in the region $s>0$ as well.  On the other hand, by \eqref{triv_k}, $L(s, \one_{2n-r}\otimes \one_{2n-r})$ has  poles of order $1,2,\dots, 2n-r-1$ at $s=2n-r-1, 2n-r-2, \dots,1$ respectively. Moreover, 
$$L(s, St_r\otimes \nu_{2n-r}^{-n})=L(s, \nu^{r-2n})L(s,\nu^{r-2n-1})\cdots L(s,\nu^{-1}).$$
If $r\leq n,$ $2n-r\geq n$ and $2n-r-1\geq n-1.$ Thus, counting the order of poles of both $L(s, St_r\otimes \nu_{2n-r}^{-n})$ and $L(s, \one_{2n-r}\otimes \one_{2n-r})$, we have  $L(s, Q_{1,r}\otimes Q_{1,r}^{\vee})$ has a pole of order greater than or equal to $n,n-1,\dots, 1$ at $s=1,2,\dots, n$ respectively. By Conjecture \ref{DPConjecture}, $(Q_{1,r})_{N,\psi}=0$ which agrees with Proposition \ref{Steinbergvscharacter}.

\subsubsection{}
Consider the representation $Z_{1,r}$ appearing in Proposition \ref{Steinbergvscharacter}. By Proposition \ref{Steinbergvscharacter}, we have $(Z_{1,r})_{N,\psi}=0$ if $r<n$ and $(Z_{1,r})_{N,\psi}$ is non-zero if $r=n.$ We shall verify  that Conjecture \ref{DPConjecture} holds for $Z_{1,r}$ as well.  Note that $Z_{1,r}=Z(\mathfrak{m})$ where 
$\mathfrak{m}=\{[\nu^{\frac{r-1}{2}}], [\nu^{\frac{r-3}{2}}],\dots, [\nu^{-(\frac{r-1}{2})}], \Delta\}$ where $\Delta=[\nu^{-\frac{2n-r-1}{2}}\cdot \nu^{-n}, \dots, \nu^{\frac{2n-r-1}{2}}\cdot\nu^{-n}].$ Note that, $\Delta$ ends at $\nu^{-\frac{r+1}{2}}.$  Applying the MW algorithm to $\mathfrak{m},$ we get
$\mathfrak{m}^t=\{\Delta', [\nu^{-\frac{r+3}{2}}], \dots, [\nu^{-\frac{-(2n-r-1)}{2}}\cdot \nu^{-n}]\}$ where $\Delta'=[\nu^{-\frac{r+1}{2}}, \dots, \nu^{\frac{r-1}{2}}].$ Consequently, $Z_{1,r}=L(\mfr{m}^t)$ yielding that $Z_{1,r}$ is a Langlands quotient of $St_{r+1}\nu^{\frac{1}{2}}\times \nu^{-\frac{r+3}{2}}\times \cdots\times \nu^{-\frac{2n-r-1}{2}}\cdot \nu^{-n}.$  Noting that the unique irreducible quotient of $\nu^{-\frac{r+3}{2}}\times \cdots\times \nu^{-\frac{2n-r-1}{2}}\cdot \nu^{-n}$ is the character $\nu_{2n-r-1}^{-n-\frac{1}{2}},$ we shall write for the purpose of computing the adjoint $L$-function that $Z_{1,r}= St_{r+1} \nu^{\frac{1}{2}}\oplus \nu_{2n-r-1}^{-n-\frac{1}{2}}.$ Note that 
$$Z_{1,r}\otimes Z_{1,r}^{\vee}= St_{r+1} \otimes St_{r+1}\oplus St_{r+1}\otimes  \nu_{2n-r-1}^{n+1}\oplus St_{r+1}\otimes \nu_{2n-r-1}^{-n-1}\oplus \one_{2n-r-1}\otimes \one_{2n-r-1}.$$
Arguing as in previous example, the first two factors in the direct sum do not contribute to poles in the region $s>0.$ We need only consider the poles contributed by $L(s,St_{r+1}\otimes \nu_{2n-r-1}^{-n-1})$ and $L(s,\one_{2n-r-1}\otimes \one_{2n-r-1}).$ By \eqref{triv_k}, $L(s,\one_{2n-r-1}\otimes \one_{2n-r-1})$ has poles at $s=1,2,\dots, 2n-r-2$ of order $2n-r-2, 2n-r-3, \dots, 1$ respectively whereas,
$$L(s,St_{r+1}\otimes \nu_{2n-r-1}^{-n-1})=L(s,\nu^{r-2n})L(s,\nu^{r-2n-1}) \cdots L(s, \nu^{n-2}).$$
We split the further verification into two cases naturally.\\
\noindent \textit{Case }1: Assume that $r=n.$ Here, $L(s,St_{n+1}\otimes \nu_{n-1}^{-n-1})$ has a pole of order $1$ at $s=1$ and $(L(s,\one_{n-1}\otimes \one_{n-1})$ has a pole of order $n-2$ at $s=1$ yielding that the order of the pole at $s=1$ is $n-1$ and not $n.$ Hence, $(Z_{1,n})_{N,\psi}\neq 0.$ \\

\noindent \textit{Case }2: Assume that $r<n.$ In this case, $r-2n<-n,$ $2n-r\geq n+1$ and hence $2n-r-2\geq n-1.$ The factor $L(s, St_{r+1}\otimes \nu_{2n-r-1}^{-n-1}))$ has a simple pole at $s=1,2,\dots ,n$ as $\{\nu^{-n},\nu^{-(n-1)},\dots ,1\}$ is contained in $\{\nu^{r-2n},\dots, \nu^{n-2}\}.$ As $2n-r-2\geq n-1,$ from $L(s, \one_{2n-r-1}\otimes \one_{2n-r-1})$ we get a pole of order $\geq 1$ at $s=n-1,$ a pole of order $\geq 2$ at $s=n-2$, $\dots$, a pole of order $\geq n-1$ at $s=1.$ Combining the contribution from both the factors, $L(s,Z_{1,r}\otimes Z_{1,r}^{\vee})$ has a pole of order $\geq 1$ at $s=n,$ pole of order $\geq 2$ at $s=n-1,$ $\dots, $ a pole of order $\geq n$ at $s=1.$ Thus, $(Z_{1,r})_{N,\psi}=0.$ 

In both the cases, Conjecture \ref{DPConjecture} agrees with Proposition \ref{Steinbergvscharacter}.
\subsection{Some other cases of the conjecture }\label{DPConjecture-othercases}
Here are few important special cases of the conjecture we have not proved in this paper. 
\subsubsection{}\label{prodtriv} 
The representation $\pi=\one_{n_1}\times \cdots \times \one_{n_m}$ of $G_{2n}$ where $n_1+\cdots+n_m=2n$ with $n_j\leq n$ for all $j$ has a nontrivial twisted Jacquet module as their adjoint $L$-function do not have a pole at $s=n.$ 
\subsubsection{} 
Suppose $n,r,m$ are positive integers such that $2n=rm.$ Let  $\tau$ be any discrete series representation (see \cite[\S 1.2, \S 2.4]{Kud1}, \cite{TadicU}) of $G_m.$  The induced representation $\tau\nu^{\frac{r-1}{2}}\times\cdots \times \tau\nu^{-\frac{r-1}{2}}$ has a unique irreducible quotient denoted by $u(\tau,r),$ known as a Speh representation. Such a representation is unitary (\cite[Theorem D, pp. 338]{TadicU}, \cite[\S 2.4]{Kud1}) and  by Conjecture \ref{DPConjecture}, $u(\tau,r)$ has a non-zero twisted Jacquet module unless it is a one dimensional representation.
\begin{remark}
We note the following generality which is applicable to verify some cases of Conjecture \ref{DPConjecture}, especially to representations of the type in  \ref{prodtriv}. Suppose $k>1$ and $\pi\in \Irr(G_{2n})$ is such that  $\pi=\pi_1\times \cdots \times \pi_k$ where for all $j,$ $\pi_j=Z(\Delta_j)\in \Irr(G_{n_j})$ for some segments $\Delta_j$  and $n_1+\cdots+n_k=2n.$ Assume that there exists $n_{j_1}, \dots, n_{j_t}$ among the $n_i$'s such that $n_{j_1}+\cdots +n_{j_t}=n.$ Then we can rearrange the $n_j$'s  and assume without loss of generality that $n_1+\cdots +n_t=n.$ We may then apply Proposition \ref{nonvanishing} to $\pi=(\pi_{n_1} \times \cdots \times \pi_{n_t}) \times (\pi_{n_{t+1}}\times \cdots \times \pi_{n_m})$   to get  $\pi_{N,\psi}\neq 0.$
\end{remark}

\section*{Acknowledgements}
This work forms a significant component of the PhD thesis of C. Harshitha at the Indian Institute of Science Education and Research Tirupati and she acknowledges the financial support provided by the National Board of Higher Mathematics (NBHM), Govt of India and  Prime Minister's Research Fellowship (PMRF ID 0900452) program, Govt of India during this work. The authors thank Dipendra Prasad for many interesting conversations, insightful comments and constant encouragement during the course of the preparation of this article and especially for the proofs of Theorem \ref{generic}, and Propositions  \ref{cuspidal-selfdual-finite} \& \ref{Weilgrprepn}.


\begin{thebibliography}{99}
	
\bibitem{Badulescu-Renard}
	Badulescu, A. I., Renard, D. {\it Zelevinski involution and Moeglin-Waldspurger algorithm for $GL_n(D)$}, Functional Analysis IX, 9-15, Various Publ. Ser.(Aarhus), 48.
 	\bibitem{BLM}
 	\author	{I. Badulescu, E. Lapid and A. M\'inguez},
 	{\it Une condition suffisante pour
 		l'irr\'eductibilit\'e d'une induite parabolique}, Ann. Inst. Fourier, Grenoble, Tome 63, no 6 (2013), p. 2239-2266.
 	\bibitem{KH1}
 	K. Balasubramanian, H. Khurana, {\it A certain twisted Jacquet module of \rm{GL}(4) over a finite field}, J. Pure Appl. Algebra {\bf 226} (2022), Paper no. 106392, 16 pp.
 	
 	\bibitem{KH2}
 	K. Balasubramanian, H. Khurana,
 {\it On a twisted Jacquet module of \rm{GL}(6) over a finite field}, New York Journal of Mathematics {\bf 29}, (2023), 874--910.
 	
 	\bibitem{KH3}
 	K. Balasubramanian, A. Dangodara, H. Khurana, {\it On a twisted Jacquet module of \rm{GL}(2n) over a finite field}, J. Number Theory, {\bf 271},(2025), 458--474.
 	
 	\bibitem{BZ1}
 	I. N. Bernstein, A. V. Zelevinsky, {\it Representations of \rm{GL}(n,F) where F is a non-archimedean local field}, Russian Math. Surveys, {\bf 31:3} (1976), 1--68.
 	
 	\bibitem{BZ2}
 	I. N. Bernstein,  A. V. Zelevinsky, {\it Induced representations of reductive p-adic groups I}, Ann. Scient. E.N.S., 4e serie, t., {\bf 10} (1977), 441--472.
 	
 	\bibitem{BD}
 	P. Blanc, P. Delorme, {\it Vecteurs distributions H-invariants de repr\'esentations induites, pour un
 	espace sym\'etrique r\'eductif $p$-adique $G/H$}, Ann. Inst. Fourier, Grenoble, Tome 58, no 1 (2008), p. 213-261.
 	
 \bibitem{BH-book}
 C. Bushnell, G. Henniart,
 		{\it The Local Langlands Conjecture for $GL(2)$}, Springer-Verlag, Berlin-Heidelberg, (2006).
 	
  	\bibitem{DPSarah}
 	S. Dijols, D. Prasad, {\it Symplectic models for unitary groups}, Trans. Amer. Math. Soc. {\bf 372} (2019), no. 3, 1833--1866.
 	
 		
 	\bibitem{GS}
 	W. T. Gan, S. Takeda, {\it On Shalika Periods and a theorem of Jacquet-Martin}, Amer. J. Math. {\bf 132} (2010), 475--528.
 	
 	\bibitem{GK}
 	I.M. Gelfand, D.A. Kazhdan, {\it Representations of the group $GL(n,K)$ where $K$ is a local field}, in {\it Lie Groups and their representations}, Wiley, London, 1975, pp. 95-118.
 	
 	\bibitem{DPGross1992}
 	B. Gross, D. Prasad, {\it On the decomposition of a representation of $SO_n$ when restricted to $SO_{n-1}$}, Can. J. Math. Vol. {\bf 44}(5), 1992 pp.974-1002.
 	
 	\bibitem{OfirZahi}
 	O. Gorodetsky, Z. Hazan, {\it  On certain degenerate Whittaker models for cuspidal representations of \rm{GL}(k·n, $\F_q$)}, Math. Z. {\bf 291} (2019), 609--633.
 	
 	
 	\bibitem{HV1}
 	C. Harshitha, C. G. Venketasubramanian, {\it On a Bruhat decomposition related to the Shalika subgroup of $GL(2n)$}, \url{https://doi.org/10.48550/arXiv.2512.24368}.
 	
 	\bibitem{JRCompositio}
 	H. Jacquet, S. Rallis, {\it Uniqueness of linear periods}, Compositio Math. 102, no 1 (1996), 65-123.
 	
 	\bibitem{JNQ}
 	D. Jiang, C. Nien, Y. Qin,
 		{\it Symplectic supercuspidal representations of ${\rm GL}(2n)$ over $p$-adic fields}, Pacific J. Math., Vol. 245, No.2, (2010).
 	
 		\bibitem{Knapp}
 	A. Knapp, {\it An introduction to the Langlands program}, Proc. Symp. in Pure Math., Vol. {\bf 61} (1997), pp. 245-302.
 	
 		\bibitem{Kud1}
 	S. S. Kudla, {it The local Langlands correspondence: the non-Archimedean case}, Proceedings of Symposia in Pure Mathematics, vol 55 part 2,(1994), 365-391.
 	
 	\bibitem{Mat}
 	N. Matringe, {\it Shalika periods and parabolic induction for $GL(n)$ over a non-archimdean local field}, Bull. Lond. Math. Soc. {\bf 49}, (2017), no.3, 417-427.
 	\bibitem{Moeglin-Waldspurger}
 	C. Moeglin,  J.-L.Waldspurger, {\it Sur I'involution de Zelevinski}, J. Reine Angew. Math. {\bf 372}(1986), 136-177. 	
 	
 	\bibitem{NS}
 	S. Nadimpalli, M. Seth, {\it Twisted Jacquet modules: a conjecture of D. Prasad}, Proc. Amer. Math. Soc. {\bf 153} (2025), no. 3, 1349-1361 (to appear).
 
 	\bibitem{Naor} 	
 	I. Naor, {\it Uneven Shalika models for general linear groups},  \url{https://arxiv.org/abs/2205.15313}
 	
 \bibitem{Omer}
 O. Offen, {\it Period integrals of automorphic forms and local distinction},
 	Relative aspects in representation theory, Langlands
 	functoriality and automorphic forms, Lecture Notes in Math., Vol. {\bf 2221} (2018), Springer, Cham, 159--195.
 			
 \bibitem{SV}
 S. K. Pandey, C. G. Venketasubramanian, {\it Structure of Twisted Jacquet modules of principal series representations of $Sp_4(F)$}, J. Pure Appl. Algebra {\bf 229} (2025), Paper No. 108060, 37 pp.
 \bibitem{PP}
 A. Parashar, S.P.  Patel, {\it On the degenerate Whittaker space for ${\rm GL_4}(\mathfrak{o}_2)$ }, J. Pure Appl. Algebra {\bf 229}(2025), no.5, paper No, 107921, 29 pp.
 
  \bibitem{DPDuke93}
  
  D. Prasad, {\it On the decomposition of a representation of $GL(3)$ restricted to $GL(2)$ over a $p$-adic field}, Duke. Math. J., Vol 69, no. 1, (1993), 167-177.
 	
 \bibitem{DP-Finite-self-dual}
 	D. Prasad, {\it On the Self-Dual Representations of Finite Groups of
 		Lie Type}, J.  Algebra  (210), pp. 298-310 (1998).
 	\bibitem{DPDegIMRN}
 	D. Prasad, {\it The space of degenerate Whittaker models for general
 		linear groups over a finite field}, Int. Math. Res. Not. {\bf 11} (2000), 579--595.
 		
 		\bibitem{DPRaghuramDuke}
 		D. Prasad, A. Raghuram, {\it Kirillov theory for $GL_2(\mathscr{D})$} where $\mathscr{D}$ is a division algebra over a non-archimedean local field, Duke Math. J., Vol. {\bf 104} (2000), no .1, 19--44.
 	
 	\bibitem{DPDegTIFR}
 	D. Prasad, {\it  The space of degenerate Whittaker models for \rm{GL(4)} over p-adic fields}, in Cohomology of arithmetic groups, L-functions and
 	automorphic forms, (Mumbai, 1998/1999), Tata. Inst. Fund. Res. Stud. Math., {\bf 15}, Tata Inst. Fund. Res., Bombay, (2001), 103--115.
 	
 	
 	
 	\bibitem{DPRamin}
 	D. Prasad,  R. Takloo-Bighash, {\it  Bessel models for \rm{GSp(4)}}, J. Reine Angew. Math. {\bf 655} (2011), 189--243.
 
 	
 	\bibitem{renard}
 	D. Renard, {\it Representations des groupes reductifs p-adiques}, Cours Spécialisés (17),
 	Société Mathématique de France, 2010.
 	
 	\bibitem{Brooks-Schmidt} 
 	B. Roberts, R. Schmidt, {\it Some results on Bessel functionals for \rm{GSp}(4)}, Doc. Math. {\bf 21} (2016), 467--553.
 	
 	\bibitem{TadicU}
 	M. Tadic, {\it Classification of unitary representations in irreducible representations of general linear group(non-archimedean case)}, Annales scientifiques de l’É.N.S. 4 e série, tome 19,  (1986), p. 335-382.
 	
 	\bibitem{CGV}
 	C. G. Venketasubramanian, {\it On representations of $GL(n)$ distinguished by $GL(n-1)$ over a $p$-adic field}, Israel J. Math., {\bf 194} (2013), 1-44.
 	
 	\bibitem{Zelevinsky}
 	A. Zelevinsky, {\it Induced representations of reductive p-adic groups. II. On
 		irreducible representations of GL(n)}, Annales scientifiques de l’É.N.S. 4 e série, tome 13, no 2 (1980), p. 165-210.

\end{thebibliography}
\end{document}